\documentclass[letterpaper,11pt,reqno]{amsart}

\usepackage[utf8]{inputenc}
\usepackage{mathcommand}
\usepackage{xpatch}
\usepackage{xcolor}
\usepackage{xstring}
\definecolor{verydarkblue}{rgb}{0,0,0.4}
\usepackage[
    breaklinks,
    colorlinks,
    citecolor=verydarkblue,
    linkcolor=verydarkblue,
    urlcolor=verydarkblue,
    pagebackref=true,
    bookmarksdepth=3,
    hyperindex,
]{hyperref}
\usepackage[T1]{fontenc}
\usepackage{fancyhdr}
\usepackage[
    hscale=0.7,
    vscale=0.75,
    headheight=13pt,
    marginpar=2cm,
    centering,
]{geometry}
\usepackage{amsmath}
\usepackage{amsthm}
\usepackage{thmtools}
\usepackage{amssymb}
\usepackage{mathtools}
\mathtoolsset{centercolon}
\usepackage{stmaryrd}
\SetSymbolFont{stmry}{bold}{U}{stmry}{m}{n}
\usepackage{mathdots}
\usepackage{mleftright}
\mleftright
\usepackage{physics}
\usepackage{embrac}
\usepackage{graphicx}
\usepackage{framed}
\usepackage{multicol}
\usepackage{paracol}
\usepackage{array}
\usepackage[inline]{enumitem}
\usepackage[alphabetic,abbrev]{amsrefs}
\usepackage{cleveref}
\usepackage{tikz}
\usetikzlibrary{cd}
\usepackage[final,nopatch=footnote]{microtype}

\makeatletter

\let\lmtt@use@light@as@normal\@empty

\patchcmd{\tocsection}{\indentlabel}{\hspace{-1pc}}{}{}

\patchcmd{\section}{\normalfont\scshape}{\large\bfseries\boldmath}{}{}
\patchcmd{\subsection}{\normalfont}{\bfseries\boldmath}{}{}
\patchcmd{\subsubsection}{\normalfont}{\bfseries\boldmath}{}{}
\def\@secnumfont{\@empty}

\AddToHook{cmd/appendix/before}{%
    \crefalias{section}{appendix}%
    \crefalias{subsection}{appendix}%
}

\patchcmd{\@settitle}{\uppercasenonmath\@title}{\LARGE}{}{}
\patchcmd{\@setauthors}{\MakeUppercase}{\large}{}{}
\patchcmd{\HyOrg@maketitle}{\uppercasenonmath\shorttitle}{}{}{}
\patchcmd{\HyOrg@maketitle}{\thispagestyle{firstpage}}{\thispagestyle{empty}}{}{}
\patchcmd{\@maketitle}{\@adminfootnotes}{{\def\@makefntext{\noindent\@makefnmark}\@adminfootnotes}}{}{}
\patchcmd{\@adminfootnotes}{\@setthanks}{\setlength\parindent{0pt}\@setthanks}{}{}

\patchcmd{\@setaddresses}{\def\email}{\def\email##1##2{\href{mailto:##2}{\texttt{##2}}}\def\pepe}{}{}
\patchcmd{\@setaddresses}{\\}{\pepe}{}{}
\patchcmd{\@setaddresses}{\par\addvspace\bigskipamount\indent}{}{}{}
\patchcmd{\@setaddresses}{(\ignorespaces}{\switchcolumn\vspace{1em}\begingroup\bfseries\ignorespaces}{}{}
\patchcmd{\@setaddresses}{\unskip) }{\unskip\endgroup\\}{}{}
\patchcmd{\@setaddresses}{\scshape}{}{}{}
\patchcmd{\@setaddresses}{\addresses}{{\vspace{2em}\begin{paracol}{2}\setlength{\parindent}{0pt}\switchcolumn\addresses\end{paracol}}}{}{}

\DefineAdditiveKey{bib}{secondauthor}{\name}
\DefineSimpleKey{bib}{archiveprefix}
\DefineSimpleKey{bib}{eprinttype}
\DefineSimpleKey{bib}{eprintclass}
\DefineSimpleKey{bib}{primaryclass}
\NewDocumentCommand{\my@ifeq}{mmm}{%
    \edef\my@a{#1}%
    \edef\my@b{#2}%
    \ifx\my@a\my@b#3\fi%
}
\apptocmd{\bib@resolve@xrefs}{\postprocess@entry}{}{}
\NewDocumentCommand{\postprocess@entry}{}{%
    \IfEmptyBibField{eprint}{}{%
        \def\eprint@url{https://arxiv.org/abs/\BibField{eprint}}%
        \def\eprint@txt{%
            \PackageWarning{roirefs}{Unrecognized eprint type `\my@etype`  given for entry \current@citekey.}%
            {\color{red}\my@etype:\BibField{eprint}}%
        }%
        \def\parse@arxiv{%
            \def\eprint@txt{arXiv:\BibField{eprint}}%
        }%
        \def\parse@hal{%
            \def\eprint@url{https://hal.archives-ouvertes.fr/\BibField{eprint}}%
            \def\eprint@txt{\BibField{eprint}}%
        }%
        \edef\my@etype{???}%
        \IfEmptyBibField{eprinttype}{%
            \IfEmptyBibField{archiveprefix}{%
                \PackageWarning{roirefs}{No eprint type found for entry \current@citekey, assuming arXiv.}%
                \edef\my@etype{arXiv}%
            }{%
                \edef\my@etype{\BibField{archiveprefix}}%
            }%
        }{%
            \edef\my@etype{\BibField{eprinttype}}%
        }%
        \my@ifeq{\my@etype}{arxiv}{\parse@arxiv}%
        \my@ifeq{\my@etype}{arXiv}{\parse@arxiv}%
        \my@ifeq{\my@etype}{ArXiv}{\parse@arxiv}%
        \my@ifeq{\my@etype}{hal}{\parse@hal}%
        \my@ifeq{\my@etype}{HAL}{\parse@hal}%
        \my@ifeq{\my@etype}{Hal}{\parse@hal}%
        \IfEmptyBibField{url}{}{%
            \def\eprint@url{\BibField{url}}%
        }%
    }%
    \IfEmptyBibField{doi}{%
        \IfEmptyBibField{url}{%
            \IfEmptyBibField{review}{%
                \IfEmptyBibField{eprint}{%
                    \def\title@url{}%
                }{%
                    \def\title@url{\eprint@url}%
                }%
            }{%
                \def\MR##1{##1}%
                \def\fld@elt##1{##1}%
                \edef\title@url{http://www.ams.org/mathscinet-getitem?mr=\BibField{review}}%
            }%
        }{%
            \def\title@url{\BibField{url}}%
        }%
    }{%
        \edef\title@url{https://doi.org/\BibField{doi}}%
    }%
    \IfEmptyBibField{doi}{}{%
        \expandafter\def\csname bib'doi\endcsname{}%
    }%
    \IfEmptyBibField{review}{}{%
        \expandafter\def\csname bib'review\endcsname{}%
    }%
    \IfEmptyBibField{language}{}{%
        \expandafter\def\csname bib'language\endcsname{}%
    }%
}
\let\original@textit\textit
\let\original@emph\emph
\apptocmd{\bib@start}{%
    \def\textit{\PrintTitleWithLink}%
    \def\emph{\PrintTitleWithLink}%
}{}{}
\NewDocumentCommand\PrintTitleWithLink{m}{%
    \ifx\title@url\empty%
        \PackageWarning{roirefs}{No link found for entry \current@citekey.}%
        \original@textit{#1}%
    \else%
        \href{\title@url}{\original@textit{#1}}%
    \fi%
    \def\texit{\original@textit}%
    \def\emph{\original@emph}%
}
\let\AmsRefsPrintAuthors\PrintAuthors
\RenewDocumentCommand\PrintAuthors{m}{%
    \AmsRefsPrintAuthors{#1}%
    \IfEmptyBibField{secondauthor}{}{, with \AmsRefsPrintAuthors{\BibField{secondauthor}}}%
}
\RenewDocumentCommand\eprint{m}{%
    \href{\eprint@url}{\tt \eprint@txt}%
}
\global\BR@BackrefAlttrue
\RenewDocumentCommand\print@backrefs{m}{%
    ~
    \begingroup%
        \expandafter\providecommand\csname brc@#1\endcsname{0}%
        \expandafter\ifcase\csname brc@#1\endcsname%
            \PackageWarning{roirefs}{Reference never cited: \current@citekey.}%
            {\color{purple} Not cited.}%
        \or%
            Cited on page \expandafter\csname brl@#1\endcsname.%
        \else%
            Cited on pages \expandafter\csname brl@#1\endcsname.%
        \fi%
    \endgroup
}

\makeatother

\NewDocumentCommand\dclthm{ m O{sibling=thm} m m O{#4s} }{%
    \declaretheorem[
        style=#1, #2, name=#4,
        refname={#4,#5}, Refname={#4,#5},
    ]{#3}
}

\dclthm{plain}[numberwithin=subsection]{thm}{Theorem}
\dclthm{plain}{mainthm}{Theorem}

\renewcommand{\thethm}{%
    \ifnum\value{subsection}>0\relax%
        \thesubsection%
    \else
        \thesection%
    \fi%
    .\arabic{thm}%
}

\RenewCommandCopy{\theHthm}{\thethm}

\dclthm {plain}      {cor}      {Corollary} [Corollaries]
\dclthm {plain}      {lem}      {Lemma}
\dclthm {plain}      {prop}     {Proposition}
\dclthm {plain}      {que}      {Question}
\dclthm {plain}      {conj}     {Conjecture}
\dclthm {definition} {dff}      {Definition}
\dclthm {definition} {notation} {Notation}
\dclthm {definition} {stt}      {Setting}
\dclthm {remark}     {rmk}      {Remark}
\dclthm {remark}     {xmp}      {Example}

\crefname{section}{Section}{Sections}
\Crefname{section}{Section}{Sections}

\numberwithin{equation}{section}

\crefname{diagram}{diagram}{diagrams}
\Crefname{diagram}{Diagram}{Diagrams}
\creflabelformat{diagram}{(#2#1#3)}

\makeatletter
\NewDocumentCommand{\Label}{ O{} m }{\label@optarg[#1]{#2}}
\makeatother

\DeclareDocumentCommand\mdef{mO{}m}{%
    \DeclareDocumentMathCommand{#1}{#2}{#3}
}

\mdef \AA       {\mathbb A}
\mdef \CC       {\mathbb C}
\mdef \FF       {\mathbb F}
\mdef \NN       {\mathbb N}
\mdef \PP       {\mathbb P}
\mdef \QQ       {\mathbb Q}
\mdef \RR       {\mathbb R}
\mdef \ZZ       {\mathbb Z}

\mdef \fa       {\mathfrak a}
\mdef \fm       {\mathfrak m}
\mdef \fn       {\mathfrak n}
\mdef \fN       {\mathfrak N}
\mdef \fp       {\mathfrak p}
\mdef \fq       {\mathfrak q}

\mdef \uf       {\underline f}
\mdef \ux       {\underline x}
\mdef \uy       {\underline y}

\mdef \bA       {\mathbf A}
\mdef \bJ       {\mathbf J}
\mdef \bL       {\mathbf L}
\mdef \bR       {\mathbf R}
\mdef \two      {\mathbf 2}

\mdef \cA       {\mathcal A}
\mdef \cB       {\mathcal B}
\mdef \cC       {\mathcal C}
\mdef \cD       {\mathcal D}
\mdef \cE       {\mathcal E}
\mdef \cF       {\mathcal F}
\mdef \cG       {\mathcal G}
\mdef \cH       {\mathcal H}
\mdef \cI       {\mathcal I}
\mdef \cJ       {\mathcal J}
\mdef \cM       {\mathcal M}
\mdef \cN       {\mathcal N}
\mdef \cO       {\mathcal O}
\mdef \cP       {\mathcal P}
\mdef \cV       {\mathcal V}
\mdef \cW       {\mathcal W}
\mdef \cX       {\mathcal X}
\mdef \cY       {\mathcal Y}
\mdef \cZ       {\mathcal Z}

\mdef \catA     {\mathsf A}
\mdef \catB     {\mathsf B}
\mdef \catC     {\mathsf C}
\mdef \catD     {\mathsf D}
\mdef \catE     {\mathsf E}
\mdef \catF     {\mathsf F}
\mdef \catG     {\mathsf G}
\mdef \catH     {\mathsf H}
\mdef \catI     {\mathsf I}
\mdef \catJ     {\mathsf J}

\mdef \aDer     {\operatorname{aDer}}
\mdef \ann      {\operatorname{ann}}
\mdef \AQ       {\operatorname{AQ}}
\mdef \bHom     {\operatorname{\mathbf{Hom}}}
\mdef \codim    {\operatorname{codim}}
\mdef \cofib    {\operatorname{cofib}}
\mdef \coker    {\operatorname{coker}}
\mdef \cone     {\operatorname{cone}}
\mdef \Csi      {\operatorname{Csi}}
\mdef \Der      {\operatorname{Der}}
\mdef \diff     {\operatorname{diff}}
\mdef \dJet     {\operatorname{dJet}}
\mdef \End      {\operatorname{End}}
\mdef \Exp      {\operatorname{Exp}}
\mdef \Ext      {\operatorname{Ext}}
\mdef \fd       {\operatorname{fd}}
\mdef \fib      {\operatorname{fib}}
\mdef \Fitt     {\operatorname{Fitt}}
\mdef \Frac     {\operatorname{Frac}}
\mdef \frk      {\operatorname{frk}}
\mdef \Gr       {\operatorname{Gr}}
\mdef \Hom      {\operatorname{Hom}}
\mdef \Hull     {\operatorname{Hull}}
\mdef \id       {\operatorname{id}}
\mdef \Im       {\operatorname{Im}}
\mdef \Jac      {\operatorname{Jac}}
\mdef \Jet      {\operatorname{Jet}}
\mdef \Kosz     {\operatorname{Kosz}}
\mdef \Lan      {\operatorname{Lan}}
\mdef \LSym     {\operatorname{{\mathbf L}Sym}}
\mdef \Maps     {\operatorname{Maps}}
\mdef \Ner      {\operatorname{Ner}}
\mdef \NL       {\operatorname{NL}}
\mdef \Ob       {\operatorname{Ob}}
\mdef \ord      {\operatorname{ord}}
\mdef \Ran      {\operatorname{Ran}}
\mdef \rank     {\operatorname{rank}}
\mdef \RHom     {\operatorname{\mathbf{R}Hom}}
\mdef \rk       {\operatorname{rk}}
\mdef \Sing     {\operatorname{Sing}}
\mdef \Spec     {\operatorname{Spec}}
\mdef \Spf      {\operatorname{Spf}}
\mdef \Supp     {\operatorname{Supp}}
\mdef \supp     {\operatorname{supp}}
\mdef \Sym      {\operatorname{Sym}}
\mdef \Tor      {\operatorname{Tor}}
\mdef \Tot      {\operatorname{Tot}}
\mdef \vol      {\operatorname{vol}}
\mdef \coeq     {\operatorname{coeq}}
\mdef \dom      {\operatorname{dom}}
\mdef \codom    {\operatorname{codom}}
\mdef \trdeg    {\operatorname{tr.deg}}

\mdef \aAlg     {\mathsf{aAlg}}
\mdef \aAlgMod  {\mathsf{aAlgMod}}
\mdef \aCRing   {\mathsf{aCRing}}
\mdef \Aff      {\mathsf{Aff}}
\mdef \Alg      {\mathsf{Alg}}
\mdef \AlgMod   {\mathsf{AlgMod}}
\mdef \aMod     {\mathsf{aMod}}
\mdef \Ani      {\mathsf{Ani}}
\mdef \Anima    {\mathsf{Anima}}
\mdef \aPsh     {\mathsf{aPsh}}
\mdef \aSind    {\mathsf{aSind}}
\mdef \Ch       {\mathsf{Ch}}
\mdef \CRing    {\mathsf{CRing}}
\mdef \dAff     {\mathsf{dAff}}
\mdef \DGAlg    {\mathsf{DGAlg}}
\mdef \dSch     {\mathsf{dSch}}
\mdef \finprod  {\mathsf{FP}}
\mdef \FinSet   {\mathsf{FinSet}}
\mdef \FreeMod  {\mathsf{FreeMod}}
\mdef \Fun      {\mathsf{Fun}}
\mdef \Ind      {\mathsf{Ind}}
\mdef \Mod      {\mathsf{Mod}}
\mdef \Mor      {\mathsf{Mor}}
\mdef \Poly     {\mathsf{Poly}}
\mdef \Psh      {\mathsf{Psh}}
\mdef \sAb      {\mathsf{sAb}}
\mdef \sAlg     {\mathsf{sAlg}}
\mdef \Sch      {\mathsf{Sch}}
\mdef \SchMod   {\mathsf{SchMod}}
\mdef \sCRing   {\mathsf{sCRing}}
\mdef \Set      {\mathsf{Set}}
\mdef \sif      {\mathsf{SIF}}
\mdef \Sind     {\mathsf{Sind}}
\mdef \sMod     {\mathsf{sMod}}
\mdef \Spaces   {\mathsf{Spaces}}
\mdef \sSet     {\mathsf{sSet}}
\mdef \vir      {\mathsf{SIF}}

\mdef \cl       {\mathrm{cl}}
\mdef \cn       {\mathrm{cn}}
\mdef \der      {\mathrm{der}}
\mdef \fg       {\mathrm{fg}}
\mdef \fpr      {\mathrm{fp}}
\mdef \free     {\mathrm{free}}
\mdef \ft       {\mathrm{ft}}
\mdef \op       {\mathrm{op}}
\mdef \red      {\mathrm{red}}
\mdef \sfp      {\mathrm{sfp}}

\mdef \embcodim {\mathrm{emb.codim}}
\mdef \embdim   {\mathrm{emb.dim}}
\mdef \jetcodim {\mathrm{jet.codim}}
\mdef \vembdim  {\mathrm{vir.emb.dim}}

\mdef \Arg      {\rule{1.5ex}{.4pt}}
\mdef \Char     {char}

\mdef \by       {{\times}}

\mdef \llparen  {(\!(}
\mdef \rrparen  {)\!)}
\mdef \sslash   {/\!\!/}

\mdef \into     {\hookrightarrow}
\mdef \onto     {\twoheadedrightarrow}
\mdef \xra      {\xrightarrow}
\mdef \xla      {\xleftarrow}
\mdef \lra      [O{\quad}]{\xrightarrow{\,#1\,}}

\mdef \Lotimes  {\mathop\otimes\limits^{\scriptscriptstyle\bL}}

\mdef \sbt      {{\mspace{1mu}\vcenter{\hbox{\scalebox{.5}{$\bullet$}}}}}

\makeatletter

\DeclareDocumentMathCommand{\mylim@}{mmm}{%
  \vtop{\m@th\ialign{##\cr
    \hfil$#2\operator@font #1$\hfil\cr
    \noalign{\nointerlineskip\kern1.5\ex@}#3\cr
    \noalign{\nointerlineskip\kern-\ex@}\cr}}%
}
\DeclareDocumentMathCommand{\invlim}{}{%
  \mathop{\mathpalette{\mylim@{lim}}{\leftarrowfill@\scriptscriptstyle}}\nmlimits@
}
\DeclareDocumentMathCommand{\colim}{}{%
  \mathop{\mathpalette{\mylim@{colim}}{\rightarrowfill@\scriptscriptstyle}}\nmlimits@
}
\DeclareDocumentMathCommand{\hocolim}{}{%
  \mathop{\mathpalette{\mylim@{hocolim}}{\rightarrowfill@\scriptscriptstyle}}\nmlimits@
}
\DeclareDocumentMathCommand{\holim}{}{%
  \mathop{\mathpalette{\mylim@{holim}}{\leftarrowfill@\scriptscriptstyle}}\nmlimits@
}

\makeatother

\DeclareDocumentTextCommand{\SPcite}{mO{Tag}}{%
    \cite[\href%
        {https://stacks.math.columbia.edu/tag/#1}%
        {#2~#1}%
    ]{SP}%
}

\DeclareDocumentTextCommand{\Kercite}{mO{Tag}}{%
    \cite[\href%
        {https://kerodon.net/tag/#1}%
        {#2~#1}%
    ]{Ker}%
}

\begin{document}


\title{Derived jet and arc spaces}

\author{Roi Docampo}

\address[R.\ Docampo]{%
Department of Mathematics\\
University of Oklahoma\\
601 Elm Avenue, Room 423\\
Norman, OK 73019 (USA)%
}

\email{roi@ou.edu}

\author{Lance Edward Miller}

\address[L.\ E.\ Miller]{%
Department of Mathematical Sciences\\
University of Arkansas\\
850 West Dickson Street, Room 309\\
Fayetteville, AR 72701 (USA)%
}

\email{lem016@uark.edu}

\author{C.\ Eric Overton-Walker}

\address[C.\ E.\ Overton-Walker]{%
~\vspace{-1em}
}

\email{eric@locallyringed.space}




\begin{abstract}
We study jet schemes and arc spaces in the context of derived algebraic
geometry. Explicitly, we consider the jet and arc functors in the category of
schemes and study their animations to the category of derived schemes---what we
call the derived jet and arc spaces. We show that the derived constructions
agree with the classical versions when the base scheme is smooth, or more
generally for local complete intersection log canonical singularities, giving a
derived interpretation to a theorem of Mustaţă. For more singular spaces we get
new singularity invariants in the form of higher homotopy groups. We also study
cotangent complexes for derived jet and arc spaces, generalizing previous
formulas for sheaves of differentials of classical jet and arc spaces. Several
applications are obtained. Specifically, we revisit recent results on the local
structure of arc spaces from the lens of cotangent complexes, giving more
unified proofs and removing unnecessary hypotheses. In particular, we extend a
version of Reguera's curve selection lemma for arc spaces to the case of
non-perfect base fields.
\end{abstract}

\maketitle

\tableofcontents

\section{Introduction}

\label{sec:introduction}

Arc spaces and jet schemes, i.e., parameter spaces of arc and jets, form a
basic tool in the study of singularities of algebraic varieties and their
birational geometry. Classically, one studies invariants of resolution of
singularities by relating them to topological invariants of jet and arc spaces.
For example, discrepancies and log canonical thresholds are computed in terms
of dimensions of jet schemes or asymptotic codimensions in arc spaces. Aiming
to capture finer invariants, there has been a recent surge in the study of jet
and arc spaces from the point of view of their \emph{scheme-theoretic}
structure \cite{dFD20,CdFD22,CdFD23}. This leads to more natural
characterizations and more effective computational tools. For example, Mather
discrepancies are computed directly as embedding dimensions of certain local
rings \cite{MR18}. Fundamental for these developments is the main theorem of
\cite{dFD20}, which provides explicit formulas to compute and understand
sheaves differentials on jet and arc spaces.

One is naturally motivated to extend the main result of \cite{dFD20} to get
formulas for the \emph{cotangent complexes} of jet and arc spaces. Since
cotangent complexes are inherently the derived analogues of sheaves of
differentials, we are led to a much more ambitious goal: laying down the
foundations of the theory of jet and arc spaces in the context of \emph{derived
algebraic geometry}.

Before going into technical details, we give a quick summary of the main
contributions of this paper. Throughout we refer to a scheme that is not
derived as \emph{classical}.

\begin{itemize}[left=.5em,itemsep=.5em]

\item
We give a detailed construction for the derived arc space $\bL J_\infty (X)$
and the derived jet schemes $\bL J_n(X)$ associated to a derived scheme $X$. We
show these derived schemes represent natural moduli problems. Derived jet and
arc spaces are classical when $X$ is smooth, thus leading to a natural role for
them in the study of singularities.

\item
We show that when $X$ is a classical local complete intersection (lci), then
$\bL J_n(X)$ is classical if and only if $J_n(X)$ is lci of the expected
dimension. By a theorem of Mustaţă, this is in turn equivalent to log
canonicity, and we get a derived interpretation of this phenomenon.

\item
We give a formula for the cotangent complex $\bL_{\bL J_\infty (X)}$ of the
derived arc space generalizing the main result of \cite{dFD20}, and similarly
for derived jet schemes. We show that a formula of the same type cannot hold
for the classical arc/jet spaces, justifying the introduction of derived
techniques.

\item
We revisit the main theorems in \cite{dFD20,CdFD22,CdFD23} and their proofs
from the lens of the cotangent complex. We remove unnecessary hypotheses and
provide more streamlined and unified arguments. As a main application, we
obtain a general version of Reguera's \emph{curve selection lemma} for arc
spaces \cite{Reg06,dFD20,Reg21} without any conditions on the ground field.
Previous versions required perfect base fields.

\end{itemize}

Apart from the applications proposed in this paper, we see the non-trivial
higher derived structure of derived arc and jet spaces as a fertile new ground
to control subtle invariants of singularities, but to keep the article within a
manageable size and scope, we must leave a more robust study of these ideas to
future investigations.

We point out that ours is not the first attempt at using arc and jet spaces in
a derived setting. As far as we are aware, derived versions of jet and arc
schemes are first mentioned in \cite[Sec.~9.2]{GR12} as part of a study of loop
spaces and affine Grassmannians from the point of view derived algebraic
geometry. Further investigations appear in \cite{Hen15}, now with applications
to higher-dimensional loop spaces. All of these works use differential graded
algebras in the foundations for derived algebraic geometry \cite{TV08,GR12},
which has the shortcoming of only providing a robust theory in characteristic
0. Since our intended applications are in arbitrary characteristic, e.g.\ we
are interested in non-perfect base fields, and the technical details are not
significantly harder, we have chosen to use simplicial algebras and the theory
of animation. To aid the interested reader less familiar with the background
and foundations of derived algebraic geometry, as well as to set notation and
conventions for this paper, we include in \cref{sec:preliminaries} an expansive
and detailed summary of the needed constituent parts of derived algebraic
geometry via animation.


In the remainder of the introduction we give a more detailed account of the
main results of the paper.

\subsection*{Derived jet and arc schemes}

Fix a base ring $k$ and a $k$-scheme $X$. Recall that the $n$-th jet
scheme $J_n(X)$ of $X$ is the parameter space of $n$-jets. In other words, if
we consider the functor of $n$-jets on $X$,
\[
    \Jet_n^X \colon \Sch_k \to \Set,
    \qquad
    \Jet_n^X(Z) = \Hom_{\Sch_k}(Z \times_{k} \Spec k[t]/(t^{n+1}), X),
\]
then $J_n(X)$ is the $k$-scheme representing $\Jet_n^X$. One immediately checks
that this construction gives rise to a functor $J_n \colon \Sch_k \to \Sch_k$.

There are two possibilities for constructing a derived analogue. First, one can
naturally extend the functor $\Jet^X_n$ to the $\infty$-category $\dSch_k$ of
derived $k$-schemes:
\[
    \dJet_n^X \colon \dSch_k \to \Spaces,
    \qquad
    \dJet_n^X(Z)
    =
    \Maps_{\dSch_k}(Z \times^\bL_k \Spec k[t]/(t^{n+1}), X).
\]
Second, we can consider the animation $\bL J_n \colon \dSch_k \to \dSch_k$,
which informally should be thought as the ``left derived'' functor of $J_n$. We
call $\bL J_n$ the \emph{derived $n$-jet functor}. Our first theorem is that
these two approaches give rise to the same answer.

\begin{mainthm}[\cref{thm:allagree,thm:globalizejets}]
\label{main-theorem-A}
For any a derived $k$-scheme $X$, the derived $n$-jet scheme $\bL J_n(X)$
represents the functor $\dJet^X_n$.
\end{mainthm}

The arc space $J_\infty(X)$ is defined as the limit of the jet schemes under
the truncation maps, which one shows is always represented by a scheme. For
derived arc spaces we have three possible notions: the limit $\dJet_\infty^X$
of the functor of points; the animation $\bL J_\infty$ of the functor
$J_\infty$; and the homotopy limit $\holim \bL J_n$. They all give the same
answer:

\begin{mainthm}[\cref{thm:ani_arc_hocolim,thm:allagree,thm:globalizejets}]
\label{main-theorem-B}
For any a derived $k$-scheme $X$, the derived arc space $\bL J_\infty(X)$
represents the functor $\dJet^X_\infty$ and is naturally weakly equivalent to
the homotopy limit $\holim \bL J_n(X)$.
\end{mainthm}

For a derived scheme $Z$, we denote by $\pi_0(Z)$ the underlying classical
scheme.
%

\begin{mainthm}[\cref{thm:SpecSym}]
\label{main-theorem-C}
Let $n \in \NN \cup \{ \infty \}$, and let $X$ be a derived $k$-scheme.
\begin{enumerate}
\item\label[part]{main-theorem-C.1}
    $\pi_0 (\bL J_n (X)) = J_n(\pi_0(X))$.
\item\label[part]{main-theorem-C.2}
    $\bL J_1 (X) = \Spec_X (\LSym_{\cO_X} (\bL_{X/k}))$.
\end{enumerate}
Here $\bL_{X/k}$ is the cotangent complex of $X$ over $k$.
\end{mainthm}

\Cref{main-theorem-C.1} of \cref{main-theorem-C} asserts that $\bL J_n(X)$
should be thought as a derived enrichment of the classical $J_n(X)$ and
\Cref{main-theorem-C.2} of \cref{main-theorem-C} states that the first derived
jet scheme should be thought as the total space of the cotangent complex, just
as the first underived jet scheme is the total space of the sheaf of
differentials.

Classical jet and arc spaces have a predictable structure for smooth (or even
formally smooth) schemes, but become interesting in the presence of
singularities. Subtle invariants of singularities can be defined or controlled
using the topological and scheme structures on jet and arc spaces. As mentioned
above, it is our hope that derived jet and arc spaces can capture deeper
invariants in the form of higher homotopy groups. This is interesting even when
the underlying scheme $X$ is not necessarily derived, i.e., classical.
Specifically, we will see that if $X$ is a smooth $k$-scheme, then $\bL
J_\infty(X)$ and $\bL J_n(X)$ are classical for all $n$. So we can think of
$\pi_i(\cO_{\bL J_\infty(X)})$ and $\pi_i(\cO_{\bL J_n(X)})$ as new singularity
invariants.

\subsection*{Quasi-smoothness and local complete intersections}

The construction of the derived arc and jet spaces is particularly explicit in
what is known as the \emph{quasi-smooth} case. Consider a polynomial ring
$k[\ux]$ for some fixed set of variables $\ux = x_1, \ldots, x_d$, and let $\uf
= f_1, \ldots, f_c$ be a sequence of polynomials in $k[\ux]$. Classically it is
natural to consider the quotient $k[\ux]/(\uf)$ and the corresponding affine
scheme. In derived algebraic geometry it is more natural to consider instead
the \emph{derived quotient} $k[\ux]\sslash(\uf)$, which is defined using an
appropriate pushout diagram in the $\infty$-category of animated algebras, as
we elaborate in detail in \cref{sec:intro-koszul-regular}. The derived
$k$-schemes locally constructed in this way are called \emph{quasi-smooth}.

\begin{mainthm}[\cref{thm:LJderived-quotients,lem:LJquasi-smooth}]
\label{main-theorem-D}
The derived jet schemes of a quasi-smooth scheme are quasi-smooth.
Furthermore, for $n \in \NN \cup \{\infty\}$
\[
    \text{if}
    \qquad
    X = \Spec
    \begin{array}{@{\,}c@{\,}}
        k[\ux]
        \\\hline\\[-1.1em]\hline
        (\uf)
    \end{array}
    \qquad
    \text{then}
    \qquad
    \bL J_n(X) = \Spec
    \begin{array}{@{\,}c@{\,}}
        k[\ux, \ux^{(1)}, \ldots, \ux^{(n)}]
        \\\hline\\[-1.1em]\hline
        (\uf, \uf^{(1)}, \ldots \uf^{(n)})
    \end{array},
\]
where $f_i^{(q)}$ denotes the $q$-th Hasse-Schmidt differential of $f_i$.
\end{mainthm}

In other words, the familiar procedure to generate the equations of the jet
schemes via Hasse-Schmidt differentials also works for derived jet schemes, as
long as quotients are replaced by derived quotients.

The usual quotient is the classical algebra underlying the derived quotient,
that is, $\pi_0(k[\ux]\sslash(\uf)) = k[\ux]/(\uf)$. The higher homotopy groups
$\pi_i(k[\ux]\sslash(\uf))$ are naturally identified with the Koszul homology
of the sequence $\uf$. In particular, we have an equivalence
$k[\ux]\sslash(\uf) = k[\ux]/(\uf)$ precisely when $\uf$ is a regular sequence.
In other words, \cref{main-theorem-D} gives an explicit presentation of the
derived jet schemes in the \emph{local complete intersection case}, and in this
situation it interprets the homotopy groups $\pi_i(\cO_{\bL J_n(X)})$ in terms of
Koszul homology of the natural defining equations of the jet schemes. As an
interesting application we prove the following result.

\begin{mainthm}[\cref{thm:mjlc-discrete-jets}]
\label{main-theorem-E}
Let $X$ be a classical reduced scheme of finite type over a field $k$ of
characteristic zero. If $X$ is lci, then $X$ is log MJ-canonical if and only if
$\bL J_n(X)$ is classical for all $n \in \NN$.
\end{mainthm}

Here \emph{log MJ-canonical} refers to the framework of singularities provided
by the theory of \emph{Mather-Jacobian discrepancies} \cite{dFD14,EIM16,EI15}.
If $X$ is embedded in a smooth variety $Z$ with codimension $c$, then $X$ is
log MJ-canonical precisely when the pair $(Z, cX)$ is log canonical in the
usual sense. In particular, we can see \cref{main-theorem-E} as a derived
interpretation of a theorem of Mustaţă \cite{Mus01,EMY03,EM04}, as we discuss
in \cref{sec:lci_non_discreteness} in more detail. Notice the presence of the
lci hypothesis in \cref{main-theorem-E}. It would be interesting to study the
restrictions that the classicality of the derived jet schemes imposes on the
singularities of a variety beyond the lci case.

\subsection*{Cotangent complexes of jets and arcs}

For clarity of exposition, we focus on the affine case. An affine derived
$k$-scheme is of the form $X = \Spec(A_\sbt)$, where $A_\sbt$ is an animated
$k$-algebra, i.e., a simplicial $k$-algebra considered up to weak equivalence.
For $n \in \NN \cup \{\infty\}$ the derived jet/arc schemes $\bL J_n(X)$ are
also affine, and we write $\bL J_n(A_\sbt)$ for the corresponding animated
$k$-algebras. From their modular interpretation, we obtain the \emph{universal}
jets and arcs, which algebraically are given by
\[
    A_\sbt \lra \bL J_n(A_\sbt)[t]/(t^{n+1})
    \qquad\text{and}\qquad
    A_\sbt \lra \bL J_\infty(A_\sbt)\llbracket t \rrbracket.
\]
We consider the following:
\[
    P_n^\der(A_\sbt)
    =
    \frac{
        t^{-n} \, \bL J_n(A_\sbt)[t]
    }{
        t \, \bL J_n(A_\sbt)[t]
    }
    \qquad\text{and}\qquad
    P_\infty^\der(A_\sbt)
    =
    \frac{
        \bL J_\infty(A_\sbt)\llparen t \rrparen
    }{
        t \, \bL J_\infty(A_\sbt)\llbracket t \rrbracket
    }.
\]
These are naturally animated modules over $\bL J_n(A_\sbt)$ or $\bL
J_\infty(A_\sbt)$, that is, connective objects in the corresponding derived
categories. Via the universal jet/arc they also gain a structure of animated
$A_\sbt$-modules, and we will consider them as animated bimodules.

\begin{mainthm}[\cref{thm:affinederiveddFD,thm:deriveddFD}]
\label{main-theorem-F}
For $n\in \NN \cup \{\infty\}$,
\[
    \bL_{\bL J_n(A_\sbt)/k}
    \cong
    \bL_{A_\sbt/k} \otimes_{A_\sbt}^\bL P_n^\der(A_\sbt).
\]
\end{mainthm}

When we take $0$-th homotopy we recover the main formula of \cite{dFD20}.
Concretely, $A = \pi_0(A_\sbt)$ is a classical $k$-algebra. We define $P_n(A)$
and $P_\infty(A)$ by the same formulas as above but replacing the derived
jet/arc functors by the classical jet/arc functors. Then by applying $\pi_0$ to
the formula in \cref{main-theorem-F} we obtain
\begin{equation}
    \label{eq:diff-formula}
    \Omega_{J_n(A)/k}
    \cong
    \Omega_{A/k}
    \otimes_A
    P_n(A),
\end{equation}
which is exactly \cite[Theorems A and B]{dFD20}.

Notice that our formula is for the cotangent complex of the derived jet/arc
schemes, not of the classical counterparts. In fact we prove the following.

\begin{mainthm}[\cref{thm:we_need_derived_jets}]
Let $n \in \NN \cup \{\infty\}$ and consider a classical $k$-algebra $A$.
There is an equivalence
\[
    \bL_{J_n(A)/k}
    \cong
    \bL_{A/k} \otimes_A^\bL P_n(A)
\]
if and only if the derived jet/arc space $\bL J_n(A)$ is classical.
\end{mainthm}

In light of \cref{main-theorem-E} there are many examples of classical schemes
whose derived jet schemes are not classical. For any of these we do not know
how to compute the cotangent complex of the classical jet schemes. This is the
main reason we decided to use derived algebraic geometry in our study.

As in the classical case, the formula in \cref{main-theorem-F} allows us to get
explicit descriptions of $\bL_{\bL J_n(X)/k}$ whenever $\bL_{X/k}$ is computable.
In the spirit of \cite[Sec.~6]{dFD20} we are also able to control fibers of
cotangent complexes, that is, certain André-Quillen homology groups. To achieve
this, we were led to investigate a derived variant of the notion of Fitting
ideals. Specifically, using the theory of \emph{cohomological support loci}
developed in \cite{GL87}, we consider \emph{cohomological support ideals}
$\Csi_{(i,p)}(M_\sbt)$ associated certain complexes $M_\sbt$ in the derived
category of $X$. These lead in turn to the \emph{higher Jacobian ideals},
$\Jac_X^{(i,p)} = \Csi_{(i,p)}(\bL_{X/k})$, which are guaranteed to be defined
for schemes essentially of finite type over $k$. These ideals control the
fibers of the cotangent complex of derived jet/arc schemes, as discussed in
detail in \cref{sec:fibers-fitting}. A sample result is the following.

\begin{mainthm}[\cref{cor:cotangent-fibers-non-deg}]
Let $X$ be a classical scheme essentially of finite type over a field $k$, and
let $\alpha$ be a non-degenerate arc on $X$ with residue field $k_\alpha$.
Then:
\[
    \dim_{k_\alpha}
    (
    \pi_i(
    \bL_{\bL J_\infty(X)/k}
    \otimes^\bL
    k_\alpha
    )
    )
    =
    \ord_\alpha\left(
        \Jac_X^{(i-1,0)}
    \right)
    \qquad
    \text{for $i \ge 2$}.
\]
\end{mainthm}

\subsection*{Cotangent maps}

In \cite{dFD20,CdFD22,CdFD23} the formula \eqref{eq:diff-formula} was used
effectively to establish a series of results about the local structure of arc
spaces and of maps between arc spaces, with applications to the study of
invariants of singularities. But sheaves of differentials turn out to be only
partially adequate for this analysis, and many of the obtained results had
extraneous hypotheses and somewhat unintuitive proofs. Most notably, the ground
field was often required to be perfect.

In the current paper we use \cref{main-theorem-F} instead of formula
\eqref{eq:diff-formula} to study several of the results of
\cite{dFD20,CdFD22,CdFD23}. As a main consequence of the new approach we remove
the perfectness assumption on the ground field in all results. Moreover, our
arguments involve directly the homotopy groups of cotangent complexes and their
associated fundamental triangles, which we believe is more natural and
clarifies the proofs.

The main setup is as follows. We consider maps $f \colon X \to Y$ between
classical schemes essentially of finite type over a field $k$, as well as the
induced maps $f_\infty \colon J_\infty(X) \to J_\infty(Y)$ between the
corresponding classical arc spaces. We pick an arc $\alpha \in J_\infty(X)$ and
its image $\beta = f_\infty(\alpha) \in J_\infty(Y)$. Writing $k_\alpha$ and
$k_\beta$ for the residue fields of $\alpha$ and $\beta$, and $\fm_\alpha$ and
$\fm_\beta$ for their defining ideals, we are interested in understanding the
Zariski cotangent spaces $\fm_\alpha/\fm_\alpha^2$ and $\fm_\beta/\fm_\beta^2$,
as well as the cotangent map
\[
    T_\alpha^{*}f_\infty
    \colon
    \fm_\beta/\fm_\beta^{2}\otimes_{k_\beta}k_\alpha
    \to
    \fm_\alpha/\fm_\alpha^{2}
\]
induced by $f_\infty$. We let $\eta \in \Spec k_\alpha\llbracket t \rrbracket$ be the non-closed
point and consider $\alpha(\eta) \in X$, what we call the \emph{generic point}
of the arc $\alpha$. We introduce the following two quantities:
\[
    d = \dim_{k_{\alpha(\eta)}}(\Omega_{X/k} \otimes_{\cO_X} k_{\alpha(\eta)})
    \quad\text{and}\quad
    r = \dim_{k_{\alpha(\eta)}}(\Omega_{X/Y} \otimes_{\cO_X} k_{\alpha(\eta)}).
\]
Notice that $d = \dim_{\alpha(\eta)}(X)$ when $X$ is smooth at $\alpha(\eta)$,
but it could be bigger in general. We write $\Jac_f = \Fitt^0(\Omega_{X/Y})$
for the Jacobian ideal of $f$. Recall that the \emph{embedding dimension} of
the local ring $\cO_{J_\infty(X),\alpha}$ is the dimension of
$\fm_\alpha/\fm_\alpha^2$, and similarly for $\beta$. The following theorem
features general linear projections---see
\cref{thm:generic_projection_properties} for a precise articulation.

\begin{mainthm}[\cref{%
    thm:gen_etale_implies_embedding_dimension_bounds%
    ,%
    thm:smooth_at_generic_point_implies_bound_on_kernel_of_cotangent_map%
    ,%
    thm:unramified_at_alpha_eta_implies_surjective_cotangent_map%
    ,%
    thm:general_linear_projection_implies_cotangent_map_is_an_isomorphism%
}]

With no assumption on the ground field, the following claims hold:

\begin{enumerate}
    \item
        If $f$ is unramified at $\alpha(\eta)$, then $T^*_\alpha f_\infty$ is
        surjective and $k_\alpha/k_\beta$ is finite separable.
    \item
        If $f$ is smooth at $\alpha(\eta)$, then
        \[
            \dim_{k_\alpha}(\ker (T_\alpha^{*}f_\infty))
            \leq
            \ord_\alpha(\Fitt^{r}(\Omega_{X/Y})).
        \]
    \item
        If $f$ is generically étale, then
        \[
            \quad\qquad
            \embdim(\cO_{J_\infty (X),\alpha})
            \leq
            \embdim(\cO_{J_\infty (Y),\beta})
            \leq
            \embdim(\cO_{J_\infty (X),\alpha})+\ord_\alpha(\Jac_{f}).
        \]
    \item
        If $f$ is generically étale and $X$ is smooth, then
        \[
            \embdim(\cO_{J_\infty (Y),\beta})
            =
            \embdim(\cO_{J_\infty (X),\alpha})+\ord_\alpha(\Jac_{f}).
        \]
    \item
        If $X \subset \AA^n$ and $f \colon X \to \AA^d$ is the morphism induced
        by a general linear projection $\AA^n \to \AA^d$, then $T^*_\alpha
        f_\infty$ is an isomorphism and therefore
        \[
            \embdim(\cO_{J_\infty(X),\alpha})
            =
            \embdim(\cO_{J_\infty(Y),\beta}).
        \]
\end{enumerate}
\end{mainthm}

Notice that all the statements involve only classical arc spaces, and not their
derived counterparts. The crucial observation is that the Zariski cotangent
space can be identified with $\fm_\alpha/\fm_\alpha^2 = \pi_1(\bL_{k_\alpha/\bL
J_\infty(X)})$, and this can be controlled using \cref{main-theorem-F}.

\subsection*{The curve selection lemma}

Our final application relates to \emph{Reguera's curve selection lemma}
\cite{Reg06,Reg21}. This is a foundational result in the theory of arc spaces,
and it is a key ingredient in all known approaches to the Nash problem
\cite{Nas95,FdBPP12,dFD16}, see \cite{dF18} for an introduction. In the version
that we will present, the lemma is formulated as a finiteness statement
regarding completions of local rings on arc spaces. The previously known
version \cite[Thm.~10.7]{dFD20} required a perfect ground field, and our main
contribution, as in the case of the study of cotangent maps, is the removal of
this hypothesis.

To state the result, we need the notions of \emph{stable arc} and of \emph{jet
codimension} of an arc. The jet codimension of an arc $\alpha \in J_\infty(X)$
is the quantity
\[
        \jetcodim(\alpha, J_\infty(X))
        =
        \lim_{n \to \infty} \left(
            (n+1) \, d
            - \dim(\overline{\{\alpha_{n}\}})
        \right)
        \in \NN \cup \{ \infty \}
\]
where $\alpha_n \in J_n(X)$ denotes the truncation of $\alpha$ to the $n$-th
jet level. Stable arcs are the generic points of irreducible constructible
subsets not contained in the non-smooth locus, see \cref{sec:stable_arcs} for a
detailed discussion. Important examples of stable arcs are the \emph{maximal
divisorial arcs} $\alpha_E$ associated to divisorial valuations $E$, whose
geometry can be used to control discrepancies and therefore invariants of
singularities in the context of the minimal model program.

\begin{mainthm}[\cref{thm:edim-equals-jetcodim,thm:stable-via-edim}]
Let $k$ be a field, let $X$ be a classical scheme and assume that it is
essentially of finite type and generically smooth over $k$. For $\alpha\in
J_\infty(X)$ an arc,
\[
    \embdim(\cO_{J_\infty(X),\alpha})
    =
    \jetcodim(\alpha,J_\infty(X)).
\]
Moreover, the following are equivalent:
\begin{enumerate}
    \item $\alpha$ is stable.
    \item $\jetcodim(\alpha,J_\infty(X)) < \infty$.
    \item $\embdim(\cO_{J_\infty(X),\alpha}) < \infty$.
    \item The completed local ring $\widehat\cO_{J_\infty(X),\alpha}$ is
    Noetherian.
\end{enumerate}
\end{mainthm}

Notice that again the final statement is purely classical. But derived
techniques appear in the proofs in a crucial way, again via the identification
$\fm_\alpha/\fm_\alpha^2 = \pi_1(\bL_{k_\alpha/\bL J_\infty(X)})$ and
\cref{main-theorem-F}, and we believe are unavoidable in our approach when
dealing with the non-perfect case. Particularly interesting is the technical
\cref{lem:boundonimagelambdan-non-deg}, which features a higher Jacobian ideal
and is fundamental for relating embedding dimensions with jet codimensions.

\subsection*{Structure of the paper}

In \cref{sec:preliminaries} we overview the basics of derived algebraic
geometry, using the language of animation. This also serves as an introduction
to the foundations for those readers interested in the purely underived
applications of our work, an expanded and roughly self-contained summary of the
needed background to understand the methods. \Cref{sec:DerJets} lays the
foundations for derived jets and arcs, starting in the affine case in
\cref{sec:AffCase} and continuing with \cref{sec:Globcase} for the global
considerations. Classicality results, lci considerations, and connections to
Mustaţă's theorem are covered in \cref{sec:discrete-non-discrete}.
\Cref{sec:DerJetsCotCplx} builds the statement and proof of
\cref{main-theorem-F}. Finally, the rest of the article contains a series of
applications. \Cref{sec:fibers-fitting} outlines the theory of cohomological
support ideals and higher Jacobians, and \cref{sec:emb-dim} contains technical
results regarding the computation of embedding dimensions. The applications to
the study of cotangent maps appear in \cref{sec:cotangent_maps}, and our
version of the curve selection lemma is proved in \cref{sec:curveselection}.

\subsection*{Acknowledgments}

We were first inspired to look at derived techniques for the study of arc
spaces at a conference on derived algebraic geometry at SLMath (formerly MSRI)
in 2019. During the development of this article, and during the process of
learning the techniques of derived algebraic geometry, we benefited greatly
from discussions with many colleagues and from interactions in the homotopy
theory Discord server. We are particularly thankful to Bhargav Bhatt, Emile
Bouaziz, Christopher Chiu, Tommaso de Fernex, Adeel Khan, Mircea Mustaţă, Denis
Nardin, and Srikanth Iyengar.

\section{Preliminaries: derived schemes and animation}

\label{sec:preliminaries}

We use the foundations for derived algebraic geometry laid out by Lurie. The
main references are \cite{Lur09,DAG}, but see also \cite{Toe14} and the
excellent survey \cite{Cis19}. We use the theory of animation as presented in
\cite[Sec.\ 5]{CS20} as the primary vehicle for describing categories of
derived schemes and functors between them. The goal of this section is to
summarize the derived setting we are using, fix some standard notation, and
provide an overview for non-experts. A reader who is familiar with the basics
of derived algebraic geometry and animation might want to move directly to
\cref{sec:DerJets}.

\subsection{Ground ring}

Throughout, all rings are considered to be commutative and unital. We fix a
base ring, which we denote by $k$. For many applications, $k$ can
be taken to be a field, possibly algebraically closed, but these assumptions
are not necessary for our developments and must be avoided in some applications.
Most of our theory can be extended to be fully relative, where $k$ is replaced
by an arbitrary scheme $S$, possibly derived. But for clarity of exposition and
to avoid heavy technicalities, we have chosen to present our results only for
an affine underived base.

\subsection{1-categories}

To emphasize the distinction with $\infty$-categories, we use the term
$1$-category to refer to a category in the usual sense. We denote by $\Alg_k$ and
$\Sch_k$ the $1$-categories of $k$-algebras and $k$-schemes respectively. We
write $\Aff_k \cong \Alg_k^\op$ for the $1$-category of affine $k$-schemes. The
$1$-categories of polynomial $k$-algebras, finite-type polynomial $k$-algebras,
finitely presentable $k$-algebras, and finitely presentable $k$-schemes will be
denoted by $\Poly_k$, $\Poly_k^\ft$, $\Alg_k^\fpr$, and $\Sch_k^\fpr$
respectively.

For a $1$-category $\catC$, we denote by $\Ob(\catC)$ the class of objects and
we denote by $\Psh(\catC)$ its category of presheaves of sets. Recall that this
is the functor category $\Psh(\catC) = \Fun(\catC^\op,\Set)$.

\subsection{Projectively generated 1-categories}

For the process of animation described below, a key feature of $\Alg_k$ is the
fact that it is an example of a finitely presentable $1$-category \cite{AR94}.
It is also a projectively generated $1$-category in the sense of Lurie; see
\cite[Def.~5.5.8.23]{Lur09}, where $\Poly_k^\ft$ gives an essentially small
collection of compact projective generators. To get an intuition for what this
means, first observe that every finitely presentable $k$-algebra $A$ can be
written as a reflexive coequalizer of objects in $\Poly_k^\ft$:
\[
\begin{tikzcd}
    k[x_1, \ldots, x_n, y_1, \ldots, y_m]
    \arrow[rr, bend left=15,  "f",
        start anchor={[yshift=1ex]east}, end anchor={[yshift=1ex]west}]
    \arrow[rr, bend right=15, "g"',
        start anchor={[yshift=-1ex]east}, end anchor={[yshift=-1ex]west}]
    &&
    k[x_1, \ldots, x_n]
    \arrow[ll,"s" description]
    \arrow[r,"p"]
    &
    \coeq(f,g) = A.
\end{tikzcd}
\]
Here $p$ sends the $x_i$ to $k$-algebra generators of $A$, the three maps
$f,g,s$ fix all the $x_i$, $f$ sends the $y_j$ to the relations among the
algebra generators of $A$, i.e., to a choice of generators of the kernel of
$p$, and $g$ sends the $y_j$ to $0$. Second, notice that an arbitrary
$k$-algebra $A$ is a filtered colimit of finitely presentable $k$-algebras
because it is a union of its finitely generated subalgebras. Combining these
two observations, we see that every $k$-algebra is a colimit of finite-type
polynomial algebras:
\[
    A = \colim P_i
    \qquad
    P_i \in \Ob(\Poly_k^\ft).
\]
Moreover, this colimit is of a special type as it was obtained by combining
reflexive coequalizers and filtered colimits. It is called a \emph{sifted}
colimit. We will not discuss the notion of sifted colimit further, and will
not need it, but we refer the reader to \cite{ARV10, CS20} for details.
Summarizing, every $k$-algebra is canonically a sifted colimit of finite-type
polynomial algebras.

This can be taken further, and we can show that $\Alg_k$ can be reconstructed
from $\Poly_k^\ft$ in the following way \cite{AR94, ARV10}. The Yoneda
embedding
\[
    \Alg_k \hookrightarrow \Psh(\Alg_k) = \Fun(\Alg_k^\op, \Set)
\]
induces by restriction a canonical functor $\Alg_k \to \Psh(\Poly_k^\ft)$,
which can be shown to be fully faithful and to preserve sifted colimits. Up to
equivalence, this identifies $\Alg_k$ with the free completion of $\Poly_k^\ft$
under sifted colimits, which we denote $\Sind(\Poly_k^\ft)$. We get the
following diagram:
\[
    \Poly_k^\ft
    \hookrightarrow
    \Alg_k
    \cong
    \Sind(\Poly_k^\ft)
    \hookrightarrow
    \Psh(\Poly_k^\ft).
\]
Two properties of $\Sind(\Poly_k^\ft)$ are particularly useful when making
categorical constructions. First, we have an equivalence
\[
    \Sind(\Poly_k^\ft) \cong \Fun_\finprod(\Poly_k^{\ft,\op}, \Set)
    \hookrightarrow
    \Psh(\Poly_k^\ft) = \Fun(\Poly_k^{\ft,\op}, \Set),
\]
where $\Fun_\finprod$ denotes the sub-category of $\Fun$ spanned by the
functors that preserve finite products. In words, the objects of
$\Sind(\Poly_k^\ft) \cong \Alg_k$ can be described as the
finite-product-preserving functors $\Poly_k^{\ft,\op} \to \Set$. Second, we
have the universal property of free completions under sifted colimits: if
$\catD$ is any $1$-category admitting all sifted colimits, then
\[
    \Fun_\sif(\Sind(\Poly_k^\ft), \catD)  \to \Fun(\Poly_k^\ft, \catD)
    \qquad
    F \mapsto F|_{\Poly_k^\ft}
\]
is an equivalence of categories. Here $\Fun_\sif$ denotes the full subcategory
of $\Fun$ spanned by the functors that preserve sifted colimits. In other words, a
sifted-colimit-preserving functor $\Alg_k \to \catD$ is determined by its
values on $\Poly_k^\ft$.

For this last property, we want to recall one of the main results of
\cite{ARV10}: since $\Alg_k$ admits all finite colimits, a functor $\Alg_k \to
\catD$ preserves all sifted colimits if and only if it preserves all filtered
limits and all reflexive coequalizers. It is for this reason that we will not
need to use the notion of sifted colimit directly.

These considerations generalize to a larger class of $1$-categories. An object
$X$ in a $1$-category $\catC$ is said to be \emph{strongly finitely
presentable} or a \emph{compact projective} object, provided the functor
$\Hom_{\catC}(X, \Arg)$ preserves sifted colimits. The $1$-category $\catC$ is
\emph{projectively generated} if $\catC$ admits all small colimits and there
exists a subcategory $\catC_0 \subset \catC$ such that: $\catC_0$ is
essentially small, i.e., equivalent to a small $1$-category; $\catC_0$ admits
all finite coproducts; the objects of $\catC_0$ are all strongly finitely
presentable in $\catC$; and every object of $\catC$ is a sifted colimit of
objects in $\catC_0$. In this situation we get:
\[
    \catC_0
    \hookrightarrow
    \catC
    \cong
    \Sind(\catC_0)
    \hookrightarrow
    \Psh(\catC_0).
\]
As above, the objects of $\catC$ get identified with the functors $\catC_0^\op
\to \Set$ that preserve finite products. If $\catD$ admits all sifted
colimits, a sifted-colimit-preserving functor $F \colon \catC \to \catD$ is
determined by its restriction $F|_{\catC_0}$. Moreover, a functor $\catC \to
\catD$ preserves all sifted colimits if and only if it preserves filtered
colimits and reflexive coequalizers.

For a $1$-category $\catC$, we denote by $\catC^\sfp$ the subcategory spanned
by the strongly finitely presentable objects. When $\catC^\sfp$ is essentially
small and generates $\catC$ under sifted colimits, it is the most natural
choice for $\catC_0$ in the definition of projectively generated category. In
the case of algebras, the objects of $\Alg_k^\sfp$ are retracts of finite-type
polynomial $k$-algebras.

\subsection{Simplicial sets and algebras}

Before discussing animation, we establish notation for simplicial objects. For
$n \in \NN$ we denote by $[n]$ the set $\{ 0,\ldots,n\}$ and by $\Delta$ the
category with objects $\{ [n] \}_{n \in \NN}$ and morphisms non-decreasing
order-preserving maps. A simplicial object in a $1$-category $\catC$ is a
functor $\Delta^\op \to \catC$. We denote by $\sSet$ the $1$-category of
simplicial sets and by $\sAlg_k$ the $1$-category of simplicial $k$-algebras.
We denote an object of $\sSet$ by $X_\sbt$, and by $X_n$ the corresponding
image of $[n]$. These come equipped with face maps $d^n_i \colon X_n \to
X_{n-1}$ and degeneracy maps $s_i^n \colon X_n \to X_{n+1}$ for $0 \leq i \leq
n$ satisfying the familiar relations.

For simplicial sets $X_\sbt$ and $Y_\sbt$, a morphism of simplicial sets
$X_\sbt \to Y_\sbt$ is just a natural transformation and can be described as a
collection of functions $\varphi = (\varphi_n \colon X_n \to Y_n)_{n \in \NN}$,
satisfying the usual compatibilities with the face and degeneracy maps. We can
form a mapping simplicial set $\bHom(X_\sbt, Y_\sbt)$ with the $n$-simplices
$\bHom(X_\sbt, Y_\sbt)_n$ given by $\Hom_\sSet(X_\sbt \times \Delta^n,
Y_\sbt)$. There is a natural forgetful functor $\sAlg_k \to \sSet$ and
morphisms $\varphi \colon R_\sbt \to S_\sbt$ in $\sAlg_k$ are characterized as
morphisms of simplicial sets $\varphi$ such that each $\varphi_n \colon R_n \to
S_n$ is a morphism of $k$-algebras.

To each simplicial set $X_\sbt$ we can associate its geometric realization
$|X_\sbt|$, and we get a functor from $\sSet$ to the $1$-category of
topological spaces. Its right adjoint sends a topological space to the
simplicial set of its singular chains.

For a simplicial set $X_\sbt$, we denote by $\pi_i(X_\sbt)$ the $i$-th
\emph{homotopy group} of $|X_\sbt|$.  One calls $X_\sbt$ \emph{discrete} if
$\pi_i(X_\sbt) = 0$ for $i > 0$. A morphism $X_\sbt \to Y_\sbt$ is called a
\emph{weak equivalence} if the induced group homomorphisms $\pi_i(X_\sbt) \to
\pi_i(Y_\sbt)$ are isomorphisms for all $i$.

For a simplicial $k$-algebra $R_\sbt$, define $\pi_i(R_\sbt)$ as the homotopy
group of the underlying simplicial set. We get the notions of a discrete
simplicial $k$-algebra and of a weak equivalence between simplicial
$k$-algebras. Notice that the underlying simplicial set of a simplicial
$k$-algebra is always a Kan complex, so its homotopy groups can be computed
directly from the data of the simplicial set, see \cite[Sec.~8.3]{Wei94}.
Alternatively, via the Dold-Kan correspondence the $1$-category of simplicial
$k$-modules is equivalent to the $1$-category of non-negatively graded chain
complexes of $k$-modules, and one can see that the homotopy groups
$\pi_i(R_\sbt)$ agree with homology groups $H_i(NR_\sbt)$ of the normalized
chain complex $NR_\sbt$ naturally associated to $R_\sbt$
\cite[Thm.~8.3.8]{Wei94}. In particular, the homotopy groups $\pi_i(R_\sbt)$ of
a simplicial $k$-algebra are $k$-modules and $\pi_0(R_\sbt)$ is a $k$-algebra.

\subsection{\(\infty\)-categories}

For a full treatment of $\infty$-categories we refer the reader to
\cite{Lur09}. All references to $\infty$-categories are to be taken as
$(\infty,1)$-categories. There are several possible formalizations of
$\infty$-categories, including quasi-categories (our preferred model),
topological categories, and simplicial categories. Our discussion will not
depend on the model used. For objects $X$ and $Y$ in an $\infty$-category
$\catC$, we denote by $\Maps_\catC(X,Y)$ the space of morphism from $X$ to $Y$.
We think of it as an $\infty$-groupoid and while its concrete realization might
depend on the model used, its homotopy type is unambiguous. We remind the
reader that most categorical constructions generalize to the
$\infty$-categorical setting: we have notions of equivalences between
$\infty$-categories, the opposite of an $\infty$-category, functors between
$\infty$-categories, $\infty$-categories of functors, subcategories, initial
and final objects, limits and colimits, adjunctions, and Kan extensions. Every
$1$-category can be thought of as an $\infty$-category with discrete mapping
spaces.

It is worth remarking that when the involved $\infty$-categories are not
discrete, the $\infty$-categorical notions of limits and colimits resemble
homotopy limits and colimits. For example, their relation to mapping spaces is
as follows:
\[
    \Maps_\catC(X, \invlim_i Y_i)
    \cong
    \holim_i \Maps_\catC(X, Y_i),
    \quad
    \Maps_\catC(\colim_i X_i, Y)
    \cong
    \holim_i \Maps_\catC(X_i, Y).
\]
For this reason, we sometimes use the symbols $\holim X_i$ or $\hocolim X_i$
for $\infty$-categorical limits and colimits. Recall that geometric
realizations of simplicial sets can be constructed as homotopy colimits.
Because of the above remarks, when $X_\sbt \colon \Delta^\op \to \catC$ is a
simplicial object in an $\infty$-category $\catC$, the colimit $\colim_n X_n$,
if it exists, is called the ($\infty$-categorical) geometric realization of
$X_\sbt$ in $\catC$ \cite[Not.~6.1.2.12]{Lur09}.

\subsection{Animation of categories}

Consider a projectively generated $1$-category $\catC$, with $\catC_0$ a choice
of compact projective generators. Our main example will be $\catC = \Alg_k$ and
the choice $\catC_0 = \Poly_k^\ft$. As discussed above, $\catC$ can be
reconstructed up to equivalence as the $1$-categorical free completion of
$\catC_0$ under sifted colimits: $\catC \cong \Sind(\catC_0)$. The
$\infty$-categorical version of this construction was carefully developed in
\cite[Sec.~5.5]{Lur09} and was later popularized under the name
\emph{animation} in \cite[Sec.~5.1]{CS20}. Concretely, it is shown in
\cite[Sec.~5.1.4]{CS20}, following \cite[Sec.~5.5.8]{Lur09}, that there is an
$\infty$-category $\aSind(\catC_0)$ satisfying the following universal
property: if $\catD$ is any $\infty$-category admitting all sifted colimits,
then the natural restriction functor
\begin{equation}
\label{eq:sind-univ-prop}
    \Fun_\sif(\aSind(\catC_0), \catD)  \to \Fun(\catC_0, \catD)
    \qquad
    F \mapsto F|_{\catC_0}
\end{equation}
is an equivalence of $\infty$-categories. As before, $\Fun_\sif$ denotes the
full subcategory of $\Fun$ spanned by the functors that preserve sifted
colimits. A small difference appears in the $\infty$-categorical setting: a
functor $\aSind(\catC_0) \to \catD$ preserves sifted colimits if and only it
preserves filtered colimits and $\infty$-categorical geometric realizations.
The $\infty$-category $\aSind(\catC_0)$ is uniquely determined up to
equivalence, and is independent of the choice of $\catC_0$. Following
\cite[Sec.~5.1.4]{CS20} we use the notation $\Ani(\catC) =
\aSind(\catC_0)$ and call this $\infty$-category the \emph{animation} of $\catC$.

As in the $1$-categorical setting, $\Ani(\catC)$ admits an alternative
description in terms of presheaves.
As explained in \cite[Sec.~5.5]{Lur09}, there is an $\infty$-category
of $\infty$-presheaves on $\catC_0$, that is, an $\infty$-category of functors
from $\catC_0^\op$ to the $\infty$-category of spaces. We denote it by
$\aPsh(\catC_0)$, and we have the following equivalence:
\[
    \Ani(\catC) =
    \aSind(\catC_0) \cong \aPsh_\finprod(\catC_0)
    \hookrightarrow
    \aPsh(\catC_0),
\]
where $\aPsh_\finprod$ denotes the sub-category of $\aPsh$ spanned by the
functors that preserve finite products.
We get the following diagram:
\[
    \catC_0
    \hookrightarrow
    \Ani(\catC)
    \hookrightarrow
    \aPsh(\catC_0).
\]
The objects of $\catC_0$ remain strongly finitely presentable in $\Ani(\catC)$,
in the sense that when $P$ is an object of $\catC_0$, the functor
$\Maps_{\Ani(\catC)}(P, \Arg)$ preserves $\infty$-categorical sifted colimits.

Using that $\catC \cong \Sind(\catC_0)$ and $\Ani(\catC) = \aSind(\catC_0)$,
one can identify $\catC$ with $\tau_{\le 0}(\Ani(\catC))$, the subcategory
spanned by the $0$-truncated objects in $\Ani(\catC)$, see
\cite[Def.~5.5.6.1, Rem.~5.5.8.26]{Lur09}.
Moreover, there is a $0$-truncation functor $\pi_0 \colon \Ani(\catC) \to
\catC$, see \cite[Prop.~5.5.6.18]{Lur09},
which is seen to be left adjoint to the inclusion
\[
    \catC \hookrightarrow \Ani(\catC).
\]
As a consequence $\catC$ is a full subcategory of $\Ani(\catC)$,
meaning that we have
\[
    \Maps_{\Ani(\catC)}(X, Y)
    \cong
    \Hom_{\catC}(X, Y)
    \qquad
    \text{when $X,Y$ are in $\catC$.}
\]
We emphasize that this says that the mapping spaces on the left have the
homotopy type of the discrete hom-sets on the right.

In \cite[Sec.~5.5.9]{Lur09} there is a third description for $\Ani(\catC)$. It
can be obtained from the $1$-category of simplicial objects in $\catC$ by
inverting weak equivalences with respect to an appropriate model structure. The
objects of $\Ani(\catC)$ can be represented by simplicial objects in $\catC$ and
the inclusion $\catC \hookrightarrow \Ani(\catC$) identifies the objects of
$\catC$ with the discrete (i.e., $0$-truncated) objects of $\Ani(\catC)$. If an
object $\cX$ of $\Ani(\catC)$ is written as a simplicial object $\cX = X_\sbt$
in $\catC$, it can also be thought as simplicial object in $\Ani(\catC)$, and
in this case the corresponding $\infty$-categorical geometric realization in
$\Ani(\catC)$ recovers the original object: $\cX = \hocolim_n X_n$.

\subsection{Anima and animated algebras}

In special cases, the animation of a $1$-category is familiar. For example, the
animation of the $1$-category of sets $\Set$ is precisely the $\infty$-category
of spaces in the sense of Lurie. Following \cite[Sec.\ 5]{CS20}, objects of
this $\infty$-category will be called \emph{anima}, and we use the notation
\[
\Anima = \Ani(\Set).
\]
In this case $\Set^\sfp = \FinSet$ is the $1$-category of finite sets, and we
have identifications
$\Anima \cong \aSind(\FinSet) \cong \aPsh_\finprod(\FinSet)$. Objects of
$\Anima$ can be represented by simplicial sets, and $\Anima$ can be obtained from
the $1$-category of simplicial sets $\sSet$ by inverting weak equivalences.

The animation of the $1$-category of $k$-algebras is denoted
\[
    \aAlg_k = \Ani(\Alg_k).
\]
In this case we have two natural choices for a collection of compact projective
generators: either $\Poly_k^\ft$ (finite-type polynomial algebras, our
preferred choice) or $\Alg_k^\sfp$ (retracts of finite-type polynomial
algebras, the choice in \cite{CS20}). These choices give equivalent
$\infty$-categories. Objects of $\aAlg_k$ will be called \emph{animated
$k$-algebras}.

We can look at objects in $\aAlg_k$ from different points of view. First, as
discussed above $\aAlg_k$ is obtained from the $1$-category of simplicial
$k$-algebras $\sAlg_k$ by inverting weak equivalences. We still use the
notation $R_\sbt$ for an animated algebra, with the understanding that it is
morally only considered up to weak equivalence. In particular, we have homotopy
groups $\pi_i(R_\sbt)$ for animated algebras. Throughout, we refer to discrete
animated $k$-algebras as \emph{classical}. This choice is made to emphasize
that animated algebras which have no higher homotopy come from the world of
classical algebraic geometry, but might not correspond to a ``discrete space''
in any geometric sense. Thus the inclusion $\Alg_k \hookrightarrow \aAlg_k$
identifies the (non-animated) $k$-algebras with classical $k$-algebras. Among
the objects of the $1$-category $\sAlg_k$ there is a distinguished class of
\emph{cofibrant simplicial algebras} which includes all simplicial algebras
$P_\sbt$ such that $P_n$ is a polynomial $k$-algebra for all $n$. Any animated
$k$-algebra $R_\sbt$ is weakly equivalent (i.e., equivalent as objects of
$\aAlg_k$) to a cofibrant simplicial algebra $P_\sbt$ (which can be taken to be
of the above special form). Such equivalences $P_\sbt \to R_\sbt$ are called
\emph{cofibrant replacements}. These replacements can be used to compute
mapping spaces. Specifically, given a second animated $k$-algebra $S_\sbt$ and
a cofibrant replacement $Q_\sbt \to S_\sbt$, then the mapping space
$\Maps_{\aAlg_k}(R_\sbt, S_\sbt)$ is homotopy equivalent to the geometric
realization of the simplicial set $\bHom_{\sAlg_k}(P_\sbt, Q_\sbt)$.

Alternatively, we can use the identification of $\aAlg_k$ with $\aSind(\Poly_k^\ft)$.
From this point of view, every animated $k$-algebra $R_\sbt$ can be expressed
as an $\infty$-categorical sifted colimit
\[
    R_\sbt \cong \hocolim_{d \in \cD} P_d
\]
where each $P_d$ is an object in $\Poly_k^\ft$. Notice that
each $P_d$ is classical, but the colimit need not be. If the colimit were taken
in the $1$-category $\Alg_k$ one would recover the $0$-truncation $\pi_0(R_\sbt)
\cong \colim P_d \in \Alg_k$. An animated algebra $R_\sbt$ can be expressed as
a colimit as above in many different ways, but there is a distinguished choice.
Consider the comma $\infty$-category $(\Poly_k^\ft \downarrow R_\sbt)$, whose
objects are morphisms $\varphi \colon P_\varphi \to R_\sbt$ with $P_\varphi \in
\Poly_k^\ft$. The $\infty$-category $(\Poly_k^\ft \downarrow R_\sbt)$ is sifted and
\[
    R_\sbt \cong \hocolim_{\varphi \in (\Poly_k^\ft \downarrow R_\sbt)} P_\varphi.
\]
This colimit is called the \emph{canonical colimit} for $R_\sbt$ with respect
to $\Poly_k^\ft$, and we say that every animated algebra \emph{is a canonical
colimit} of $\Poly_k^\ft$-objects. Notice that when $R$ is a classical
$k$-algebra the canonical colimit is classical, so the $1$-categorical canonical
colimit gives the same answer: $R \cong \colim_\varphi P_\varphi \cong
\hocolim_\varphi P_\varphi$.

The two ways of looking at animated algebras are related via
$\infty$-categorical geometric realizations. Explicitly, given an equivalence
$P_\sbt \to R_\sbt$ in $\aAlg_k$, for example a cofibrant replacement, we
always have that the $\infty$-categorical geometric realization of $P_\sbt$ as
a simplicial object in $\aAlg_k$ recovers $R_\sbt \cong \hocolim_n P_n$. But
notice that in general we can only guarantee that each $P_n$ is in $\Poly_k$,
and not necessarily in $\Poly_k^\ft$.

\subsection{Animation of functors}

Let $\catC$ and $\catD$ be projectively generated $1$-categories, and consider
a functor $F \colon \catC \to \catD$. By restriction to $\catC_0$ and
composition with $\catD \hookrightarrow \Ani(\catD)$ we get a functor $\catC_0
\to \Ani(\catD)$. Since $\Ani(\catC) = \aSind(\catC_0)$, the
universal property expressed by \cref{eq:sind-univ-prop} gives a functor
\[
    \Ani(F) \colon \Ani(\catC) \to \Ani(\catD)
\]
called the \emph{animation} of $F$. By construction, the restriction of
$\Ani(F)$ to $\catC_0$ (and even to $\catC^\sfp$) agrees with $F$.

From \cref{eq:sind-univ-prop}, $\Ani(F)$ preserves $\infty$-categorical sifted
colimits, and this gives a way of computing its values. Assume $\catC =
\Alg_k$, our main case of interest. Given an animated $k$-algebra $R_\sbt$,
write it as a sifted colimit $R_\sbt \cong \hocolim_n P_n$, where each $P_n \in
\Poly_k^\ft$. For example, consider the canonical colimit with respect to
$\Poly_k^\ft$, or consider a cofibrant replacement with terms in $\Poly_k^\ft$
when it exists. We have
\[
    \Ani(F)(R_\sbt) \cong \hocolim_n F(P_n).
\]
When the colimit is induced by by a cofibrant replacement $P_\sbt \to R_\sbt$
with terms $P_n$ in $\Poly_k^\ft$ we get an equivalence $F(P_\sbt) \to
\Ani(F)(R_\sbt)$, where $F(P_\sbt)$ is the simplicial object obtained by
composing $P_\sbt$ with $F$.

\subsection{Left derived functors}

If $F \colon \catC \to \catD$ preserves $1$-categorical sifted colimits we
often denote $\Ani(F)$ also by $\bL F$ and call it the \emph{left derived
functor} of $F$. In this case, for any object $X_\sbt$ of $\Ani(\catC)$,
written as a sifted colimit $X_\sbt \cong \hocolim_n P_n$ with terms in
$\catC_0$, we have:
\[
    \pi_0 (\bL F(X_\sbt))
    \cong \colim_n F(P_n)
    \cong F(\colim_n P_n)
    \cong F(\pi_0(X_\sbt))
    \qquad
    \text{in $\catC$}.
\]
Furthermore, consider the subcategory $\Ind(\catC_0) \subseteq \catC$ spanned
by the objects in $\catC$ which are filtered limits of objects in $\catC_0$. By
\cite[Prop.~5.5.8.10(6), Cor.~5.5.7.4(1)]{Lur09} the categories $\catC$,
$\catD$ are stable under $\infty$-categorical filtered colimits in the
corresponding animations $\Ani(\catC)$, $\Ani(\catD)$. Given an object $X \in
\Ind(\catC_0)$, write it as a filtered colimit $X = \colim_j Q_j$ in $\catC$
with terms in $\catC_0$. Then we also have $X = \hocolim_j Q_j$ in
$\Ani(\catC)$, and therefore $\bL F(X) = \hocolim_j F(Q_j)$. Since this
homotopy colimit is filtered and the terms are $0$-truncated, the result must
be $0$-truncated, so $\bL F(X) = \colim_j F(Q_j) = F(X)$, using that $F$
preserves sifted (and in particular filtered) colimits. Summarizing:
\[
    \bL F(X) = F(X)
    \qquad
    \text{when $X$ is in $\Ind(\catC_0)$}.
\]
When $\catC = \Set$ we have that $\Ind(\catC_0) = \Ind(\FinSet) = \Set$. We see
that any left derived functor $\bL F \colon \Anima \to \Ani(\catD)$ agrees with $F$
on all of $\Set$.

If $\catC = \Alg_k$ then $\Ind(\catC_0) = \Ind(\Poly_k^\ft)$ contains
$\Poly_k$. In this case, any animated algebra $R_\sbt$ admits a cofibrant
replacement $P_\sbt \to R_\sbt$ where each $P_n \in \Poly_k$. We have that
\[
    \bL F(R_\sbt) \cong
    \hocolim_n \bL F(P_n) \cong \hocolim_n F(P_n),
\]
and we see that we have an equivalence $F(P_\sbt) \to \bL F(R_\sbt)$.
Notice in this situation we can
use less restrictive cofibrant replacements which exist for all animated
algebras, as opposed to the case of general animations $\Ani(F)$ that required
cofibrant replacements by simplicial objects with terms of finite type.

\subsection{Animated modules and derived categories}

\label{sec:animated-modules}

For any ring $R$, the $1$-category of $R$-modules $\Mod_R$ is projectively
generated, and a natural choice of compact projective generators is given by
$\FreeMod_R^\fg$, the $1$-category of finitely-generated free $R$-modules. We
will write $\aMod_R = \Ani(\Mod_R)$, and refer to its objects as \emph{animated
$R$-modules}. As discussed above, $\aMod_R$ is obtained from the $1$-category
of simplicial $R$-modules by inverting weak equivalences. Via the Dold-Kan
correspondence \cite[Ex.~5.1.6(2)]{CS20} we have an identification $\aMod_R
\cong D^{\ge 0}(R) \subseteq D(R)$, where $D(R)$ denotes the
($\infty$-categorical incarnation of the) \emph{derived category} of
$R$-modules, and $D^{\ge 0}(R)$ is the subcategory spanned by the
\emph{connective complexes}, those with vanishing homology groups in negative
degrees. Notice that we use \emph{homological indexing} throughout, so a
complex $M_\sbt$ is connective when $H_i(M_\sbt) = 0 $ for $i<0$, and the
translation functor is given by $H_i(M_\sbt[j]) = H_{i-j}(M_\sbt)$. Also recall
that the homotopy of an animated module can be computed as the homology of the
corresponding complex, that is $\pi_i(M_\sbt) = H_i(M_\sbt)$. See
\cref{sec:der-emb-dim} below for more more comments on connectivity of
complexes.

For right-exact functors $F$ between categories of modules, the above
construction of the animation $\bL F$ agrees with the classical construction of
the left derived functor. Right derived functors require the full derived
$\infty$-category $D(R)$ and are not obtained directly via animation.

For an animated $k$-algebra $R_\sbt$, Lurie describes in
\cite[Sect.~25.2.1]{SAG} the construction of the $\infty$-category of
$R_\sbt$-modules, which we denote $D(R_\sbt)$. Lurie uses the notation
$\Mod_{R_\sbt}$, but we prefer to use to $D(R_\sbt)$ to avoid potential
confusion when $R_\sbt$ is classical. The subcategory of connective objects in
$D(R_\sbt)$ can be constructed via animation as follows.

Write $\AlgMod_k$ for the $1$-category of ``modules on $k$-algebras''. Its
objects are pairs $(R, M)$, where $R \in \Alg_k$ and $M \in \Mod_R$, and its
arrows are given by pairs $(\phi, \psi)$ where $\phi \colon R \to S$ is a
$k$-algebra morphism and $\psi \colon M \to N$ is an $R$-module morphism. The
$1$-category $\AlgMod_k$ is projectively generated and projective compact
generators can be chosen to be of the form $(P, P^m)$, where $P$ is a
finite-type polynomial $k$-algebra, and $P^m$ is a finitely generated free
$P$-module. Consider the animation $\aAlgMod_k = \Ani(\AlgMod_k)$; its objects
can be represented by pairs of the form $(R_\sbt, M_\sbt)$, where $R_\sbt$ is
an animated $k$-algebra, and $M_\sbt$ is a (weak homotopy class of a)
simplicial module over $R_\sbt$. Notice that we have a natural projection
$\AlgMod_k \to \Alg_k$, which induces via animation a functor $\aAlgMod_k \to
\aAlg_k$. We use the notation $\aMod_{R_\sbt}$ for the fiber
$\infty$-category $\aAlgMod_k \times_{\aAlg_k} \{ R_\sbt \}$, and call its
objects \emph{animated $R_\sbt$-modules}. When $R$ is classical, this fiber is
canonically equivalent to $\aMod_R = \Ani(\Mod_R)$. As before, the Dold-Kan
correspondence gives an identification between $\aMod_{R_\sbt}$ and $D^{\ge
0}(R_\sbt)$, the connective part of the derived category $D(R_\sbt)$.

\subsection{Derived tensor products and extension of scalars}

The tensor product of modules gives a functor of three variables: $(R, M, N)
\mapsto M \otimes_R N$. Its animation is called the \emph{derived tensor
product}, and denoted $M_\sbt \otimes_{R_\sbt}^\bL N_\sbt$, where now $R_\sbt$
is an animated $k$-algebra, and $M_\sbt$, $N_\sbt$ are animated
$R_\sbt$-modules. As usual, we have that $M_\sbt \otimes_{R_\sbt}^\bL N_\sbt
\cong F_\sbt \otimes_{P_\sbt} G_\sbt$, where $(P_\sbt, F_\sbt, G_\sbt)$ is a
cofibrant replacement of $(R_\sbt, M_\sbt, N_\sbt)$ and $F_\sbt
\otimes_{P_\sbt} G_\sbt$ denotes the \emph{simplicial tensor product} (i.e.,
the term-wise tensor product).

As in the classical setting of $1$-categorical homological algebra, there are
several possible variations of this construction, all of which give the same
answer. For example, if $R$ and $N$ are classical, one has equivalences $M_\sbt
\otimes_R^\bL N \cong \Ani(\Arg\otimes_R N)(M_\sbt) \cong \bar F_\sbt \otimes_R
N$, where $\bar F_\sbt$ is equivalent to $M_\sbt$ and has terms that are free
$R$-modules (for example $\bar F_\sbt = F_\sbt \otimes_{P_\sbt} R$). More
generally, we have the following \emph{simplicial version of balancing of Tor}:
\[
    M_\sbt \otimes_{R_\sbt}^\bL N_\sbt
    \cong F_\sbt \otimes_{P_\sbt} G_\sbt
    \cong \bar F_\sbt \otimes_{R_\sbt} N_\sbt
    \cong M_\sbt \otimes_{R_\sbt} \bar G_\sbt.
\]

Derived tensor products give the correct notion of \emph{extension of scalars}
between categories of animated modules. More precisely, given a morphism of
animated $k$-algebras $R_\sbt \to S_\sbt$, it can be shown that the above
construction gives a functor $\Arg \otimes_{R_\sbt}^\bL S_\sbt \colon
\aMod_{R_\sbt} \to \aMod_{S_\sbt}$.
With the above notation we have
$
    M_\sbt \otimes_{R_\sbt}^\bL S_\sbt
    \cong F_\sbt \otimes_{P_\sbt} S_\sbt
    \cong \bar F_\sbt \otimes_{R_\sbt} S_\sbt
$.
Notice that when $R_\sbt \to S_\sbt$ is an equivalence, so is $\Arg
\otimes_{R_\sbt}^\bL S_\sbt$. This notion of extension of scalars, together
with the obvious notion of restriction of scalars, give the projection
$\aAlgMod_k \to \aAlg_k$ the structure of a Cartesian and co-Cartesian
fibration.

\subsection{The cotangent complex}

The cotangent complex arises via animation, see for example
\cite[Ex.~2.2]{BMS19}. Explicitly, consider the functor $F \colon \Alg_k \to
\Mod_k$ given by the module of differentials $F(R) = \Omega_{R/k}$, and observe
that $F$ preserves colimits. Its left derived functor $\bL F \colon
\aAlg_k \to \aMod_k \subseteq D(k)$ gives the cotangent complex, $\bL F(R_\sbt)
= \bL_{R_\sbt/k}$. Indeed, by this construction, if $P$ is a classical
polynomial algebra, not necessarily of finite type, then $\bL_{P/k} =
\Omega_{P/k}$ is also classical. Furthermore, given a cofibrant replacement
$P_\sbt \to R_\sbt$ with terms $P_n$ in $\Poly_k$, then $\Omega_{P_\sbt/k} \to
\bL_{R_\sbt/k}$ is also a cofibrant replacement. Since $P_\sbt \to R_\sbt$ is
an equivalence in $\aAlg_k$ and $\Omega_{P_\sbt/k}$ is degree-wise free over
$P_\sbt$, we get
\[
    \bL_{R_\sbt/k}
    \cong
    \Omega_{P_\sbt/k} \otimes_{P_\sbt} R_\sbt
\]
in $\aMod_k$, recovering the classical definition of the cotangent
complex.

A slightly more sophisticated version of this construction
\cite[Sect.~25.3]{SAG} gives $\bL_{R_\sbt/k}$ as an element inside of
$\aMod_{R_\sbt} \subseteq D(R_\sbt)$, instead of just in $\aMod_k$. This
clarifies the presence of the extension of scalars functor
$\Arg\otimes_{P_\sbt}^\bL R_\sbt$ in the classical definition. For this,
consider the functor $\Omega \colon \Alg_k \to \AlgMod_k$ given by $\Omega(R) =
(R, \Omega_{R/k})$. The functor $\Omega$ preserves colimits as it is a left adjoint and
it admits a left derived functor $\bL\Omega \colon \aAlg_k \to \aAlgMod_k$ of
the form $\bL \Omega(R_\sbt) = (R_\sbt, \bL_{R_\sbt/k})$. In particular,
$$\bL_{R_\sbt/k} \in \aAlgMod_k \times_{\aAlg_k} \{ R_\sbt \} = \aMod_{R_\sbt}
\subseteq D(R_\sbt).$$ As above, if $P_\sbt \to R_\sbt$ is a cofibrant
replacement, then $\bL \Omega(R_\sbt)$ is naturally equivalent to $(P_\sbt,
\Omega_{P_\sbt/k})$, but the latter belongs to $\aMod_{P_\sbt}$ instead of the
equivalent $\aMod_{R_\sbt}$.

Given a morphism $S_\sbt \to R_\sbt$ of animated algebras there is also an
associated \emph{relative cotangent complex} $\bL_{R_\sbt/S_\sbt}$. It can be
introduced via animation by considering the arrow category of $\Alg_k$, but it
can also be constructed more directly at the same time as the fundamental
triangle \cite[Sect.~25.3]{SAG}. By functoriality of $\Omega$ one gets a
natural morphism $(S_\sbt, \bL_{S_\sbt/k}) \to (R_\sbt, \bL_{R_\sbt/k})$, and
hence a morphism $\bL_{S_\sbt/k} \otimes_{S_\sbt}^\bL R_\sbt \to
\bL_{R_\sbt/k}$ in $\aMod_{R_\sbt} \subseteq D(R_\sbt)$. We let
$\bL_{R_\sbt/S_\sbt}$ be the mapping cone of this map. Notice that
$\bL_{R_\sbt/S_\sbt}$ is clearly connective, so it gives an object of
$\aMod_{R_\sbt}$. By construction we have a natural triangle
\[
    \bL_{S_\sbt/k} \otimes_{S_\sbt}^\bL R_\sbt
    \lra
    \bL_{R_\sbt/k}
    \lra
    \bL_{R_\sbt/S_\sbt}
    \lra
    (\bL_{S_\sbt/k} \otimes_{S_\sbt}^\bL R_\sbt)[1]
\]
called the \emph{fundamental triangle}. Taking zeroth homotopy we get the
\emph{first fundamental sequence of modules of differentials}
\[
    \Omega_{\pi_0(S_\sbt)/k} \otimes_{\pi_0(S_\sbt)} \pi_0(R_\sbt)
    \lra
    \Omega_{\pi_0(R_\sbt)/k}
    \lra
    \Omega_{\pi_0(R_\sbt)/\pi_0(S_\sbt)}
    \lra
    0.
\]
When $R$ and $S$ are classical and $R = S/I$, the \emph{second fundamental
sequence of modules of differentials} agrees with
\[
    \pi_1(\bL_{R/S}) = I/I^2
    \lra
    \Omega_{S/k} \otimes_{S} R
    \lra
    \Omega_{R/k}
    \lra
    0.
\]

\subsection{Animated derivations}

\label{sec:anim-ders}

Instead of using animation, the cotangent complex can also be defined via a
universal property \cite[Con.~25.3.1.6]{SAG}. This is analogous to the
construction of modules of differentials via derivations, and will play a key role
in our proofs in \Cref{sec:DerJetsCotCplx}. Recall that classically we have the
following natural isomorphism
\[
    \Hom_{\Mod_A}(\Omega_{A/k}, M)
    \cong
    \Der_k(A, M),
\]
and that we can construct the module of derivations on the right in the
following way:
\[
    \Der_k(A, M) = \Hom_{\Alg_k/A}(A, A \oplus \varepsilon M).
\]
Here $A \oplus \varepsilon M$ denotes the \emph{split square-zero extension} of
$A$ via the $A$-module $M$, which has multiplication given by
\[
    (a,\varepsilon m)(a',\varepsilon m')
    =
    (aa',\varepsilon(am'+a'm)).
\]
Viewed as objects of $\Alg_k/A$ one takes the identity for the map $A \to A$
and the first projection for the map $A\oplus\varepsilon M \to A$. The slice
category construction encodes the fact that derivations can be identified with
sections of the natural projection $A\oplus\varepsilon M\to A$.

For the cotangent complex, we have an ``animated'' version of the above
description \cite[Con.~25.3.1.6]{SAG}. Namely:
\[
    \Maps_{\Mod_{A_\sbt}}(\bL_{A_\sbt/k},M_\sbt)
    \cong
    \aDer_{k}(A_\sbt,M_\sbt)
    =
    \Maps_{\aAlg_{k}/A_\sbt}(
        A_\sbt,
        A_\sbt\oplus\varepsilon M_\sbt
    ).
\]
The equality on the right can be taken to be the definition of the space of
\emph{animated derivations} appearing in the middle. The animated
$A_\sbt$-module $A_\sbt\oplus\varepsilon M_\sbt$ is given the structure of
an animated $k$-algebra by animating the split square-zero extension functor
$(A,M) \mapsto A \oplus\varepsilon M$.

We choose to write $\aDer(\Arg,\Arg)$ instead of simply $\Der(\Arg,\Arg)$ to
emphasize the fact that this is a mapping space calculated in the animated
setting, and to avoid potential confusion. For example, for a classical algebra
$A$ the cotangent complex $\bL_{A/k}$ corepresents the functor
$\aDer_{k}(A,\Arg)$, whereas the module of differentials $\Omega_{A/k}$
corepresents the functor $\Der_{k}(A,\Arg)$.

\subsection{Derived quotients, Koszul complexes, and regular sequences}

\label{sec:intro-koszul-regular}

Let $R_\sbt$ be an animated $k$-algebra. By an \emph{element} $a \in R_\sbt$ we
mean a map $k[x] \to R_\sbt$ in $\aAlg_k$, which can be identified with an
element of the $k$-algebra $\pi_0(R_\sbt)$. Similarly, a \emph{sequence of
elements} $a_1, \ldots, a_c \in R_\sbt$ corresponds to a map
$k[x_1,\ldots,x_c] \to R_\sbt$.
Given such a sequence $a_1, \ldots, a_c \in R_\sbt$, we define the
corresponding \emph{derived quotient} via
\[
    R_\sbt \sslash (a_1, \ldots, a_c)
    =
    R_\sbt \otimes_{k[x_1, \ldots, x_c]}^\bL k,
\]
where $k[x_1, \ldots, x_c] \to k$ is the map which sends each $x_i$ to zero.
Notice that the underlying classical algebra is the standard quotient:
\[
    \pi_0(R_\sbt \sslash (a_1, \ldots, a_c))
    =
    \pi_0(R_\sbt) \otimes_{k[x_1, \ldots, x_c]} k
    \cong
    \pi_0(R_\sbt) / (a_1, \ldots, a_c).
\]

When $R$ is classical, the higher homotopy groups of derived quotients are
related to Koszul complexes. More precisely, we can compute homotopy by
forgetting the algebra structure and thinking of $R \sslash (a_1, \ldots, a_c)$
as an object in $\aMod_R \subseteq D(R)$. When $c=1$, a free resolution of $k =
k[x]/(x)$ as a $k[x]$-module is the two-term complex $0 \to k[x] \xra{\cdot x}
k[x] \to 0$, so $R\sslash(a)$ is quasi-isomorphic to the complex $K_\sbt(a)$
given by $0 \to R \xra{\cdot a} R \to 0$. In general $R \sslash (a_1, \ldots, a_c)$
is quasi-isomorphic to $K_\sbt(a_1, \ldots, a_c) = K_\sbt(a_1) \otimes_R \cdots
\otimes_R K_\sbt(a_c)$, which is known as the \emph{Koszul complex} of the
sequence $(a_1, \ldots, a_c)$. Summarizing:
\[
    \pi_i(R \sslash (a_1, \ldots, a_c))
    =
    H_i(K_\sbt(a_1, \ldots, a_c)).
\]

Recall that a sequence $a_1, \ldots, a_c \in R$  is said to be \emph{Koszul
regular} if the corresponding Koszul complex $K_\sbt(a_1, \ldots, a_c)$ has
vanishing homology in degrees $i > 0$. In other words, if the natural map $R
\sslash (a_1, \ldots, a_c) \to R / (a_1, \ldots, a_c)$ is an equivalence.

When computing with arc spaces we will need to consider quotients by infinitely
generated ideals, and it will be convenient to have notation in place for this
type of derived quotients. We identify a \emph{sequence of elements $\{a_i
\,|\, i\in I\}$ in $R_\sbt$ indexed by a set $I$} with a $k$-algebra map
$k[x_i \,|\, i\in I] \to R_\sbt$ and write
\[
    R_\sbt \sslash (a_i \,|\, i \in I)
    =
    R_\sbt \otimes_{k[x_i \,|\, i \in I]}^\bL k.
\]
Notice that Koszul complexes are not defined in this generality.

\subsection{Derived schemes}

Following \cite[Ch.~25]{SAG} and \cite{Toe14}, our approach to derived schemes
is based on animated algebras. Concretely, a \emph{derived $k$-scheme} is a
locally ringed space $(X, \cO_X)$ consisting of a topological space $X$ and a
sheaf of animated $k$-algebras $\cO_X$ such that $(X, \pi_0(\cO_X))$ is a
$k$-scheme and $\pi_i(\cO_X)$ is a quasi-coherent sheaf of
$\pi_0(\cO_X)$-modules for each $i \geq 0$.

We denote by $\dSch_k$ the $\infty$-category of derived $k$-schemes. The
\emph{truncation} of a derived scheme $(X, \cO_X)$ is the classical scheme $(X,
\pi_0(\cO_X))$. As it is common, we will often suppress $\cO_X$ from the
notation, and simply denote derived schemes by $X$. When doing this, we will
write $\pi_0(X)$ for the truncation of $X$.

The category of classical $k$-schemes sits inside $\dSch_k$ in the natural way,
\[
    \Sch_k \hookrightarrow \dSch_k,
\]
with truncation giving a right adjoint \cite[pg. 31]{Toe14}. A derived
$k$-scheme $X$ is \emph{affine} if $\pi_0(X)$ is, and the $\infty$-category of
derived affine $k$-schemes $\dAff_k$ may be identified with the opposite of the
$\infty$-category of animated $k$-algebras \cite[pg. 32]{Toe14}:
\[
    \dAff_k = \aAlg_k^\op.
\]
For $R_\sbt$ an animated $k$-algebra the corresponding object of $\dAff_k$ is
denoted $\Spec R_\sbt$. The truncation of $X = \Spec R_\sbt$ is $\pi_0(X) =
\Spec \pi_0(R_\sbt)$.

To each derived $k$-scheme $X$ we can associate a derived $\infty$-category
$D(X)$, as well as its connective part $\aMod_{\cO_X} \cong D^{\ge 0}(X)$.
Objects of $\aMod_{\cO_X}$ are given by sheaves of animated modules over
$\cO_X$ with quasi-coherent homotopy groups, or, equivalently, by connective
complexes of sheaves with quasi-coherent homology and an appropriate action by
$\cO_X$. If $X = \Spec R_\sbt$ is affine we get canonical equivalences
$\aMod_{\cO_X}
\cong \aMod_{R_\sbt}$ and $D(X) \cong D(R_\sbt)$. Functorial constructions on
schemes with good gluing properties can often be animated without surprises.
For example, we can define by gluing the \emph{cotangent complex} $\bL_{X/k}$
for any derived $k$-scheme $X$, and we get an animated $\cO_X$-module
$\bL_{X/k} \in \aMod_{\cO_X} \subseteq D(X)$. If $U = \Spec R_\sbt \subseteq X$
is an open affine, then the restriction of $\bL_{X/k}$ to $U$ is given by
$\bL_{R_\sbt/k}$. For a morphism $f \colon X \to Y$ of derived schemes we have
the \emph{relative cotangent complex} $\bL_{X/Y}$, which fits in the
corresponding fundamental triangle:
\[
    \bL_{Y/k} \otimes_{\cO_Y}^\bL \cO_X
    \lra
    \bL_{X/k}
    \lra
    \bL_{X/Y}
    \lra
    (\bL_{Y/k} \times_{\cO_Y}^\bL \cO_X)[1]
\]

\subsection{Smoothness, quasi-smoothness, and local complete intersections}

\label{sec:smooth-quasi-smooth-lci}

There are natural notions of \emph{smooth}, \emph{étale}, and \emph{unramified}
morphisms between derived schemes, as well as the usual ``formal variations'';
see \cite[Sec.~3.4]{DAG} and \cite[Sec.~1.2.6, 1.2.7, 2.2.2]{TV08} for details.
For our purposes, it will be enough to note that these notions can be
characterized in terms of cotangent complexes, as follows. Let $f \colon X \to
Y$ be a map locally of finite presentation between derived schemes. Then $f$ is
unramified if $\bL_{X/Y} = 0$. It is étale if $\bL_{Y/k} \otimes^\bL_{\cO_Y}
\cO_X \to \bL_{X/k}$ is an equivalence. And it is smooth if $\bL_{X/Y}$ is
equivalent to a locally free $\cO_X$-module of finite rank. Over a classical
base ring $k$, smooth derived $k$-schemes are the same as classical smooth
$k$-schemes.

In the derived setting we also have a weaker notion: a map $f \colon X \to Y$
of derived schemes is \emph{quasi-smooth} if it is locally of finite
presentation and $\bL_{X/Y}$ is of Tor amplitude at most~$1$; see
\cite[Def.~3.4.15]{DAG} and \cite[2.3.13]{KR25}. A quasi-smooth map always admits a
factorization
\[
    X \xra{i} Z \xra{p} Y,
\]
where $i$ is a quasi-smooth closed immersion and $p$ is smooth. In the derived
setting \emph{closed immersion} just means that the underlying classical map of
schemes $\pi_0(X) \to \pi_0(Z)$ is a closed immersion. As explained in detail
in \cite[Sec.~2]{KR25}, quasi-smooth closed immersions are given affine-locally
by derived quotients $\Spec(A_\sbt\sslash(a_1, \ldots, a_n)) \to
\Spec(A_\sbt)$, and can be characterized as the closed immersions for which the
shifted cotangent complex $\bL_{X/Z}[-1]$ is a locally free $\cO_X$-module of
finite rank.
%

When $X$ and $Y$ are classical, $Z$ can also be taken to be classical, and the
closed immersion $i \colon X \hookrightarrow Z$ is affine-locally of the form
$\Spec(A/I) \to \Spec(A)$ where $A$ is smooth over $\cO_Y$ and the ideal $I
\subset A$ is generated by a (Koszul) regular sequence $a_1, \ldots, a_c \in
A$.
%
%
In other words, a quasi-smooth map between classical schemes is the same as a
\emph{local complete intersection map of schemes}. And a quasi-smooth closed
immersion between classical schemes is the same as a \emph{(Koszul) regular
immersion}.

Specializing to the case $Y = \Spec k$, we get the notion of a
\emph{quasi-smooth $k$-scheme}. Classical quasi-smooth $k$-schemes are the same
as classical local complete intersection $k$-schemes.

\subsection{Embedding dimension and the cotangent complex}

\label{sec:der-emb-dim}

We finish our preliminaries on derived algebraic geometry by discussing the
connection between embedding dimension and cotangent complexes. The material
here is likely well-known to experts, but since we could not find a suitable
reference in the literature, we include full details.


\begin{dff}
\label{def:embdim}
Let $X$ be a derived scheme, consider its classical truncation $X_\cl =
\pi_0(X)$, and a point $x \in X_\cl$ with maximal ideal $\fm_x \subset
\cO_{X_\cl,x}$ and residue field $k_x = \cO_{X_\cl,x}/\fm_x$. The quantity
\[
    \embdim(\cO_{X,x}) :=
    \embdim(\cO_{X_\cl,x}) :=
    \dim_{k_x}(\fm_x/\fm_x^2)
    \in \NN \cup \{\infty\}
\]
is called the \emph{embedding dimension} of $X$ at the point $x$.
\end{dff}

It is a standard result that the embedding dimension on classical schemes can
be computed using cotangent complexes. This fact extends to derived schemes,
as the next result shows.

\begin{prop}
\label{prop:embdim-via-pi1}
In the situation of \cref{def:embdim}, the point $x$ gives a map of derived
schemes $\Spec(k_x) \to X$. We use the shorthand $\bL_{k_x/X} =
\bL_{\Spec(k_x)/X}$. We have that
\[
    \pi_1(\bL_{k_x/X}) \cong \fm_x/\fm_x^2,
\qquad\text{and hence}\qquad
    \embdim(\cO_{X,x})
    =
    \dim_{k_x}(\pi_1(\bL_{k_x/X})).
\]
\end{prop}

Before presenting the proof, we need to recall some terminology and a result of
Lurie. Fix an animated $k$-algebra $A_\sbt$ and consider its derived
$\infty$-category $D(A_\sbt)$. Recall that we use homological indexing, that
homotopy and homology agree, $\pi_i(M_\sbt) = H_i(M_\sbt)$, and that the
translation functor satisfies $\pi_i(M_\sbt[j]) = \pi_{i-j}(M_\sbt)$. Given a
morphism $\varphi \colon M_\sbt \to N_\sbt$ in $D(A_\sbt)$ we can always form
the mapping cone $\cone(\varphi)$. Following Lurie's terminology we also call
it the (homotopy) \emph{cofiber} of $\varphi$, written $\cofib(\varphi) =
\cone(\varphi)$. The translate $\fib(\varphi) = \cone(\varphi)[-1]$ is called
the (homotopy) \emph{fiber} of $\varphi$. They fit in the following
distinguished triangles:
\[
    M_\sbt
    \xra{\ \varphi\ }
    N_\sbt
    \to
    \cofib(\varphi)
    \to
    M_\sbt[1]
\qquad
    \text{and}
\qquad
    \fib(\varphi)
    \to
    M_\sbt
    \xra{\ \varphi\ }
    N_\sbt
    \to
    \fib(\varphi)[1].
\]
An object $M_\sbt$ of $D(A_\sbt)$ is said to be \emph{$i$-connective} if
$\pi_j(M_\sbt) = 0$ for all $j<i$, and a map $\varphi$ is \emph{$i$-connective}
if its fiber $\fib(\varphi)$ is $i$-connective (equivalently, if its cofiber
$\cofib(\varphi)$ is $(i+1)$-connective). For maps, this implies that the
induced maps on homotopy $\pi_j(\varphi)$ are isomorphisms for $j < i$. We use
the term \emph{connective} to mean $0$-connective. Via the Dold-Kan
correspondence, the subcategory $D(A_\sbt)_{\ge 0}$ of connective complexes is
naturally equivalent to the category of animated $A_\sbt$-modules
$\aMod_{A_\sbt}$. When $i \ge 0$, an $i$-connective object $M_\sbt$ admits a
cofibrant replacement $P_\sbt \to M_\sbt$ where the simplicial module $P_\sbt$
satisfies $P_j = 0$ for $j < i$. In particular, if $K$ is a classical
$A_\sbt$-algebra and $M_\sbt$ is an $i$-connective animated $A_\sbt$-module,
then $M_\sbt \otimes_{A_\sbt}^\bL K$ is an $i$-connective animated $K$-module.
A morphism $\varphi \colon A_\sbt \to B_\sbt$ of animated $k$-algebras induces,
after forgetting the algebra structure, a morphism in $\aMod_{A_\sbt}$. As
such, we can consider its fiber $\fib(\varphi)$ and cofiber $\cofib(\varphi)$,
and we have a notion of $i$-connective morphism of animated $k$-algebras.
Notice that the fiber and cofiber are only animated $A_\sbt$-modules, not
algebras. Every morphism of animated $k$-algebras is $(-1)$-connective, and it
is connective if it is surjective on $\pi_0$. With this terminology in place,
we can state the following fundamental result of Lurie.

\begin{thm}[{\cite[Prop.~25.3.6.1]{SAG}}]
\label{thm:lurie-hurewicz}
Let $\varphi \colon A_\sbt \to B_\sbt$ be a morphism of animated $k$-algebras.
There exists a natural map $\varepsilon_\varphi \colon \cofib(\varphi)
\otimes^\bL_{A_\sbt} B_\sbt \to \bL_{B_\sbt/A_\sbt}$. If $\varphi$ is
connective, then $\varepsilon_\varphi$ is $2$-connective. If $\varphi$ is
$i$-connective for $i>0$, then $\varepsilon_\varphi$ is $(i+3)$-connective.
\end{thm}

As an immediate consequence, we obtain the following well-known result
\SPcite{08RA}.

\begin{prop}
\label{prop:embdim-via-pi1-classical}
If $B$ is a classical $k$-algebra and $\psi \colon B \to K$ is a surjection
with kernel $\fm$, then $\pi_1(\bL_{K/B}) \cong \fm/\fm^2$.
\end{prop}

\begin{proof}
The distinguished triangle $\fib(\psi)\to B \to K \to \fib(\psi)[1]$ in $D(A)$
gives that the fiber is $\fib(\psi) = \fm$, so $\psi$ is connective and
$\cofib(\psi) = \fm[1]$. From \cref{thm:lurie-hurewicz} we see that
$\varepsilon_\psi$ is $2$-connective and therefore $ \pi_1(\bL_{K/B}) \cong
\pi_1(\fm[1] \otimes_B^\bL K) = \pi_0(\fm \otimes_B^\bL K) = \fm/\fm^2 $.
\end{proof}

\begin{lem}
\label{lem:truncation-1-connective}
Let $A_\sbt$ be an animated $k$-algebra. Its truncation map $\varphi \colon
A_\sbt \to \pi_0(A_\sbt)$ is $1$-connective and $\pi_i(A_\sbt) =
\pi_i(\fib(\varphi)) = \pi_{i+1}(\cofib(\varphi))$ for $i > 0$.
\end{lem}

\begin{proof}
Immediate from the triangle $\fib(\varphi) \to A_\sbt \to \pi_0(A_\sbt) \to
\fib(\varphi)[1]$ and the fact that the classical $k$-algebra $\pi_0(A_\sbt)$
has vanishing higher homotopy.
\end{proof}

\begin{proof}[Proof of \cref{prop:embdim-via-pi1}]
Work affine-locally and let $A_\sbt = \cO_{X,x}$, $B = \pi_0(A_\sbt)$, and $K =
k_x$. From \cref{lem:truncation-1-connective} we know that $\varphi \colon
A_\sbt \to B$ is $1$-connective. From \cref{thm:lurie-hurewicz} we get a
$4$-connective morphism $\cofib(\varphi)\otimes_{A_\sbt}^\bL B \to
\bL_{B/A_\sbt}$. After tensoring with the classical algebra $K$ we get a
$4$-connective morphism $\cofib(\varphi) \otimes_{A_\sbt}^\bL K \to
\bL_{B/A_\sbt} \otimes_B^\bL K$. Again using \cref{lem:truncation-1-connective}
we see that $\cofib(\varphi)$ is $2$-connective, hence $\cofib(\varphi)
\otimes_{A_\sbt}^\bL K$ is $2$-connective, and therefore $\bL_{B/A_\sbt}
\otimes_B^\bL K$ is $2$-connective. That is, $\pi_i(\bL_{B/A_\sbt} \otimes_B^\bL
K) = 0$ for $i = 0, 1$. From the triangle $\bL_{B/A_\sbt} \otimes_B^\bL K \to
\bL_{K/A_\sbt} \to \bL_{K/B}$ we see that $\pi_1(\bL_{K/A_\sbt}) =
\pi_1(\bL_{K/B})$, and we conclude by \cref{prop:embdim-via-pi1-classical}.
\end{proof}

\section{Derived jet and arc spaces}

\label{sec:DerJets}

In this section we construct derived jet and arc spaces. There are two
competing methods. The first is to start as in the classical case with the
functor of points. The second is to extend the classical functorial
construction as a derived functor via animation. We will show that these two
approaches yield the same answer.

\subsection{Affine case}

\label{sec:AffCase}

We start with the affine setting. The first approach to derived jets and arcs,
perhaps the most natural, is to consider points with values on the derived
version of truncated polynomials and power series, namely $A_\sbt[t]/(t^{n+1})$
for jets and $A_\sbt\llbracket t \rrbracket$ for arcs.

\begin{dff}
\label{dff:Wm}
Consider the functors $(\Arg)[t]/(t^{n+1})$ and $(\Arg)\llbracket t \rrbracket$
on $\Alg_k$. Degree-wise application induces functors on simplicial
$k$-algebras which preserve weak equivalences. We will use the same notation
for the corresponding functors on $\aAlg_k$. That is, for an animated algebra
$A_\sbt$ we will write $A_\sbt[t]/(t^{n+1})$ and $A_\sbt\llbracket t
\rrbracket$ for the values of these functors.
\end{dff}

\begin{rmk}
\label{rmk:animated-power-series-functors}
When $n$ is finite we have $$A_\sbt[t]/(t^{n+1}) = A_\sbt \otimes_k
k[t]/(t^{n+1}) = A_\sbt \otimes_k^\bL k[t]/(t^{n+1}),$$ since $k[t]/(t^{n+1})$
is a free $k$-algebra. In particular, the functor $(\Arg)[t]/(t^{n+1})$ on
$\aAlg_k$ is the animation (in fact the left derived functor) of the
corresponding functor on $\Alg_k$.

When $n = \infty$, the situation is more subtle and the natural map $A
\otimes_k k\llbracket t \rrbracket \to A\llbracket t \rrbracket$ can fail to be
an isomorphism. As the functor $(\Arg)\llbracket t \rrbracket$ does not
preserve sifted (or even filtered) colimits, we do not consider its animation.
\end{rmk}

\begin{dff}
In the classical setting, for a classical $k$-algebra $R$, we define the
classical \emph{$n$-th jet functors} $\Jet_n^{R} \colon \Alg_k \to \Set$,
where $n$ is a non-negative integer, and the classical \emph{arc functor}
$\Jet_\infty^{R} \colon \Alg_k \to \Set$ by
\[
    A \mapsto \Hom_{\Alg_k}(R, A[t]/(t^{n+1}))
    \qquad\text{and}\qquad
    A \mapsto \Hom_{\Alg_k}(R, A\llbracket t \rrbracket).
\]
In the animated setting, for an animated $k$-algebra $R_\sbt$, we define the
\emph{derived $n$-th jet functors} $\dJet_n^{R_\sbt} \colon \aAlg_k \to
\Anima$ and the \emph{derived arc functor} $\dJet_\infty^{R_\sbt} \colon
\aAlg_k \to \Anima$ by
\[
    A_\sbt \mapsto \Maps_{\aAlg_k}(R_\sbt, A_\sbt[t]/(t^{n+1}))
    \qquad\text{and}\qquad
    A_\sbt \mapsto \Maps_{\aAlg_k}(R_\sbt, A_\sbt\llbracket t \rrbracket).
\]
We abuse notation and use the same symbols to denote the corresponding
contravariant functors on affine schemes and derived affine schemes. That is,
for $n \in \NN \cup \{\infty\}$ we have:
\begin{align*}
    \Jet_n^X = \Jet_n^R &\colon \Aff_k \to \Set,
    && \text{if $X = \Spec R$},\\
    \dJet_n^{\cX} = \dJet_n^{R_\sbt} &\colon \dAff_k \to \Anima,
    && \text{if $\cX = \Spec R_\sbt$}.
\end{align*}

\end{dff}

A second candidate is simply the animation of the functorial jet/arc space
construction. Recall that in the classic setting for a $k$-algebra $R$ and $n
\in \NN \cup \{\infty\}$, the functor $\Jet_n^R$ is corepresentable by an
algebra $J_n(R)$. The construction of $J_n(R)$ is functorial and admits an
explicit description involving algebras of Hasse-Schmidt differentials
\cite{Voj07}.

\begin{dff}
Let $n \in \NN \cup \{\infty\}$. We denote by $\bL J_n \colon \aAlg_k \to
\aAlg_k$ the animation of the functor $J_n \colon \Alg_k \to \Alg_k$. As above, we
use the same notation for the corresponding functor $\bL J_n \colon \dAff_k \to
\dAff_k$ on affine derived schemes. When $n < \infty$, we call $\bL J_n$ the
\emph{affine derived $n$-jet space functor}, and $\bL J_\infty$ the
\emph{affine derived arc space functor}.
\end{dff}

In the classical setting, the functors $J_n$ all preserve $1$-categorical
colimits (they are left adjoints), justifying our notation $\bL J_n$ instead of
$\Ani(J_n)$. As for any left derived functor, we have functorial isomorphisms
\[
    \pi_0(\bL J_n(R_\sbt)) \cong J_n( \pi_0(R_\sbt))
    \qquad
    \text{for any $n \in \NN \cup \{\infty\}$}.
\]
In particular, for a classical algebra $A$, the space $\Maps_{\aAlg_k}( \bL
J_n(R_\sbt), A)$ is homotopy equivalent to the discrete set
$\Hom_{\Alg_k}(J_n(\pi_0(R_\sbt)),A)$. To compute $\bL J_n (R_\sbt)$ for an
animated algebra $R_\sbt$, one considers a cofibrant replacement $P_\sbt \to
R_\sbt$ where all the terms $P_m$ are in $\Poly_k$, not necessarily of finite
type, so these replacements always exist. Then $J_n(P_\sbt) \to \bL
J_n(R_\sbt)$ is a cofibrant replacement. Here we have used that the classical
jet/arc space functors satisfy $J_n(\Poly_k) \subseteq \Poly_k$ (see
\cite{Voj07}). If replacements $P_\sbt \to R_\sbt$ have been chosen
functorially for all animated algebras $R_\sbt$, for example via the standard
resolution, we can use them to construct $\bL J_n$ as:
\[
    \bL J_n(R_\sbt) = J_n(P_\sbt) \otimes_{P_\sbt} R_\sbt.
\]
The extension of scalars, while in principle unnecessary, is standard and
guarantees that $\bL J_n(R_\sbt)$ is an animated $R_\sbt$-algebra (and not just
weakly equivalent to one). Also, notice that we can assume that we have an
equality $\bL J_0(R_\sbt) = R_\sbt$ and not just an equivalence.

If $n \le m \le \infty$ we have natural transformations $J_n(R) \to J_m(R)$
between the classical jet/arc space functors called \emph{truncation maps}.
They are obtained via adjunction from the functorial truncation maps
$A[t]/(t^{m+1}) \to A[t]/(t^{n+1})$ and $A\llbracket t \rrbracket \to
A[t]/(t^{n+1})$. Also, we have a natural isomorphism:
\[
    J_\infty(R) = \colim_n J_n(R).
\]
From the definition of derived jet/arc space functors, we immediately see that
we have induced \emph{derived truncation maps} $\bL J_n(R_\sbt) \to
\bL J_m(R_\sbt)$, natural for $R_\sbt$ in $\aAlg_k$.

\begin{thm}
\label{thm:ani_arc_hocolim}
The natural map $\hocolim_n \bL J_n(R_\sbt) \to \bL J_\infty(R_\sbt)$ is an
equivalence for any animated algebra $R_\sbt$.
\end{thm}

\begin{proof}
Consider a cofibrant replacement $P_\sbt \to R_\sbt$ with terms $P_m$ in
$\Poly_k$. Then:
\[
    \hocolim_n \bL J_n(R_\sbt)
    \cong \hocolim_n \hocolim_m J_n(P_m)
    \cong \hocolim_m \hocolim_n J_n(P_m).
\]
Since the innermost homotopy colimit in $n$ is filtered, and the terms
$J_n(P_m)$ are classical, this innermost colimit is classical and agrees with the
colimit in $\Alg_k$. We see:
\[
    \hocolim_m \hocolim_n J_n(P_m)
    \cong \hocolim_m \colim_n J_n(P_m)
    \cong \hocolim_m J_\infty(P_m)
    \cong \bL J_\infty(R_\sbt).
    \qedhere
\]
\end{proof}

For an animated $k$-algebra $R_\sbt$ we
can also take the functor of points of the derived jet/arc scheme $h^{\bL J_n
(R_\sbt)} \colon \aAlg_k \to \Anima$, which is given by
\[
    h^{\bL J_n (R_\sbt)}(A_\sbt)
    =
    \Maps_{\aAlg_k}( \bL J_n (R_\sbt), A_\sbt).
\]

For each fixed animated $k$-algebra $R_\sbt$ and $n \in \NN \cup \{\infty\}$
we have constructed two covariant functors $\aAlg_k \to \Anima$, namely $h^{\bL
J_n (R_\sbt)}$ and $\dJet_n^{R_\sbt}$. We will show that these functors
agree. The case of jets is a corollary to the following general argument about
adjunctions and animations.

\begin{thm}
\label{thm:ani_adj}
Let $\catC$, $\catD$ be projectively generated $1$-categories, with $\catC_0$
and $\catD_0$ choices of projective compact generators. Let
$F\colon\catC\to\catD$ and $G\colon\catD\to\catC$ form an adjoint pair $F
\dashv G$, and suppose $F(\catC_0)\subseteq\catD_0$. Then there is an
$\infty$-categorical adjunction $\Ani(F) \dashv \Ani(G)$.
\end{thm}

\begin{proof}
Let $C\in\Ani(\catC)$ and $D\in\Ani(\catD)$. Write $C$ as a sifted colimit
$C=\hocolim_i P_i$, and similarly write $D=\hocolim_j Q_j$, where $P_i
\in\catC_0$ and $Q_j \in\catD_0$. Assume the colimits are canonical, so the
construction below is functorial. Notice that the objects $P_i$ and $Q_j$ are
strongly finitely presentable for all $i$ and $j$. As the animation $\Ani(F)$
preserves sifted colimits and restricts to $F$ on $\catC_0$, we have:
\begin{align*}
    \Maps_{\Ani(\catD)}(\Ani(F)(C),D)
    &\cong
    \Maps_{\Ani(\catD)}(\Ani(F)(\hocolim_{i}P_{i}),\hocolim_{j}Q_{j})
    \\&\cong
    \Maps_{\Ani(\catD)}(\hocolim_{i}F(P_{i}),\hocolim_{j}Q_{j}).
\intertext{
    Now commute the colimit out of the first factor. Since by hypothesis
    $F(P_{i})\in\catD_0$, we see that $F(P_i)$ is strongly finitely
    presentable, so we can also commute the colimit out of the second factor.
    Then:
}
    \Maps_{\Ani(\catD)}(\hocolim_{i}F(P_{i}),\hocolim_{j}Q_{j})
    &\cong
    \holim_{i}\Maps_{\Ani(\catD)}(F(P_{i}),\hocolim_{j}Q_{j})\\
    &\cong
    \holim_{i}\hocolim_{j}\Maps_{\Ani(\catD)}(F(P_{i}),Q_{j}).
\intertext{
    Since $F(P_i)$ and $Q_j$ are in $\catD_0 \subseteq \catD$, and $\catD$ is a
    full subcategory of $\Ani(\catD)$, these mapping spaces are discrete and
    agree with the corresponding hom sets in $\catD$. We get:
}
    \holim_{i}\hocolim_{j}\Maps_{\Ani(\catD)}(F(P_{i}),Q_{j})
    &\cong
    \holim_{i}\hocolim_{j}\Hom_\catD(F(P_{i}),Q_{j}).
\intertext{
    Next we may apply the adjunction between $F$ and $G$:
}
    \holim_{i}\hocolim_{j}\Hom_\catD(F(P_{i}),Q_{j}).
    &\cong
    \holim_{i}\hocolim_{j}\Hom_\catC(P_{i},G(Q_{j})).
\intertext{
    Finally, reversing all the previous steps allows us to conclude the
    adjunction between $\Ani(F)$ and $\Ani(G)$, as desired.
}
    \holim_{i}\hocolim_{j}\Hom_\catC(P_{i},G(Q_{j})).
    &\cong
    \holim_{i}\hocolim_{j}\Maps_{\Ani(\catC)}(P_{i},G(Q_{j}))\\
    &\cong
    \holim_{i}\Maps_{\Ani(\catC)}(P_{i},\hocolim_{j}G(Q_{j}))\\
    &\cong
    \Maps_{\Ani(\catC)}(\hocolim_{i}P_{i},\hocolim_{j}G(Q_{j}))\\
    &\cong
    \Maps_{\Ani(\catC)}(\hocolim_{i}P_{i},\Ani(G)(\hocolim_{j}Q_{j}))\\
    &\cong
    \Maps_{\Ani(\catC)}(C,\Ani(G)(D)).
    \qedhere
\end{align*}
\end{proof}

\begin{thm}
\label{thm:allagree}
Let $n \in \NN \cup \{\infty\}$. For any animated $k$-algebra $R_\sbt$, the
derived jet/arc functor $\dJet_n^{R_\sbt}$ is corepresentable by the derived
jet/arc space $\bL J_n (R_\sbt)$.
\end{thm}

\begin{proof}
Assume first that $n < \infty$. Let $\catC = \catD = \Alg_k$, with $\catC_0 =
\catD_0 = \Poly_k^\ft$. If we apply \cref{thm:ani_adj} with
$F(\Arg)=J_n(\Arg)$ and $G(\Arg)=(\Arg)[t]/(t^{n+1})$ we see that
\begin{align*}
    \dJet_n^{R_\sbt}(A_\sbt)
    &\cong \Maps_{\aAlg_k}(R_\sbt, A_\sbt[t]/(t^{n+1})) \\
    &= \Maps_{\aAlg_k}(R_\sbt, \Ani(G)(A_\sbt)) \\
    &\cong \Maps_{\aAlg_k}(\Ani(F)(R_\sbt), A_\sbt) \\
    &= \Maps_{\aAlg_k}(\bL J_n(R_\sbt), A_\sbt) \\
    &= h^{\bL J_n(R_\sbt)}(A_\sbt).
\end{align*}
To apply \cref{thm:ani_adj} we need to show that $J_n(\Poly_k^\ft) \subseteq
\Poly_k^\ft$. This is well-known:
\[
    J_n(k[x_1, \ldots, x_d])
    =
    k[
        x_1^{(0)},
        x_1^{(1)},
        \ldots
        x_1^{(n)},
        \ldots,
        x_d^{(0)},
        x_d^{(1)},
        \ldots
        x_d^{(n)}
    ],
\]
where $x_{i}^{(q)}$ represents the $q$-th Hasse-Schmidt differential of
$x_i$ (see \cite{Voj07}).

Corepresentability when $n = \infty$ is now formal; it follows from
commutativity of limits with mapping spaces, the corepresentability when $m <
\infty$, and \cref{thm:ani_arc_hocolim}:
\begin{align*}
\Maps_{\aAlg_{k}}(\bL J_\infty(R_\sbt), A_\sbt)
& \cong \Maps_{\aAlg_{k}}(\hocolim_n \bL J_n(R_\sbt), A_\sbt) \\
& \cong \holim_n \Maps_{\aAlg_{k}}(\bL J_n(R_\sbt), A_\sbt)\\
& \cong \holim_n \Maps_{\aAlg_{k}}(R_\sbt, A_\sbt[t]/(t^{n+1}))\\
& \cong \Maps_{\aAlg_{k}}(R_\sbt, \holim_n A_\sbt[t]/(t^{n+1}))\\
& \cong \Maps_{\aAlg_{k}}(R_\sbt, A_\sbt\llbracket t\rrbracket).
\qedhere
\end{align*}
\end{proof}

It is also worth remarking that the functors $\Jet_n^{R_\sbt}$ are in some way
also obtained by animation, as the following
\lcnamecref{thm:ani_jet_functor_points} shows.

\begin{prop}
\label{thm:ani_jet_functor_points}
Let $R$ be a classical $k$-algebra and $n \in \NN \cup \{\infty\}$. Then the
derived jet/arc functor $\dJet_n^R \colon \aAlg_k \to \Anima$ is the animation
of the classical jet/arc functor $\Jet_n^R \colon \Alg_k \to \Set$. Therefore
the ($\infty$-categorical) functor of points of the derived jet/arc space $\bL
J_n(R)$ is the animation of the functor of points of the classical jet/arc
space $J_n(R)$.
\end{prop}

\begin{proof}
By the universal property of animation and the fact that $\bL J_n$ is a left
derived functor, to show that $\dJet_n^R$ agrees with $\Ani(\Jet_n^R)$ we only
need to show that $\dJet_n^R$ preserves $\infty$-categorical sifted colimits
and that it agrees with $\Jet_n^R$ on $\Poly_k$. Since $\dJet_n^R$ is
corepresentable by the animated algebra $\bL J_n(R)$ it preserves all colimits
\cite[Prop.~5.5.2.2]{Lur09}, and in particular all sifted colimits. If $P$ is a
polynomial algebra, we get:
\[
\begin{aligned}
    &\dJet_m^R(P) \\
    &\cong \Maps_{\aAlg_k}(R, P[t]/(t^{n+1})) \\
    &\cong \Hom_{\Alg_k}(R, P[t]/(t^{n+1})) \\
    &\cong \Jet_n^R(P),
\end{aligned}
\qquad\text{and}\qquad
\begin{aligned}
    &\dJet_\infty^R(P) \\
    &\cong \Maps_{\aAlg_k}(R, P\llbracket t \rrbracket) \\
    &\cong \Hom_{\Alg_k}(R, P\llbracket t \rrbracket) \\
    &\cong \Jet_\infty^R(P).
\end{aligned}
\]
Here we have used that $\Alg_k$ is a full subcategory of $\aAlg_k$.
\end{proof}

\begin{rmk}
It would be tempting to consider the animation of the two-variable functor
$\Jet_n \colon \Alg_k^\op \times \Alg_k \to \Set$ given by $(R,A) \mapsto
\Jet_n^R(A)$. But notice that the opposite category $\Alg_k^\op$ is not locally
presentable, so the theory of animation does not apply.
\end{rmk}

\subsection{Étale base change}

In the next subsection we globalize the functors $\bL J_n$ beyond the affine
case, and this requires understanding their behavior with respect to
localization. As in the classical setting, it is more natural to consider
general formally étale maps, instead of just localizations.

For a detailed discussion of formally étale maps of animated algebras we refer
the reader to \cite[Sec.~3.4]{DAG}. Recall first the classical notion: a
morphism $R \to S$ of algebras is said to be formally étale if for every
commutative solid diagram
\[
\begin{tikzcd}
    S \ar[r] \ar[rd, dashed]
    & B
    \\
    R \ar[u] \ar[r]
    & \widetilde B \ar[u]
\end{tikzcd}
\]
where $B = \widetilde B/I$ for a square-zero ideal $I \subseteq \widetilde B$,
there exists a dotted arrow making the diagram commute. More functorially,
consider the functors $h_R, h_S \colon \Alg_k \to \Set$ corepresented by $R,
S$, so the map $R \to S$ gives a natural transformation $h_S \to h_R$. Then the
above universal property can be expressed by saying that the natural map
\[
    h_S(\widetilde B) \to h_S(B) \times_{h_R(B)} h_R(\widetilde B)
\]
is bijective. To extend this to the animated setting, Lurie introduces the
notion of \emph{small extension} $\widetilde B_\sbt \to B_\sbt$ of animated
algebras \cite[Def.~3.3.1]{DAG}, generalizing the notion of square-zero
extensions. Then a morphism $R_\sbt \to S_\sbt$ of animated algebras is
\emph{formally étale} if for any small extension $\widetilde B_\sbt \to B_\sbt$
the induced map
\[
    h_{S_\sbt}(\widetilde B_\sbt)
    \to
    h_{S_\sbt}(B_\sbt)
    \times_{h_{R_\sbt}(B_\sbt)}^\bL
    h_{R_\sbt}(\widetilde B_\sbt)
\]
is an equivalence \cite[Rmk.~3.4.4]{DAG}. Here $h_{R_\sbt}, h_{S_\sbt} \colon
\aAlg_k \to \Anima$ corepresent $R_\sbt$, $S_\sbt$. We will not discuss the
notion of small extension in detail, and will simply note that given a
square-zero extension $A \to A/I$ between classical algebras and an animated
$A/I$-algebra $B_\sbt$, the base change $\widetilde B_\sbt = B_\sbt
\otimes_{A/I}^\bL A \to B_\sbt$ is a small extension (for details see
\cite[Prop.~3.3.3]{DAG} and the discussion thereafter). An alternative way of
producing small extensions is by animating the functor $(A, M) \mapsto A \oplus
\varepsilon M$, where $A$ is a $k$-algebra, $M$ is an $A$-module, and $A \oplus
\varepsilon M$ is the corresponding split square-zero extension; see
\cite[Sec.~5.1.8]{CS20} for a discussion of this approach.

\begin{thm}\label{thm:etalelocalization}
Let $R_\sbt \to S_\sbt$ be a formally étale morphism of animated $k$-algebras.
For $n \in \NN \cup \{\infty\}$, the natural morphism
$
    \bL J_n(R_\sbt) \otimes_{R_\sbt}^\bL S_\sbt
    \to
    \bL J_n(S_\sbt)
$
is an equivalence. If $n < \infty$, the natural morphism
$
    \bL J_n(R_\sbt) \otimes_{\bL J_{n-1}(R_\sbt)}^\bL \bL J_{n-1}(S_\sbt)
    \to
    \bL J_n(S_\sbt)
$
is an equivalence.
\end{thm}

\begin{proof}
Notice that the first statement immediately follows from the second one when $n
< \infty$. For the $n=\infty$ case, it is enough to observe that the left
derived functor $(\Arg)\otimes_{R_\sbt}^\bL S_\sbt$ commutes with colimits and
that $\bL J_\infty(\Arg)$ is a filtered colimit of $\bL J_n(\Arg)$.

For the second statement of the \lcnamecref{thm:etalelocalization}, considering
the corresponding ($\infty$-categorical) functors of points we need to show
that
\[
    h_{\bL J_n(S_\sbt)}(A_\sbt)
    \to
    h_{\bL J_n(R_\sbt)}(A_\sbt)
    \times^\bL_{h_{\bL J_{n-1}(R_\sbt)}(A_\sbt)}
    h_{\bL J_{n-1}(S_\sbt)}(A_\sbt)
\]
is an equivalence for any animated algebra $A_\sbt$. From the functor of points
description, this means that
\[
    h_{S_\sbt}(A_\sbt[t]/(t^{n+1}))
    \to
    h_{R_\sbt}(A_\sbt[t]/(t^{n+1}))
    \times^\bL_{h_{R_\sbt}(A_\sbt[t]/(t^n))}
    h_{S_\sbt}(A_\sbt[t]/(t^n))
\]
needs to be an equivalence. This follows directly from the fact that $R_\sbt
\to S_\sbt$ is formally étale and $A_\sbt[t]/(t^{n+1}) \to A_\sbt[t]/(t^n)$ is
a small extension (in fact a base-change of a square-zero extension).
\end{proof}

\subsection{Global case}

\label{sec:Globcase}

We already have all the ingredients for globalizing the functors $\bL
J_n(\Arg)$ to all derived schemes.

\begin{thm}
\label{thm:globalizejets}
For each derived $k$-scheme $X$, and $n \in \NN \cup \{\infty\}$, there
is a derived $k$-scheme $\bL J_n (X)$ which represents the functor $\dJet_n^X
\colon \aAlg_k \to \Anima$ given by
\[
    A_\sbt \mapsto \Maps_{\dSch_k}( \Spec A_\sbt[t]/t^{n+1}, X)
\]
when $n < \infty$, and by
\[
    A_\sbt \mapsto \holim_n \Maps_{\dSch_k}(\Spec A_\sbt[t]/t^{n+1}, X)
\]
when $n = \infty$. There are natural truncation maps $\bL J_m(X) \to \bL
J_n(X)$ when $n \le m \le \infty$, which are compatible with the description
of the functors of points. We have that $\bL J_0(X) \cong X$, $\bL J_\infty(X)
\cong \holim_n \bL J_n(X)$, and $\pi_0(\bL J_n(X)) = J_n(\pi_0(X))$. Moreover,
if $U = \Spec R_\sbt \to X$ is an affine open subscheme of $X$, then
\[
    \bL J_n(X) \times_X^\bL U
    \cong
    \bL J_n(U) \cong \Spec \bL J_n(R_\sbt).
\]
\end{thm}

\begin{dff}
We call the derived $k$-scheme $\bL J_n(X)$ given by \cref{thm:globalizejets}
the \emph{derived $n$-th jet space} of $X$ when $n < \infty$, and the
\emph{derived arc space} of $X$ when $n = \infty$.
\end{dff}

\begin{proof}[Proof of \cref{thm:globalizejets}]
For an animated $k$-algebra $R_\sbt$ and an element $f \in \pi_0(R_\sbt)$, the
localization map $R_\sbt \to R_\sbt[f^{-1}]$ is formally étale (see
\cite[Ex.~3.4.8]{DAG}). Therefore, by virtue of \cref{thm:etalelocalization},
we can glue $\bL J_n$ along Zariski affine covers to construct a derived scheme
$\bL J_n (X)$. It is immediate to check the existence of the natural truncation
maps, the formulas $\bL J_0(X) = X$ and $\pi_0(\bL J_n(X)) = J_n(\pi_0(X))$,
and the last claim of the \lcnamecref{thm:globalizejets}.

We start by verifying the representability claim when $n <\infty$. For this,
notice that that the underlying reduced scheme for $\Spec
A_\sbt[t]/(t^{n+1})$ is $\Spec \pi_0(A_\sbt)_\red$, so any map $\Spec
A_\sbt[t]/(t^{n+1}) \to X$ is obtained by gluing maps of the form $\Spec
A_\sbt[f^{-1}][t]/(t^{n+1}) \to U$ where $f \in \pi_0(A_\sbt)$ and $U
\subseteq X$ is a derived affine open (see \cite[Lem.~2.3]{EM09} for a similar
argument in the classical setting). The claim now follows from $\bL J_n(X)
\times_X^\bL U  = \bL J_n(U)$ and \cref{thm:allagree}.

The representability claim when $n=\infty$ follows immediately once we know the
formula $\bL J_\infty(X) \cong \holim_n \bL J_n(X)$. But this can be checked
locally and therefore is a consequence of \cref{thm:ani_arc_hocolim}.
\end{proof}

\begin{rmk}
The proof of \cref{thm:globalizejets} shows that the derived $k$-schemes $\bL
J_n(X)$ can be constructed using a relative spectrum. More precisely,
\cref{thm:etalelocalization} implies that if $\mathcal A_\sbt$ is a
quasi-coherent sheaf of animated $\cO_X$-algebras on $X$, then we can define a
quasi-coherent sheaf of algebras $\bL J_n(\mathcal A_\sbt)$. When $\mathcal
A_\sbt = \cO_X$ we get $\bL J_n(X) = \Spec_X \bL J_n(\cO_X)$. In particular,
all the truncation maps are affine.
\end{rmk}

\subsection{First derived jet schemes and cotangent complexes}

In the classical setting, the first jet scheme verifies the formula
\[
    J_1(X) \cong \Spec_X(\Sym_{\cO_X}(\Omega_{X/k})),
\]
where $\Omega_{X/k}$ is the sheaf of differentials on $X$ and $\Sym$ denotes
the sheafified symmetric algebra functor \cite[Eq.~(1.4)]{Voj07}. One often
summarizes this formula by saying that the first jet scheme $J_1(X)$ is the
total space of the tangent bundle on $X$. The next result we present,
\cref{thm:SpecSym} below, establishes the expected generalization for the first
derived jet scheme, where one just replaces $\Omega_{R/k}$ with the cotangent
complex $\bL_{R/k}$ and the symmetric algebra functor $\Sym$ with its animation
$\LSym$. As it turns out, this generalization is a formal consequence of the
basic properties of animation.

To make sense of the statement of \cref{thm:SpecSym} we need to first discuss
the animation $\LSym$. For this, consider $\Alg_k^\two$, the $1$-category of
arrows on $\Alg_k$; its objects are pairs $(R,S)$ where $R \in \Alg_k$ and $S
\in \Alg_R$ and its arrows are the obvious commutative squares. Also recall the
category $\AlgMod_k$ introduced in \cref{sec:animated-modules}. There is an
obvious forgetful functor $\Alg_k^\two \to \AlgMod_k$ whose left adjoint is the
symmetric algebra functor $\Sym \colon \AlgMod_k \to \Alg_k^\two$ given by
$(R,M) \mapsto (R, \Sym_R(M))$. Since $\Sym$ preserves colimits, we can
consider its left derived functor $\LSym \colon \aAlgMod_k \to \aAlg_k^\two$.
One immediately checks that this functor can be written in the form $(R_\sbt,
M_\sbt) \mapsto (R_\sbt, \LSym_{R_\sbt}(M_\sbt))$, where the second component
is given by a functor $\LSym_{R_\sbt} \colon \aMod_{R_\sbt} \to \aAlg_{R_\sbt}$
(here the $\infty$-category of \emph{animated $R_\sbt$-algebras}
$\aAlg_{R_\sbt}$ is the fiber $\aAlg_k^\two\times_{\aAlg_k}\{R_\sbt\}$).
Explicitly, if $(P_\sbt, F_\sbt) \to (R_\sbt, M_\sbt)$ is a cofibrant
replacement, then
\[
    \LSym_{R_\sbt}(M_\sbt)
    \cong
    \Sym_{P_\sbt}(F_\sbt)
    \otimes_{P_\sbt}
    R_\sbt
    \cong
    \Sym_{R_\sbt}(F_\sbt
    \otimes_{P_\sbt}
    R_\sbt).
\]
This construction can be glued without surprises: given an animated
$\cO_X$-module $\cM_\sbt \in \aMod_X$ we can construct a sheaf of animated
$\cO_X$-algebras $\LSym_{\cO_X}(\cM_\sbt)$. Notice that we can consider the
relative spectrum $\Spec_X(\LSym_{\cO_X}(\cM_\sbt))$.

\begin{thm}
\label{thm:SpecSym}
For $X$ a derived $k$-scheme,  there is a natural equivalence
\[
    \bL J_{1}(X) \cong \Spec_X(\LSym_{\cO_X}(\bL_{X/k}))
\]
where $\bL_{X/k}$ is the cotangent complex of $X$ over $\Spec k$.
\end{thm}

\begin{proof}
The statement is local, so we can assume that $X$ is affine and work in the category
of $k$-algebras. In the algebraic setting, the sought formula reads:
\[
    \bL J_{1}(R_\sbt) \cong \LSym_{R_\sbt}(\bL_{R_\sbt/k}).
\]
Equivalently, consider the functor $F$ given by $F(R) = (R, J_1(R))$, so that
$\bL F(R_\sbt) = (R_\sbt, \bL J_1(R_\sbt))$. Then we need to show that:
\[
    \bL F \cong \LSym \circ\, \bL\Omega.
\]
As explained above, in the classical setting we have $F \cong \Sym \circ\,
\Omega$. Also, notice that if $P \in \Poly_k^\ft$, then $\Omega_{P/k}$ is a
finitely generated free $P$-module, so $(P, \Omega_{P/k})$ is a strongly finitely
presentable object in $\AlgMod_k$. By \cite[Prop.~5.1.5(b)]{CS20}, we have that
$\bL F \cong \Ani(\Sym \circ\, \Omega) \cong \Ani(\Sym) \circ\, \Ani(\Omega) =
\LSym \circ\, \bL\Omega$, as wanted.

Alternatively, we can check the equivalence $\bL F \cong \LSym \circ\,
\bL\Omega$ directly from the constructions. If $P_\sbt \to R_\sbt$ is a
(functorially chosen) cofibrant replacement with terms in $\Poly_k$, then
$(P_\sbt, \Omega_{P_\sbt/k}) \to (R_\sbt, \bL_{R_\sbt/k})$ is a cofibrant
replacement in $\aAlgMod_k$ and therefore
\[
    \LSym_{R_\sbt}(\bL_{R_\sbt/k})
    \cong
    \Sym_{P_\sbt}(\Omega_{P_\sbt/k}) \otimes_{P_\sbt} R_\sbt
    \cong
    J_1(P_\sbt) \otimes_{P_\sbt} R_\sbt
    \cong
    \bL J_1(R_\sbt). \qedhere
\]
\end{proof}

\section{Classicality, quasi-smoothness, and log-canonicity}

\label{sec:discrete-non-discrete}
\label{sec:lci_non_discreteness}

We view the homotopy groups of the derived jet schemes $\bL J_n(X)$ as measures
of the singularities of $X$. This is justified by the following elementary
fact.

\begin{thm}
\label{thm:smoothobjectshavediscretejetsandarcs}
If $X$ is a smooth $k$-scheme, then $\bL J_n(X)$ is classical for all $n \in
\NN \cup \{\infty\}$.
\end{thm}

\begin{proof}
Notice that being classical is a local property, and that $\bL J_n(X)$ is
covered by affine opens of the form $\bL J_n(U)$ where $U \subseteq X$ is an
affine open. So the
\lcnamecref{thm:smoothobjectshavediscretejetsandarcs}
follows if we show that any point $x \in X$ is contained in an affine open $U
\subseteq X$ for which $\bL J_n(U)$ is classical.

For this, notice that the smoothness of $X$ implies that $x$ is contained in an
affine open $U \subseteq X$ admitting an étale map $U \to \AA^d$ for some $d
\in \NN$. By \cref{thm:etalelocalization} we have an equivalence $\bL J_n(U)
\cong \bL J_n(\AA^d)\times_{\AA^d}^\bL U$. Since $\bL J_n(\AA^d)$ is classical
and $U$ is flat over $\AA^d$, we see that $\bL J_n(U)$ is also classical, as
wanted.
\end{proof}

Notice that it follows that when $X$ is smooth, the jet schemes $\bL J_n(X) =
J_n(X)$ are also smooth. As we show now, this remains true when smoothness is
replaced with quasi-smoothness. Recall our discussion about quasi-smoothness
and the connection with local complete intersections from
\cref{sec:intro-koszul-regular,sec:smooth-quasi-smooth-lci}.

To present the result, it will be helpful to recall the explicit description of
the classical jet and arc functors in terms of Hasse-Schmidt differentials
\cite{Voj07}. If $k[\ux] = k[x_i \,|\, i \in I]$ is a polynomial $k$-algebra,
possibly not of finite type, then the associated jet/arc space algebra is also
polynomial:
\[
    J_n(k[\ux])
    = k[\ux, \ux^{(1)}, \ldots, \ux^{(n)}]
    = k[\ux][x_i^{(q)} \,|\, i \in I,\ q = 1,\ldots,n ],
    \qquad
    n \in \NN \cup \{\infty\}.
\]
Moreover, this algebra is the target of the universal $m$-th order
Hasse-Schmidt derivation on $k[\ux]$. That is $J_n(k[\ux])$ is characterized by
having a map
\[
    D_{k[\ux]} \colon k[\ux] \lra J_n(k[\ux]),
    \qquad
    f \to (f^{(0)}, f^{(1)}, \ldots, f^{(n)}),
\]
satisfying
\[
    f^{(0)} = f,
    \qquad
    (f+g)^{(q)} = f^{(q)} + g^{(q)},
    \qquad
    (fg)^{(q)} = \sum_{u+v=q} f^{(u)}g^{(v)},
\]
for $f, g \in k[\ux],\ q=0,\ldots,n$, and $a^{(q)} = 0$ for $a\in k,\ q=1,\ldots,n$.
When $k$ is a $\QQ$-algebra, we can replace Hasse-Schmidt derivations by usual
derivations, and $f, f', f'', \ldots$ can be computed with the usual rules of
implicit differentiation. For arbitrary bases one instead uses implicit divided
differentiation. For a general classical $k$-algebra $A$ presented as a
quotient of a polynomial algebra we have:
\[
    \text{if}
    \quad
    A = \frac{
        k[\ux]
    }{
        (\uf)
    }
    \quad
    \text{then}
    \quad
    J_n(A) = \frac{
        k[\ux, \ux^{(1)}, \ldots, \ux^{(n)}]
    }{
        (\uf, \uf^{(1)}, \ldots \uf^{(n)})
    },
\]
where we are using the notation
\[
    (\uf) = (f_j \,|\, j \in J)
    \quad\text{and}\quad
    (\uf, \uf^{(1)}, \ldots \uf^{(n)}) =
    ( f_j^{(q)} \,|\, j \in J,\ q = 0,\ldots,n ).
\]
As before, $J_n(A)$ is the target of $D_A \colon A \to J_n(A)$, the universal
$n$-th order Hasse-Schmidt derivation on $A$.

\begin{thm}
\label{thm:LJderived-quotients}
Let $R_\sbt$ be an animated $k$-algebra, and consider a (possibly infinite)
sequence $(\uf)$ of elements of $R_\sbt$ (see \Cref{sec:intro-koszul-regular}).
Then, for $n\in \NN \cup \{\infty\}$ we have
\[
    \bL J_n(R_\sbt \sslash (\uf))
    \cong
    \bL J_n(R_\sbt) \sslash (\uf, \uf^{(1)}, \ldots, \uf^{(n)}).
\]
In particular, if $R$ is a classical $k$-algebra and $(\uf) = (f_1, \ldots,
f_c)$ is a Koszul regular sequence in $R$ then
\[
    \bL J_n(R / (\uf))
    \cong
    \bL J_n(R) \sslash (\uf, \uf^{(1)}, \ldots, \uf^{(n)}).
\]
\end{thm}

\begin{proof}
This is an immediate consequence of $\bL J_n$ being left adjoint. In detail, we
have the following pushout diagram in $\aAlg_k$:
\[
    \begin{tikzcd}
        k[\ux] \ar[r,"\uf"] \ar[d]
        & R_\sbt \ar[d]
    \\
        k \ar[r]
        & R_\sbt \sslash (\uf).
        \ar[ul,phantom,very near start, "\ulcorner"]
    \end{tikzcd}
\]
Notice that $\bL J_n(k[\ux]) = J_n(k[\ux]) = k[\ux, \ux^{(1)}, \ldots,
\ux^{(n)}]$, so the map $\bL J_n(\uf)$ induces a sequence of elements $(\uf,
\uf^{(1)}, \ldots \uf^{(n)})$ in $\bL J_n(R_\sbt)$. Since $\bL J_n$ is left
adjoint, it transforms the above pushout diagram into a pushout diagram, giving
the result.
\end{proof}

\begin{thm}
\label{lem:LJquasi-smooth}
If $\colon X \to Z$ is a quasi-smooth closed immersion, then for all $n \in
\NN$ the induced map $\bL J_n(X)\to \bL J_n(Z)$ is also a quasi-smooth closed
immersion. Moreover, if $X$ is a quasi-smooth $k$-scheme, so are the derived
jet schemes $\bL J_n(X)$ for $n \in \NN$.
\end{thm}

\begin{proof}
The first statement is a consequence of $J_n$ preserving finite type polynomial
algebras, and of \Cref{thm:LJderived-quotients}. In detail, we can work
affine-locally and assume that $Z = \Spec R_\sbt$ and $X = \Spec R_\sbt \sslash
(\uf)$, where $(\uf) = (f_1, \ldots, f_c)$ is a finite sequence of elements in
$R_\sbt$. Then the sequence $(\uf, \uf^{(1)}, \ldots \uf^{(n)})$ of elements of
$\bL J_n(R_\sbt)$ is also finite, as $n <\infty$. Then the result follows from
\Cref{thm:LJderived-quotients}.

For the second statement, recall that a quasi-smooth $k$-scheme $X$ admits a
quasi-smooth closed immersion $X \to Z$ where $Z$ is a smooth $k$-scheme. Now
use the first statement and \Cref{thm:smoothobjectshavediscretejetsandarcs}.
\end{proof}

\begin{thm}
\label{thm:lci-jets-quasi-smooth}
Let $X$ be a classical $k$-scheme, and assume that it is lci. Then $\bL J_n(X)$
is quasi-smooth. Moreover, if $X$ is equidimensional of dimension $d$, then
$\bL J_n(X)$ is classical if and only if $J_n(X)$ is lci of dimension $(n+1)d$.
\end{thm}

\begin{proof}
The first statement follows directly from \Cref{lem:LJquasi-smooth}. For the
second, work affine-locally and write $X = \Spec R/(\uf)$, where $R$ is a
smooth $k$-algebra and $(\uf) = (f_1, \ldots, f_c)$ is a Koszul regular
sequence in $R$. Then from
\Cref{thm:smoothobjectshavediscretejetsandarcs,thm:LJderived-quotients} we have
\[
    \bL J_n(R/(\uf)) = J_n(R) \sslash (\uf, \uf^{(1)}, \ldots, \uf^{(n)}),
    \quad
    J_n(R/(\uf)) = J_n(R) / (\uf, \uf^{(1)}, \ldots, \uf^{(n)}),
\]
and the result follows.
\end{proof}

\begin{rmk}
\label{rmk:lci-jets-quasi-smooth}
It follows from the above results that in the lci case we can compute homotopy
groups of derived jets using Koszul homology. Concretely, consider a complete
intersection algebra $A = k[x_1, \ldots, x_d]/(f_1, \ldots, f_c)$, where $(f_1,
\ldots, f_c)$ is a regular sequence in $k[x_1, \ldots, x_d]$. Then
\[
    \pi_i(\bL J_n(A))
    =
    H_i(K_\sbt(
        f_1, \ldots, f_c,
        f^{(1)}_1, \ldots, f^{(1)}_c,
        \ldots,
        f^{(n)}_1, \ldots, f^{(n)}_c
    )),
\]
where $K_\sbt(f_1, f_1^{(1)}, \ldots, f^{(n)}_c)$ is the Koszul complex of the
equations defining $J_n(A)$, i.e., the Hasse-Schmidt derivatives of the
sequence $(f_1, \ldots, f_c)$. Therefore their vanishing for $i>0$ measures
simultaneously the classicality of $\bL J_n(A)$ and the lci nature of $J_n(A)$,
which is the content of \Cref{thm:lci-jets-quasi-smooth}.
\end{rmk}

\begin{xmp}
As an elementary example, consider $A = k[x]/(x^2)$. From
\Cref{rmk:lci-jets-quasi-smooth} we see that $\bL J_1(A) = k[x,x^{(1)}] \sslash
(x^2, 2xx^{(1)})$. Since $(x^2, 2xx^{(1)})$ is not a regular sequence, we get a
first example of a non-classical derived jet scheme. Direct computations with
Koszul complexes show that $\pi_1(\bL J_1(A)) \cong k[x^{(1)}]$.
\end{xmp}

\begin{xmp}
The condition on the dimension is needed in positive characteristic, as
Hasse-Schmidt differentials can vanish. For example, assume $k$ has
characteristic $p$, fix $a \in k$, let $f = x^p-a$, and consider $A =
k[x]/(f)$. Then $f^{(r)} = 0$ if $p$ does not divide $r$, and $f^{(sp)} =
(x^{(s)})^p$. We see that for $n<p$ we have $J_n(A) = k[x, x^{(1)}, \ldots,
x^{(n)}]/(f)$, which is lci. On the other hand $\bL J_n(A) = k[x, x^{(1)},
\ldots, x^{(n)}]\sslash (f, 0, \ldots, 0)$, which is not classical. Notice that
$\bL J_n(A)$ has the expected virtual dimension $0 = (n+1) 0$, whereas $J_n(X)$
has excess dimension $n$. When $a$ is not a $p$-th power (so $k$ is not
perfect) we obtain an example of a reduced scheme with non-classical derived
jet schemes. This phenomenon cannot happen in characteristic zero, as in this
case Hasse-Schmidt differentials of non-constant non-zero elements remain
non-zero.
\end{xmp}


The above results should be considered in connection with results of Mustaţă
and collaborators related to the singularities of the minimal model program and
inversion of adjunction. Concretely, in the lci case they obtain the following
characterization of log canonicity.

\begin{thm}[\cite{Mus01,EMY03,EM04}]
\label{thm:lci-mustata}
Let $X$ be a reduced scheme of finite type over a field $k$ of characteristic
zero, and assume that $X$ is lci. Embed $X \subset Z$, where $Z$ is a smooth
$k$-variety, and let $c$ be the codimension of $X$ in $Z$. Then the following
are equivalent:
\begin{enumerate}[series=lcimus]
    \item\label{i:lci-mustata:1}
        $J_n(X)$ is lci for every $n \in \NN$.
    \item\label{i:lci-mustata:2}
        $\dim(J_n(X)) = (n+1) \dim(X)$ for every $n \in \NN$.
    \item\label{i:lci-mustata:3}
        The pair $(Z, cX)$ is log canonical.
\end{enumerate}
If $X$ is also normal, the above statements are equivalent to:
\begin{enumerate}[resume=lcimus]
    \item\label{i:lci-mustata:4}
        $X$ is log canonical.
\end{enumerate}
\end{thm}

Combining \Cref{thm:lci-mustata} with our results, we immediately obtain the
following.

\begin{cor}
\label{cor:lci-mustata}
With the same hypotheses as \Cref{thm:lci-mustata} (in particular $X$ is lci),
the statements in \Cref{thm:lci-mustata} are equivalent to:
\begin{enumerate*}[resume=lcimus]
    \item\label{i:lci-mustata:5}
        $\bL J_n(X)$ is classical for every $n \in \NN$.
\end{enumerate*}
\end{cor}

\begin{rmk}
From \Cref{cor:lci-mustata} we immediately get many examples of classical
reduced schemes having non-classical derived jet schemes. Elementary examples
are given by cones over plane curves of degree $\ge 4$.
\end{rmk}

As discussed in \cite{dFD14}, the normality condition in item
(\ref{i:lci-mustata:4}) of \Cref{thm:lci-mustata} can be avoided by using the
theory of Mather-Jacobian discrepancies \cite{dFD14,EIM16,EI15}. Concretely, we
have the following result.

\begin{thm}[\cite{dFD14}]
\label{thm:mj-log-canonical}
Let $X$ be a reduced scheme of finite type over a field $k$ of characteristic
zero, and assume that $X$ is equidimensional. Embed $X \subset Z$, where $Z$ is
a smooth $k$-variety, and let $c$ be the codimension of $X$ in $Z$. Then the
following are equivalent:
\begin{enumerate}[label={(\arabic*')}]
    \item\label{i:mj-log-canonical:1}
        $J_n(X)$ is equidimensional for every $n \in \NN$.
    \item\label{i:mj-log-canonical:2}
        $\dim(J_n(X)) = (n+1) \dim(X)$ for every $n \in \NN$.
    \item\label{i:mj-log-canonical:3}
        The pair $(Z, cX)$ is log canonical.
    \item\label{i:mj-log-canonical:4}
        $X$ is log MJ-canonical.
\end{enumerate}
\end{thm}

\begin{rmk}
The terminology ``log MJ-canonical'' used in item \ref{i:mj-log-canonical:4} of
\Cref{thm:mj-log-canonical} is taken from \cite{EI15}, and it corresponds to
what in \cite{dFD14} was called ``log J-canonical.'' The equivalence between
items \ref{i:mj-log-canonical:2} and \ref{i:mj-log-canonical:4} appears in
\cite[Corollary 5.4]{dFD14}, and the equivalence with
\ref{i:mj-log-canonical:3} is the content of the Mather-Jacobian version of
inversion of adjunction, see \cite[Theorem 4.10]{dFD14} or \cite[Proposition
3.10]{Ish13}. The equivalence with \ref{i:mj-log-canonical:1} is elementary.
\end{rmk}

Putting together \cref{thm:mj-log-canonical} and \cref{cor:lci-mustata} we
immediately get the following characterization.

\begin{cor}
\label{thm:mjlc-discrete-jets}
With the same hypotheses as \cref{thm:lci-mustata} (in particular $X$ is lci),
$X$ is log MJ-canonical if and only if $\bL J_n(X)$ is classical for all $n \in
\NN$.
\end{cor}

\section{The cotangent complex on derived jet and arc spaces}

\label{sec:DerJetsCotCplx}

We next turn to a core part of this paper: understanding the cotangent complex
of the derived jets/arcs. Our motivation is a derived version of the main
theorems of \cite{dFD20}.

\subsection{The formula in the classical setting, affine version}

We start by recalling the statements of \cite{dFD20} in detail, and we first
focus on the affine case. Consider a classical $k$-algebra $A$. The units of
the adjunctions $J_n(\Arg) \dashv (\Arg)[t]/(t^{n+1})$ and $J_\infty(\Arg)
\dashv (\Arg)\llbracket t \rrbracket$ are called the \emph{universal $n$-jet}
and \emph{universal arc}, respectively. It will be convenient to give names to
the right adjoints, so in this section we will write $W_n(\Arg) =
(\Arg)[t]/(t^{n+1})$ and $W_\infty(\Arg) = (\Arg)\llbracket t \rrbracket$. The
universal jet/arc are given by
\[
    A \lra W_n(J_n(A)),
    \qquad
    f \mapsto \sum_{i=0}^n f^{(i)} t^i,
\]
where, as recalled in detail in \Cref{sec:lci_non_discreteness}, the symbol
$f^{(i)}$ denotes the universal $i$-th Hasse-Schmidt differential of $f$.

We also need to introduce pre-duals to $W_n$. Namely, for a $k$-algebra $C$ we
set
\[
    V_n(C) = \frac{
        t^{-n} \cdot C\llbracket t \rrbracket
    }{
        t \cdot C\llbracket t \rrbracket
    },
    \quad
    \text{and}
    \quad
    V_\infty(C) = \frac{
        C\llparen t \rrparen
    }{
        t \cdot C\llbracket t \rrbracket
    }.
\]
Notice that $V_n(C)$ is both a $W_n(C)$-module and a $C$-module (i.e., a
bimodule), and that we have a pairing
\[
    V_n(C) \otimes_C W_n(C) \to C
    \quad
    \text{given by}
    \quad
    \langle t^{-i}, t^j \rangle = \delta_{ij},
\]
which naturally identifies the $C$-dual of $V_n(C)$ with $W_n(C) \cong
\Hom_C(V_n(C), C)$.

\begin{thm}[\cite{dFD20}]
\label{thm:dFDaffine}
There exists a natural isomorphism
\[
    \Omega_{J_n(A)/k}
    \cong
    \Omega_{A/k} \otimes_A P_n(A)
\]
where $P_n(A) := V_n(J_n(A))$. In the formula, $P_n(A)$ is a
$(W_n(J_n(A)),J_n(A))$-bimodule and becomes an $(A, J_n(A))$-bimodule via the
universal jet/arc $A \to W_n(J_n(A))$.
\end{thm}

In the terminology introduced in \cite{CNM23}, $P_m(A)$ is called the universal
\emph{Hasse-Schmidt module} associated to $A$.

\subsection{The formula in the animated setting, affine version}

\label{sec:DerJetsCotCplxAffine}

We adapt the proof of \cite{dFD20} to give a version of \Cref{thm:dFDaffine} in
the animated setting. First recall from \Cref{dff:Wm} that the functors $W_n$
extend directly to functors on $\aAlg_k$. Similar considerations allow us to
extend the functors $V_n$. In detail, degree-wise application induces functors
from simplicial $k$-algebras to an appropriate category of simplicial
bimodules. At the level of simplicial Abelian groups, they are just given by an
appropriate coproduct:
\[
    V_n(C_\sbt) \cong \bigoplus_{i = 0}^n C_\sbt\, t^{-i}
    \quad
    \text{in}
    \quad
    \sAb.
\]
It follows that these induced functors preserve weak equivalences, and hence
give functors at the animated level. As with $W_n$, we use the same notation
for these functors $V_n$ with domain $\aAlg_k$. For an animated $k$-algebra
$C_\sbt$, $V_n(C_\sbt)$ is an animated $(W_n(C_\sbt), C_\sbt)$-bimodule. In the
category $\aMod_{C_\sbt}$ of animated $C_\sbt$-modules, and even in the derived
category, $V_n(C_\sbt)$ is the pre-dual of $W_n(C_\sbt)$. In fact, for any
animated $C_\sbt$-module $M_\sbt$ we have:
\begin{multline}
\label{eq:predualformula}
    \RHom_{C_\sbt}
    (
        V_n(C_\sbt),
        M_\sbt
    )
\cong
    \RHom_{C_\sbt}
    (
        \oplus_i C_\sbt\,t^{-i},
        M_\sbt
    )
\\ \cong \textstyle
    \prod_i \RHom_{C_\sbt}
    (
        C_\sbt,
        M_\sbt
    )\, t^i
\cong \textstyle
    \prod_i M_\sbt\,t^i
\cong
    W_n(M_\sbt).
\end{multline}
Here $\RHom_{C_\sbt}$ denotes the internal Hom in the derived category of
$C_\sbt$. In the last equality we have used that $W_n(\Arg)$ can be applied not
just to algebras but also to modules, and that the resulting functor extends
directly to the animated setting. Notice that $W_n(M_\sbt)$ is a
$W_n(C_\sbt)$-module, and \eqref{eq:predualformula} holds at the level of
$W_n(C_\sbt)$-modules.

From \Cref{thm:allagree} we get an adjunction $\bL J_n \dashv W_n$ in $\aAlg_k$
for all $n \in \NN \cup \{\infty\}$. The unit of this adjunction has components
\[
    A_\sbt \lra W_n(\bL J_n(A_\sbt))
    \qquad
    \text{for}
    \quad
    n \in \NN \cup \{\infty\},
\]
which we call the \emph{derived universal jet} when $n <\infty$ and the
\emph{derived universal arc} when $n=\infty$. We now have all the ingredients
to state the derived version of \Cref{thm:dFDaffine}.

\begin{thm}
\label{thm:affinederiveddFD}
For $n\in \NN \cup \{\infty\}$,
there exists a natural isomorphism
\[
    \bL_{\bL J_n(A_\sbt)/k}
    \cong
    \bL_{A_\sbt/k} \otimes_{A_\sbt}^\bL P_n^\der(A_\sbt)
\]
where $P_n^\der(A_\sbt) := V_n(\bL J_n(A_\sbt))$. In the formula,
$P_n^\der(A_\sbt)$ is by definition a bimodule over $(W_n(\bL J_n(A_\sbt)),\bL
J_n(A_\sbt))$ and becomes an $(A_\sbt, \bL J_n(A_\sbt))$-bimodule via the
derived universal jet/arc $A_\sbt \to W_n(\bL J_n(A_\sbt))$.
\end{thm}

\begin{rmk}
\label{rmk:explicit-Pm-and-trunctions}
Unraveling the definitions, we see that
\[
    P_n^\der(A_\sbt) = \frac{
        t^{-n} \cdot \bL J_n(A_\sbt)\llbracket t \rrbracket
    }{
        t \cdot \bL J_n(A_\sbt)\llbracket t \rrbracket
    }
    \qquad
    \text{and}
    \qquad
    P_\infty^\der(A_\sbt) = \frac{
        \bL J_\infty(A_\sbt)\llparen t \rrparen
    }{
        t \cdot \bL J_\infty(A_\sbt)\llbracket t \rrbracket
    }.
\]
In particular, notice that $\pi_0(P_n^\der(A_\sbt)) \cong P_n(\pi_0(A_\sbt))$.
For $n \le m$ we have natural maps $P_n^\der(A_\sbt) \to P_m^\der(A_\sbt)$
of animated $(A_\sbt, \bL J_n(A_\sbt))$-bimodules; these are induced by the
truncation maps $\bL J_n (A_\sbt) \to \bL J_m(A_\sbt)$ and ``send $t$ to
$t$.'' It is immediate to check that
\[
    P_\infty^\der(A_\sbt)
    \cong \hocolim_n P_n^\der(A_\sbt)
    \quad
    \text{and therefore}
    \quad
    \bL_{\bL J_\infty(A_\sbt)/k}
    \cong \hocolim_n \bL_{\bL J_n(A_\sbt)/k}
\]
as $\bL J_\infty(A_\sbt)$-modules. In this derivation of the second colimit,
the transition maps are induced by the maps $P_n^\der(A_\sbt) \to
P_m^\der(A_\sbt)$. As we will see in \cref{rmk:truncation-ugly-details} these
transition maps agree with the natural maps $\bL_{\bL J_n(A_\sbt)/k} \to
\bL_{\bL J_m(A_\sbt)/k}$ induced by the truncation maps $\bL J_n (A_\sbt) \to
\bL J_m(A_\sbt)$.
\end{rmk}

\begin{rmk}
\label{rmk:naturality-details}
One of the crucial aspects of \cref{thm:affinederiveddFD} is that the given
isomorphism is \emph{natural}. In precise terms, we have a natural isomorphism
between
\[
    \bL_{\bL J_n(\Arg)/k}
    \qquad\text{and}\qquad
    \bL_{\Arg/k} \otimes_{\Arg}^\bL P_n^\der(\Arg),
\]
considered as functors with domain $\aAlg_k$ and codomain $\aAlgMod_k$. In
particular, if $A_\sbt \to B_\sbt$ is a morphism of animated $k$-algebras, we
have induced morphisms $\bL J_n(A_\sbt) \to \bL J_n(B_\sbt)$ in $\aAlg_k$, and
$\bL_{A_\sbt/k} \to \bL_{B_\sbt/k}$, $\bL_{\bL J_n(A_\sbt)/k} \to \bL_{\bL J_n
(B_\sbt)/k}$, and $P_n^\der(A_\sbt) \to P_n^\der(B_\sbt)$ in $\aAlgMod_k$. The
content of naturality is that the induced morphisms
\[
    \bL_{\bL J_n(A_\sbt)/k}
    \lra
    \bL_{\bL J_n (B_\sbt)/k}
    \qquad\text{and}\qquad
    \bL_{A_\sbt/k} \otimes_{A_\sbt}^\bL P_n^\der(A_\sbt)
    \lra
    \bL_{B_\sbt/k} \otimes_{B_\sbt}^\bL P_n^\der(B_\sbt)
\]
correspond to each other via the isomorphism given in
\cref{thm:affinederiveddFD}. After base change to $B_\sbt$ we get an
identification between
\[
    \bL_{\bL J_n(A_\sbt)/k} \otimes^\bL_{\bL J_n(A_\sbt)} \bL J_n(B_\sbt)
    \lra
    \bL_{\bL J_n (B_\sbt)/k}
\]
and
\[
    \bL_{A_\sbt/k} \otimes_{A_\sbt}^\bL P_n^\der(B_\sbt)
    \lra
    \bL_{B_\sbt/k} \otimes_{B_\sbt}^\bL P_n^\der(B_\sbt),
\]
where we have used the natural isomorphism
$
    P_n^\der(B_\sbt)
    \cong
    P_n^\der(A_\sbt)
    \otimes^\bL_{\bL J_n(A_\sbt)}
    \bL J_n(B_\sbt)
$.
\end{rmk}

To prove \cref{thm:affinederiveddFD} we will use the characterization of the
cotangent complex in terms of animated derivations, see \cref{sec:anim-ders}.
Since these are defined using slice categories, we need to understand how
adjunctions behave with respect to overcategories. This is straightforward in
the $1$-categorical setting, but more delicate for $\infty$-categories. The
technicalities were studied by Lurie, who provides us with the following
result.

\begin{thm}[{\cite[Proposition 5.2.5.1]{Lur09}}]
\label{thm:lurie-overadjunctions}
Suppose we are given an adjunction $F \dashv G$ between $\infty$-categories,
with $F \colon \cC \to \cD$ and $G \colon \cD \to \cC$. Consider an object $C$
in $\cC$ and the overcategories $\cC/C$ and $\cD/F(C)$, and let $f \colon \cC/C
\to \cD/F(C)$ be the functor naturally induced by $F$. Assume that $\cC$ admits
pullbacks. Then $f$ admits a right adjoint $g$. This right adjoint can be
chosen to be $g = g'' \circ g'$, where $g' \colon \cD/F(C) \to \cD/G(F(C))$ is
the functor induced by $g$ and $g'' \colon \cD/G(F(C)) \to \cD/C$ is given by
pull-back along the unit $C \to G(F(C))$.
\end{thm}

\begin{rmk}
\label{rmk:lurie-overadjunctions}
At the level of objects, given $D$ in $\cD$ with a map $\varphi \colon D \to
F(C)$, the functor $g$ is determined by the following pull-back diagram:
\[
    \begin{tikzcd}
        g(D) \arrow[r] \arrow[d]
        \ar[dr,phantom,very near start, "\lrcorner"]
        & G(D) \arrow[d, "G(\varphi)"] \\
        C \arrow[r]
        & G(F(C)).
    \end{tikzcd}
\]
\end{rmk}

\begin{lem}\label{lem:derived_lem_5_1}
Let $n\in\NN\cup\{\infty\}$. Let $A_\sbt$ be an animated $k$-algebra and
$M_\sbt$ be an animated $\bL J_n(A_\sbt)$-module. There is a natural
equivalence
\[
    \aDer_k(\bL J_n(A_\sbt), M_\sbt)
    \cong
    \aDer_k(A_\sbt, W_n(M_\sbt)),
\]
where $W_n(M_\sbt)$, which is by definition a $W_n(J_n(A_\sbt))$-module, is
considered as an $A_\sbt$-module via the universal jet/arc $A_\sbt \to
W_n(J_n(A_\sbt))$.
\end{lem}

\begin{proof}
We apply \Cref{thm:lurie-overadjunctions} with $\cC =\cD = \aAlg_k$, $F = \bL
J_n$, $G = W_n$, and $C = A_\sbt$. In \Cref{thm:allagree} we obtained the
adjunction $\bL J_n \dashv W_n$. From the animated $\bL J_n(A_\sbt)$-module
$M_\sbt$ we build the square-zero extension $\bL J_n(A_\sbt) \oplus \varepsilon
M_\sbt$, which plays the role of $D$ in \Cref{rmk:lurie-overadjunctions}. The
value of $g$ at $D$ is given by the top-left corner of the pull-back
diagram
\[
    \begin{tikzcd}
        A_\sbt \oplus \varepsilon W_n(M_\sbt)
        \arrow[r] \arrow[d]
        \ar[dr,phantom,very near start, "\lrcorner"]
        & W_n(\bL J_n(A_\sbt)) \oplus \varepsilon W_n(M_\sbt)
        \arrow[d] \\
        A_\sbt \arrow[r]
        & W_n(\bL J_n(A_\sbt)),
    \end{tikzcd}
\]
where the vertical arrows are projections onto the first summand, and the
bottom arrow is the universal jet/arc. Then \Cref{thm:lurie-overadjunctions}
gives us the following natural equivalence:
\[
    \Maps_{\aAlg_k/\bL J_n(A_\sbt)}(
        \bL J_n(A_\sbt),
        \bL J_n(A_\sbt) \oplus \varepsilon M_\sbt
    )
    \cong
    \Maps_{\aAlg_k/A_\sbt}(
        A_\sbt,
        A_\sbt \oplus \varepsilon W_n(M_\sbt)
    ).
\]
The result now follows directly from the definition of the space of animated
derivations in terms of mapping spaces in overcategories.
\end{proof}

\begin{proof}[Proof of \Cref{thm:affinederiveddFD}]
Let $M_\sbt$ be an animated $\bL J_n(A_\sbt)$-module. By the universal
property of the cotangent complex, \Cref{lem:derived_lem_5_1}, equation
\eqref{eq:predualformula}, and the derived tensor-Hom adjunction, we get the
following chain of natural equivalences:
\begin{align*}
    &\Maps_{\aMod_{\bL J_n(A_\sbt)}}(
        \bL_{\bL J_n(A_\sbt)/k},
        M_\sbt
    )
\\ & \cong
    \aDer_{k}(
        \bL J_n(A_\sbt),
        M_\sbt
    )
\\ & \cong
    \aDer_{k}(
        A_\sbt,
        W_n(M_\sbt)
    )
\\ & \cong
    \Maps_{\aMod_{A_\sbt}}(
        \bL_{A_\sbt/k},
        W_n(M_\sbt)
    )
\\ & \cong
    \Maps_{\aMod_{A_\sbt}}(
        \bL_{A_\sbt/k},
        \RHom_{\bL J_n(A_\sbt)}(
            P_n^\der(A_\sbt),
            M_\sbt
        )
    )
\\ & \cong
    \Maps_{\aMod_{\bL J_n(A_\sbt)}}(
        \bL_{A_\sbt/k}
        \otimes_{A_\sbt}^\bL
        P_n^\der(A_\sbt),
        M_\sbt
    ).
\end{align*}
By the $\infty$-Yoneda lemma, $\bL_{\bL J_n(A_\sbt)/k} \cong
\bL_{A_\sbt/k}\otimes_{A_\sbt}^\bL P_n^\der(A_\sbt)$ as $\bL
J_n(A_\sbt)$-modules.
\end{proof}

\begin{rmk}
\label{rmk:truncation-ugly-details}
The proof of \cref{thm:affinederiveddFD} gives information about the truncation
maps. Indeed, we can follow the same chain of equivalences as in the proof, but
keeping track of the truncation maps. First, if $n \le m \le \infty$, the map
$\bL_{\bL J_n(A_\sbt)} \to \bL_{\bL J_m(A_\sbt)}$ corresponds to the natural
transformation $\aDer_k(\bL J_m(A_\sbt), \Arg) \to \aDer_k(\bL
J_n(A_\sbt), \Arg)$ induced by the truncation map $\bL J_n(A) \to \bL
J_m(A)$. This in turn corresponds to a natural transformation
$\aDer_k(A_\sbt, W_m(\Arg)) \to \aDer_k(A_\sbt, W_n(\Arg))$, induced by the
map $W_m(\Arg) \to W_n(\Arg)$ obtained by modding out by $t^{n+1}$. We then
see that this gives the natural transformation
$
\RHom_{\bL J_m(A_\sbt)}(
    P_m^\der(A_\sbt),
    \Arg
)
\to
\RHom_{\bL J_n(A_\sbt)}(
    P_n^\der(A_\sbt),
    \Arg
)
$
corresponding to the map
$
    P_n^\der(A_\sbt)
    \to
    P_m^\der(A_\sbt)
$
which is induced by the truncation $\bL J_n(A) \to \bL J_m(A)$
and sends $t$ to $t$.
Finally we arrive to the map
$
        \bL_{A_\sbt/k}
        \otimes_{A_\sbt}^\bL
        P_n^\der(A_\sbt),
\to
        \bL_{A_\sbt/k}
        \otimes_{A_\sbt}^\bL
        P_m^\der(A_\sbt),
$
induced by the same map
$
    P_n^\der(A_\sbt)
    \to
    P_m^\der(A_\sbt)
$
as before.
\end{rmk}

\subsection{The formula for derived schemes, global version}

In this subsection we show how to extend the formula from
\Cref{thm:affinederiveddFD} to the non-affine setting. Before explaining the
construction, a few remarks about our approach are in order.

\begin{rmk}
Ideally we would like to generalize the main theorem of \cite{dFD20} in the
global setting to derived schemes. When $n < \infty$ this is straightforward,
but when $n = \infty$ the situation becomes very subtle. The main issue is the
nature of the universal arc, whose source is always a formal scheme which is
not a scheme. We also need to consider a global version of $P_\infty(A)$, which
turns out to be a sheaf of non-complete modules on a formal scheme. We do not
know how to develop a theory of derived formal schemes where such non-complete
modules make sense. As a workaround we consider certain non-coherent sheaves of
modules on $\bL J_\infty(X)$.
\end{rmk}

We consider the functors $\bL J_n$ and $W_n$ introduced in
\Cref{sec:DerJetsCotCplxAffine}. These are of an algebraic nature, and
associate animated $k$-algebras $\bL J_n(A_\sbt)$ and $W_n(A_\sbt)$ to any
animated $k$-algebra $A_\sbt$. Let $X$ be a derived scheme, and consider a
sheaf $\cA_\sbt$ of animated $\cO_X$-algebras. By composition with the functors
$\bL J_n$ and $W_n$, and sheafification, we can define $\bL J_n(\cA_\sbt)$ and
$W_n(\cA_\sbt)$, which are also sheaves of animated algebras. If $\cA_\sbt$ is
quasi-coherent, then $\bL J_n(\cA_\sbt)$ is also quasi-coherent for $n \in \NN
\cup \{\infty\}$. The same happens with $W_n(\cA_\sbt)$ when $n \in \NN$ but in
general $W_\infty(\cA_\sbt)$ is not quasi-coherent. When $\cA_\sbt$ is taken to
be $\cO_X$ we see that:
\[
    \bL J_n(\cO_X) = \cO_{\bL J_n(X)},
\]
where the right-hand side is considered as a sheaf of animated $\cO_X$-algebras
via push-forward through the projection $\bL J_n(X) \to X$. The unit of the
adjunction induces a map of sheaves of animated $\cO_X$-algebras
\[
    \cO_X \lra W_n(\bL J_n(\cO_X)),
\]
the global incarnation of the universal jet/arc. We next define
\[
    \mathcal P_n^\der(X) =
    P_n^\der(\cO_X) = V_n(\bL J_n(\cO_X)),
\]
which is naturally a sheaf of animated $(W_n(\bL J_n(\cO_X)),\bL
J_n(\cO_X))$-bimodules and becomes a $(\cO_X,\bL J_n(\cO_X))$-bimodule via the
global universal jet/arc. Notice that $\mathcal P_n^\der(X)$ is quasi-coherent
for any $n \in \NN \cup \{\infty\}$. The cotangent complex is also a functorial
construction, and we have that
\[
    \bL_{\cO_X/k} = \bL_{X/k}
    \quad\text{and}\quad
    \bL_{\bL J_n(\cO_X)/k} = \bL_{\bL J_n(X)/k}.
\]
After this setup, the global version of \Cref{thm:affinederiveddFD} is essentially
tautological.

\begin{thm}
\label{thm:deriveddFD}
With the notations introduced above, we have a natural equivalence
\[
    \bL_{\bL J_n(X)/k}
    \cong
    \bL_{X/k} \otimes_{\cO_X}^\bL \mathcal P_n^\der(X)
\]
for any $n \in \NN \cup \{\infty\}$.
\end{thm}

\begin{proof}
We can rewrite the sought formula as
\[
    \bL_{\bL J_n(\cO_X)/k}
    \cong
    \bL_{\cO_X/k} \otimes_{\cO_X}^\bL P_n^\der(\cO_X).
\]
This can be checked in affine charts $\Spec(A_\sbt) \subseteq X$, where we the
formula becomes:
\[
    \bL_{\bL J_n(A_\sbt)/k}
    \cong
    \bL_{A_\sbt/k} \otimes_{A_\sbt}^\bL P_n^\der(A_\sbt).
\]
But this is precisely the formula provided by \Cref{thm:affinederiveddFD}.
\end{proof}

\begin{cor}
\label{cor:deriveddFDtriangle}
Fix $n \in \NN \cup \{ \infty \}$, and let $f \colon X \to Y$ be a morphism of
schemes. Denote by $\bL_{X/Y}$ the cotangent complex of the map $f$. Similarly
let $f_n \colon \bL J_n X \to \bL J_n Y$ the induced map, and let $\bL_{\bL J_n
X/\bL J_n Y}$ be the cotangent complex of $f_n$. We have a natural equivalence
\[
    \bL_{\bL J_n X/\bL J_n Y}
    \cong
    \bL_{X/Y} \otimes^{\bL}_{\mathcal O_X} \mathcal P_n^{\der}(X).
\]
\end{cor}

\begin{proof}
We have the fundamental triangles
\[
    \bL_{\bL J_n Y/k} \otimes_{\cO_{\bL J_n Y}}^\bL \cO_{\bL J_n X}
    \lra[\ \ \varphi\ \ ]
    \bL_{\bL J_n X/k}
    \lra
    \bL_{\bL J_n X/\bL J_n Y}
    \xra{-1}
\]
and
\[
    \bL_{Y/k} \otimes_{\cO_{Y}}^\bL \cO_{X}
    \lra
    \bL_{X/k}
    \lra
    \bL_{X/Y}
    \xra{-1}.
\]
Tensoring the second triangle by
$\cP_n^{\der}(X)$ shows that
\[
    \bL_{Y/k} \otimes_{\cO_Y}^\bL \cP_n^{\der}(X)
    \lra[\ \ \psi\ \ ]
    \bL_{X/k} \otimes_{\cO_{X}}^{\bL} \cP_n^{\der}(X)
    \lra
    \bL_{X/Y} \otimes_{\cO_{ X}}^{\bL} \cP_n^{\der}(X)
    \xra{-1}
\]
is exact. Notice that the middle term is isomorphic to
$\bL_{\bL J_n X/k}$
by \cref{thm:deriveddFD}. For the first term, notice that
\[
    \bL_{\bL J_n Y/k} \otimes_{\cO_{\bL J_n Y}}^{\bL} \cO_{\bL J_n X}
\cong
    \bL_{Y/k}
    \otimes_{\cO_Y}^{\bL}
    \mathcal P_n^{\der}(Y)
    \otimes_{\cO_{\bL J_n Y}}^{\bL}
    \cO_{ \bL J_n X}
\cong
    \bL_{Y/k}
    \otimes_{\cO_Y}^{\bL}
    \cP_n^{\der}(X).
\]
To conclude, we only need to show that $\psi$ and $\phi$ agree. But this
follows from the naturality of the isomorphism provided by
\cref{thm:deriveddFD}, as discussed in \cref{rmk:naturality-details}.
\end{proof}

\subsection{Failure of the formula for classical jet and arc spaces}

\label{sec:failure}

It is natural to ask whether \Cref{thm:affinederiveddFD,thm:deriveddFD} require
derived jets, as opposed to the classical ones. That is, since
\cite[Thms.~A,~B]{dFD20} computes the sheaf of differentials of the classical
jets, one is led to ask if we cannot compute the cotangent complex of the
classical jets directly. The crux of the matter is in
\Cref{lem:derived_lem_5_1}; the cotangent complex is calculated via animated
derivations, which are mapping spaces, and it is only via the derived jets that
we have the required representability at the level of mapping spaces. As the
next results show, this is not merely an artifact of our proof technique.

\begin{thm}
\label{thm:we_need_derived_jets}
Let $n \in \NN \cup \{\infty\}$ and consider a classical $k$-algebra $A$.
There is an equivalence
\[
    \bL_{J_n(A)/k}
    \cong
    \bL_{A/k} \otimes_A^\bL P_n(A)
\]
if and only if the derived jet/arc space $\bL J_n(A)$ is classical.
\end{thm}

\begin{proof}
The composition $k \to \bL J_n(A) \to J_n(A)$ induces the following exact
triangle:
\[
    \bL_{\bL J_n(A)/k}
    \otimes_{\bL J_n(A)}^\bL
    J_n(A)
\to
    \bL_{J_n(A)/k}
\to
    \bL_{J_n(A)/\bL J_n(A)}
\xra{-1}.
\]
Applying \Cref{thm:deriveddFD} and using that $P_n^\der(A)$ is explicitly a
coproduct of copies of $\bL J_n(A)$, the first term becomes
\[
    \bL_{A/k} \otimes_A^\bL P_n^\der(A)
    \otimes_{\bL J_n(A)}^\bL
    J_n(A)
    \cong
    \bL_{A/k} \otimes_A^\bL P_n(A).
\]
Therefore the equivalence in the statement holds if and only if
$\bL_{J_n(A)/\bL J_n(A)}$ vanishes. By a criterion of Lurie
\cite[Cor.~25.3.6.6]{SAG}, this is equivalent to a weak equivalence $\bL J_n(A)
\cong J_n(A)$.
\end{proof}

\begin{rmk}
As we saw in \Cref{sec:lci_non_discreteness}, there are many examples of
classical $k$-schemes $X$ with non-classical derived jet schemes $\bL J_n(X)$.
For any of these examples, the ``classical version'' of the formula for the
cotangent complex does not hold. The obstruction is given by the homology
groups of $\bL_{J_n(A)/\bL J_n(A)}$, which would be interesting to understand.
\end{rmk}

\section{Fibers of the cotangent complex and Fitting ideals}

\label{sec:fibers-fitting}

We now apply the previous study to compute fibers of cotangent complexes on
derived arc and jet spaces. As the computation of fibers is a local problem we
work in the affine case.

\subsection{André-Quillen homology on jet and arc spaces}

Fibers of cotangent complexes are given by André-Quillen homology. For more
details on André-Quillen homology and cohomology, see \cite{Iye07}. Recall the
basic definitions: for a classical $k$-algebra $A$ and an $A$-module $M$, $
\AQ_{i}(A/k,M) = \Tor_{i}^{A}(\bL_{A/k},M) $ denotes the \emph{$i$-th
André-Quillen homology of $A$ with coefficients in $M$}, and $ \AQ^{i}(A/k,M) =
\Ext_{A}^{i}(\bL_{A/k},M) $ denotes the \emph{$i$-th André-Quillen cohomology
of $A$ with coefficients in $M$}. This generalizes naturally for animated
inputs. Specifically, if $M_\sbt$ is an animated module over an animated
$k$-algebra $A_\sbt$, we set
\begin{align*}
    \AQ_{i}(A_\sbt/k,M_\sbt)
    &:=\Tor_{i}^{A_\sbt}(\bL_{A_\sbt/k},M_\sbt)
    = \pi_{i}(\bL_{A_\sbt/k}\otimes_{A_\sbt}^\bL M_\sbt),
    \quad\text{and}
\\
    \AQ^{i}(A_\sbt/k,M_\sbt)
    &:=\Ext_{A_\sbt}^{i}(\bL_{A_\sbt/k},M_\sbt)
    = \pi_{-i}(\RHom_{A_\sbt}(\bL_{A_\sbt/k},M_\sbt)).
\end{align*}

\cref{thm:affinederiveddFD} can be rephrased in terms of André-Quillen homology
and cohomology.

\begin{thm}
Let $X=\Spec A_\sbt$ be a derived affine $k$-scheme. Let
$n\in\mathbf{N}\cup\{\infty\}$ and let $M_\sbt$ be an animated $\bL
J_n(A_\sbt)$-module. For all $i$, we have the following isomorphisms of
André-Quillen homology and cohomology:
\begin{enumerate}
    \item $
        \AQ_i(\bL J_n(A_\sbt)/k, M_\sbt)
        \cong
        \AQ_i(
            A_\sbt/k,
            P^\der_n(A_\sbt)
            \otimes_{\bL J_nA_\sbt}^\bL
            M_\sbt
        )
    $.
    \item $
        \AQ^i(\bL J_n(A_\sbt)/k, M_\sbt)
        \cong \AQ^i(
            A_\sbt/k,
            \RHom_{\bL J_n(A_\sbt)}(
                P^\der_n(A_\sbt),
                M_\sbt
            )
        )
    $.
\end{enumerate}
\end{thm}

\begin{rmk}
With the notations introduced in \Cref{sec:DerJetsCotCplxAffine}, the second
statement can be be rewritten as:
\[
        \AQ^i(\bL J_n(A_\sbt)/k, M_\sbt)
        \cong \AQ^i(
            A_\sbt/k,
            W_n(M_\sbt)
        ).
\]
\end{rmk}

\begin{proof}
By \cref{thm:affinederiveddFD}, we have a quasi-isomorphism of complexes
\begin{align*}
    \bL_{\bL J_n(A_\sbt)/k}
    \cong\bL_{A_\sbt/k}\otimes_{A_\sbt}^\bL P^\der_n(A_\sbt).
\end{align*}
To prove the first statement, observe the following chain of identifications.
\begin{align*}
    \AQ_i(\bL J_n(A_\sbt)/k,M_\sbt)&=\Tor_i^{\bL J_n(A_\sbt)}(\bL_{\bL J_n(A_\sbt)/k},M_\sbt)\\
    &\cong\Tor_i^{\bL J_n(A_\sbt)}(\bL_{A_\sbt/k}\otimes_{A_\sbt}^\bL P^\der_n(A_\sbt),M_\sbt)\\
    &=\pi_i((\bL_{A_\sbt/k}\otimes_{A_\sbt}^\bL P^\der_n(A_\sbt))\otimes_{\bL J_n(A_\sbt)}^\bL M_\sbt)\\
    &=\pi_i(\bL_{A_\sbt/k}\otimes_{A_\sbt}^\bL(P^\der_n(A_\sbt)\otimes_{\bL J_n(A_\sbt)}^\bL M_\sbt))\\
    &=\Tor_i^{A_\sbt}(\bL_{A_\sbt/k},P^\der_n(A_\sbt)\otimes_{\bL J_n(A_\sbt)}^\bL M_\sbt)\\
    &=\AQ_i(A_\sbt/k,P^\der_n(A_\sbt)\otimes_{\bL J_n(A_\sbt)}^\bL M_\sbt).
\end{align*}
Using derived tensor-Hom adjunction the following symmetric argument proves
the cohomological case.
\begin{align*}
    \AQ^i(\bL J_n(A_\sbt)/k,M_\sbt)&=\Ext_{\bL J_n(A_\sbt)}^i(\bL_{\bL J_n(A_\sbt)/k},M_\sbt)\\
    &\cong\Ext_{\bL J_n(A_\sbt)}^i(\bL_{A_\sbt/k}\otimes_{A_\sbt}^\bL P^\der_n(A_\sbt),M_\sbt)\\
    &=\pi_{-i}(\RHom_{\bL J_n(A_\sbt)}(\bL_{A_\sbt/k}\otimes_{A_\sbt}^\bL P^\der_n(A_\sbt),M_\sbt))\\
    &\cong \pi_{-i}(\RHom_{A_\sbt}(\bL_{A_\sbt/k},\RHom_{\bL J_n(A_\sbt)}(P^\der_n(A_\sbt),M_\sbt)))\\
    &=\Ext_{A_\sbt}^i(\bL_{A_\sbt/k},\RHom_{\bL J_n(A_\sbt)}(P^\der_n(A_\sbt),M_\sbt))\\
    &=\AQ^i(A_\sbt/k,\RHom_{\bL J_n(A_\sbt)}(P^\der_n(A_\sbt),M_\sbt)).
    \qedhere
\end{align*}
\end{proof}

\subsection{Initial remarks}

\label{sec:remarks-before-csi}

To make the following results more explicit, we will focus on the case where $X
= \Spec A$ is a classical affine $k$-scheme. We consider arcs $\alpha \colon A
\to k_\alpha\llbracket t \rrbracket$, where $k_\alpha$ is a (classical) field, and we are
interested in studying $\bL_{\bL J_\infty(A)/k} \otimes^\bL_{\bL J_\infty(A)}
k_\alpha$. From the preceding discussion this is equivalent to understanding
the André-Quillen homology of the arc space with coefficients in $k_\alpha$:
\begin{align*}
\AQ_i(\bL J_\infty(A)/k, k_\alpha)
&\simeq
\AQ_i(
    A/k,
    P^\der_\infty(A)
    \otimes^\bL_{\bL J_\infty(A)}
    k_\alpha
)
\\
&=
\pi_i(
    \bL_{A/k}
    \otimes^\bL_A
    P^\der_\infty(A)
    \otimes^\bL_{\bL J_\infty(A)}
    k_\alpha
).
\end{align*}
We also consider the truncations of $\alpha$, which we denote $\alpha_n$, and
look at the corresponding fibers $\bL_{\bL J_n(A)/k} \otimes^\bL_{\bL J_n(A)}
k_{\alpha_n}$:
\begin{align*}
\AQ_i(\bL J_n(A)/k, k_{\alpha_n})
&\simeq
\AQ_i(
    A/k,
    P^\der_n(A)
    \otimes^\bL_{\bL J_n(A)}
    k_{\alpha_n}
)
\\
&=
\pi_i(
    \bL_{A/k}
    \otimes^\bL_A
    P^\der_n(A)
    \otimes^\bL_{\bL J_n(A)}
    k_{\alpha_n}
).
\end{align*}
The case $i=0$ was described in \cite{dFD20}. In what follows we will use a
similar analysis to study the higher homotopy groups.

Notice that $k_\alpha$ is a $J_\infty(A)$-algebra, and it becomes a $\bL
J_\infty(A)$-algebra via the truncation map $\bL J_\infty(A) \to J_\infty(A)$.
From the definitions we immediately see that
\[
    P^\der_\infty(A) \otimes^\bL_{\bL J_\infty(A)} k_\alpha
    \simeq
    V_\infty(\bL J_\infty(A)) \otimes^\bL_{\bL J_\infty(A)} k_\alpha
    \simeq
    V_\infty(k_\alpha)
    =
    \frac{k_\alpha\llparen t\rrparen}{t\, k_\alpha\llbracket t\rrbracket},
\]
and therefore
\[
\bL_{\bL J_\infty(A)/k} \otimes^\bL_{\bL J_\infty(A)} k_\alpha
\simeq
\bL_{A/k} \otimes^\bL_{A} V_\infty(k_\alpha)
\simeq
\bL_{A/k} \otimes^\bL_{A} k_\alpha\llbracket t \rrbracket
\otimes^\bL_{k_\alpha\llbracket t \rrbracket} V_\infty(k_\alpha).
\]
Similarly, for the jets we get:
\[
\bL_{\bL J_n(A)/k} \otimes^\bL_{\bL J_n(A)} k_{\alpha_n}
\simeq
\bL_{A/k}
\otimes^\bL_{A}
\frac{
    k_{\alpha_n}\llbracket t \rrbracket
}{
    (t^{n+1})
}
\otimes^\bL
V_n(k_{\alpha_n})
\simeq
\bL_{A/k}
\otimes^\bL_{A}
\frac{
    k_{\alpha_n}\llbracket t \rrbracket
}{
    (t^{n+1})
}
.
\]
Notice that $V_n(k_{\alpha_n})$ is a free $k_{\alpha_n}\llbracket t
\rrbracket/(t^{n+1})$-module of rank $1$, giving the last isomorphism. We have
reduced our problem to taking pull-backs of $\bL_{A/k}$ along $\alpha$ and
$\alpha_n$. To do this effectively, we next discuss a variant of the notion of
Fitting ideal tailored to work for complexes in derived categories.

\subsection{Cohomological support ideals}

\label{sec:csi}

Let $A$ be a ring and consider a object $M_\sbt$ in the derived category
$D(A)$. Recall that $M_\sbt$ is said to be \emph{pseudo-coherent} if it is
quasi-isomorphic to a bounded below\footnote{Note that we use homological
indexing, so pseudo-coherent complexes are bounded below, not above.} complex
of finite free $A$-modules \SPcite{064Q}. In other words, we have
\[
    M_\sbt
    \simeq
    \Big[
    \quad
    \cdots
    \lra[\phantom{\partial_4}] A^{r_{i+1}}
    \lra[\,~\partial_{i+1}~] A^{r_i}
    \lra[\,~\partial_i~] A^{r_{i-1}}
    \lra[\phantom{\partial_1}] \cdots
    \quad
    \Big],
\]
where the $r_i$ are non-negative integers which vanish for $i$ small enough. We
call a complex of this form a \emph{resolution} of $M_\sbt$. Resolutions are
not unique, but any two are homotopic. If $A$ is local then $M_\sbt$ is
quasi-isomorphic to a unique \emph{minimal} resolution, meaning that the
matrices $\partial_i$ have entries in the maximal ideal of $A$; in fact
$M_\sbt$ is isomorphic to the direct sum of the minimal resolution and a
trivial complex \cite[Sec.~20.1]{Eis95}.

\begin{dff}
With the notation above, the \emph{cohomological support ideal} of
$M_\sbt$ of level $(i,p)$ is defined as
\[
    \Csi_{(i,p)}(M_\sbt)
    =
    \sum_{u+v=r_i-p}
    I_u(\partial_{i+1}) I_v(\partial_i),
\]
where $I_a(\partial_j)$ denotes the ideal of $A$ generated by the $a\times a$
minors of $\partial_j$. The zero loci of the cohomological support ideals are
called the \emph{cohomological support loci} of $M_\sbt$.
\end{dff}

\begin{rmk}
Using the structure of resolutions over local rings, it is easy to show that
the cohomological support ideals are independent of the choice of resolution.
They were introduced in \cite{GL87}, where it is shown that the zero locus of
the cohomological support ideal of level $(i,p)$ is
\[
    \{\,
        x \in \Spec(A)
    \,|\,
        \dim_{k_x}(\pi_i(M_\sbt \otimes^\bL_A k_x)) > p
    \,\},
\]
hence the name ``cohomological support locus.''
\end{rmk}

\begin{rmk}
As usual, we identify an $A$-module $M$ with the object $M[0]$ of $D(A)$ whose
only non-zero term is $M$ in degree zero. Thus $M$ is pseudo-coherent if it
admits a free resolution with finitely generated terms. For example, if $A$ is
Noetherian, or more generally $A$ is coherent, then every finitely generated
module is pseudo-coherent. For a pseudo-coherent $A$-module $M$, its
cohomological support ideals of level $(0,p)$ are the classical Fitting ideals:
\[
    \Csi_{(0,p)}(M) = \Fitt^{p}(M) = I_{r_0-p}(\partial_1).
\]
In particular the ideal $\Fitt^{p}(M)$ defines the locus where the fibers of $M$
have dimension $> p$. If $M_\sbt$ is connective and both $M_\sbt$ and
$\pi_0(M_\sbt)$ are pseudo-coherent, then
\[
    \Csi_{(0,p)}(M_\sbt) = \Fitt^{p}(\pi_0(M_\sbt)).
\]
\end{rmk}

\begin{rmk}
Following \SPcite{08CA}, a sheaf of $\cO_X$-complexes $\cM_\sbt$ on a classical
scheme $X$ is \emph{pseudo-coherent} if on small enough affine charts $U =
\Spec(A) \subset X$ there are pseudo-coherent complexes $M_\sbt$ over $A$ which
induce the restrictions $\cM_\sbt|_U$ up to quasi-isomorphism. The notion of
cohomological support ideal extends directly to this global setting, and we get
coherent sheaves of ideals satisfying $\Csi_{(i,p)}(\cM_\sbt)|_U \cong
\Csi_{(i,p)}(\cM_\sbt|_U)$. If $X$ is a locally Noetherian scheme (or more
generally a locally coherent scheme), any coherent $\cO_X$-module is
pseudo-coherent.
\end{rmk}

\subsection{Pull-backs of complexes along arcs and jets}

\label{sec:pull-backs-complexes-arcs}

As above, fix a classical ring $A$ and a pseudo-coherent complex $M_\sbt$ in
$\aMod_A \subset D(A)$. We now consider an arc $\alpha \colon A \to
k_\alpha\llbracket t \rrbracket$, where $k_\alpha$ is a (classical) field, and
are interested in studying the complex $M_\sbt \otimes_A^\bL k_\alpha\llbracket
t \rrbracket$ in $D(k_\alpha\llbracket t \rrbracket)$. In subsequent
subsections we will specialize this analysis to the case of the cotangent
complex.

Consider a resolution for $M_\sbt$ as above,
\[
    M_\sbt \otimes^\bL_{A} k_\alpha\llbracket t \rrbracket
    \simeq
    \Big[
    \quad
    \cdots
    \lra[\phantom{\partial_3}] k_\alpha\llbracket t \rrbracket^{r_2}
    \lra[\,~\partial_2(\alpha)~] k_\alpha\llbracket t \rrbracket^{r_1}
    \lra[\,~\partial_1(\alpha)~] k_\alpha\llbracket t \rrbracket^{r_0}
    \lra[\phantom{\partial_1}] 0
    \quad
    \Big].
\]
Using Gaussian elimination we can find bases for each term of the complex in
such a way that each of the matrices $\partial_i(\alpha)$ has a nice shape.
Explicitly, notice that $\ker(\partial_i(\alpha))$ is a free summand of
$k_\alpha\llbracket t \rrbracket^{r_i}$. If we write $k_\alpha\llbracket t
\rrbracket^{r_i} = \ker(\partial_i(\alpha)) \oplus C_{i-1}$, then the matrix of
$\partial_i(\alpha)$ with respect to this decomposition has block form $
    \left(\begin{smallmatrix}
    0 & * \\ 0 & 0 \\
    \end{smallmatrix}\right)
$. Doing row and column operations we can ensure the top-right block is
anti-diagonal, with entries powers of $t$ with increasing exponents. In other
words, if $c_i$ is the rank of $C_i$, we can write
\[
    \partial_{i+1}(\alpha) =
    \rule{0mm}{24.5mm}
    \left(
        \begin{tikzpicture}[
            baseline=-2.5,
            yscale=0.5,
            xscale=0.8,
            circle dotted/.style={
                dash pattern=on .01mm off 1.6mm,
                line cap=round,
            }
        ]
            \useasboundingbox (-3.45,-3.4) rectangle (3.45,3.4);
            \node at ( 0, 0) {0};
            \node at ( 0, 2) {0};
            \node at ( 0,-2) {0};
            \node at ( 2, 0) {0};
            \node at (-2, 0) {0};
            \node at (-2, 2) {0};
            \node at ( 2,-2) {0};
            \node at (-2,-2) {0};
            \node (A) at ( 1, 1) {\small \,$t^{a_{i,c_i}}$};
            \node (B) at ( 3, 3) {\small $t^{a_{i,1}}$\,};

            \draw[densely dotted] ( 0.3125, 3.5) -- ( 0.3125, 4);
            \draw[densely dotted] (-0.3125, 3.5) -- (-0.3125, 4.6);
            \draw ( 0.3125,-3.5) -- ( 0.3125, 3.5);
            \draw[dashed] (-0.3125,-3.5) -- (-0.3125, 3.5);

            \draw[densely dotted] ( 3.5, 0.5) -- ( 4.6, 0.5);
            \draw[densely dotted] ( 3.5,-0.5) -- ( 5.6,-0.5);
            \draw[dashed] (-3.5, 0.5) -- ( 3.5, 0.5);
            \draw (-3.5,-0.5) -- ( 3.5,-0.5);

            \draw[line width = .4mm, circle dotted] (A) -- (B);

            \draw[<->] (-3.5, 4.6) -- node[fill=white] {\tiny $c_{i+1}$} (-0.3125, 4.6);
            \draw[<->] (-3.5, 4) -- node[fill=white] {\tiny $r_{i+1}-c_i$} ( 0.3125, 4);
            \draw[<->] ( 3.5, 4) -- node[fill=white] {\tiny $c_{i}$}   ( 0.3125, 4);

            \draw[<->] ( 4.6, 3.5) -- node[fill=white] {\tiny $c_{i}$}       ( 4.6, 0.5);
            \draw[<->] ( 4.6,-3.5) -- node[fill=white,yshift=2] {\tiny $r_{i}-c_i$} ( 4.6, 0.5);
            \draw[<->] ( 5.6,-3.5) -- node[fill=white,yshift=-3] {\tiny $c_{i-1}$} ( 5.6,-0.5);
            \draw[<->] ( 5.6, 0.5) -- node[right] {\tiny $b_i$} ( 5.6,-0.5);
        \end{tikzpicture}
    \right)
    \hspace{4.7em},
\]
where the $a_{i,j}$ are non-negative integers satisfying $a_{i,j} \le
a_{i,j+1}$. It follows that the homotopy groups are
\[
    \pi_i(M_\sbt \otimes^\bL_{A} k_\alpha\llbracket t \rrbracket)
    \cong
    k_\alpha\llbracket t \rrbracket^{b_{i}}
    \oplus
    \bigoplus_{j=1}^{c_i}
    \frac{
        k_\alpha\llbracket t \rrbracket
    }{
        (t^{a_{i,j}})
    },
    \qquad
    b_i = r_{i} - c_i - c_{i-1}.
\]
We call $b_i$ the $i$-th \emph{Betti number} of $M_\sbt$ with respect to
$\alpha$, and $a_{i,1}, \ldots, a_{i,c_{i+1}}$ the \emph{invariant factors} of
$M_\sbt$ with respect to $\alpha$ at level $i$. To indicate the dependency on
$\alpha$ and $M_\sbt$ we will sometimes write $b_i = b_i(\alpha) = b_i(\alpha,
M_\sbt)$, and similarly for $a_{i,j}$ and $c_i$. For convenience we will set
$a_{i,j} = \infty$ when $j>c_i$, and $a_{i,j}=0$ when $j<0$. Notice that this
is compatible with the usual convention $t^\infty=0$.

A similar analysis can be done for jets instead of arcs. If $\alpha_n \colon A
\to k_{\alpha_n}[t]/(t^{n+1})$ is an $n$-jet, then the matrices of the complex
$M_\sbt \otimes_A^\bL k_{\alpha_n}[t]/(t^{n+1})$ admit a block form of the same
shape as above. We get notions of \emph{Betti numbers} and \emph{invariant
factors} of $M_\sbt$ with respect to the jet $\alpha_n$. The homotopy groups
are now more involved:
\[
    \pi_i\left(
        M_\sbt
        \otimes^\bL_{A}
        \frac{
            k_{\alpha_n}[t]
        }{
            (t^{n+1})
        }
    \right)
    \cong
    \left(
        \frac{
            k_{\alpha_n}[t]
        }{
            (t^{n+1})
        }
    \right)^{b_i}
    \oplus
    \bigoplus_{j=1}^{c_i}
    \frac{
        k_{\alpha_n}[t]
    }{
        (t^{a_{i,j}})
    }
    \oplus
    \bigoplus_{j=1}^{c_{i-1}}
    \frac{
        (t^{n+1-a_{i-1,j}})
    }{
        (t^{n+1})
    }.
\]
Their dimension is:
\[
    \dim_{k_{\alpha_n}}
    \left(
    \pi_i\left(
        M_\sbt
        \otimes^\bL_{A}
        \frac{
            k_{\alpha_n}[t]
        }{
            (t^{n+1})
        }
    \right)
    \right)
    =
    (n+1)b_i
    + a_{i,1} + \cdots + a_{i,c_i}
    + a_{i-1,1} + \cdots + a_{i-1,c_{i-1}}.
\]

The invariant factors determine the order of contact with respect to the
cohomological support ideals. Explicitly, in the case of arcs a direct
computation shows that
\[
    \ord_\alpha(
        \Csi_{(i,p)}(M_\sbt)
    )
    =
    \min_{u+v=r_i-p}
    \{
        a_{i,1} + \cdots + a_{i,u}
        + a_{i-1,1} + \cdots + a_{i-1,v}
    \}.
\]
Recall that $r_i = c_i + b_i + c_{i-1}$, and that $a_{i,j}=\infty$ when $j>c_i$. In
particular we have:
\[
    \ord_\alpha(
        \Csi_{(i,b_i)}(M_\sbt)
    )
    =
    a_{i,1} + \cdots + a_{i,c_i}
    + a_{i-1,1} + \cdots + a_{i-1,c_{i-1}}
    ,
\]
and also:
\[
    a_{i,1} + \cdots + a_{i,c_i}
    =
    \sum_{j=0}^i
        (-1)^j
    \ord_\alpha(
        \Csi_{(i-j,b_{i-j})}(M_\sbt)
    )
    .
\]

Similar formulas hold for jets, provided the involved orders of contact remain
below $n$. For example, if $n \ge \ord_{\alpha_n}(\Csi_{(i,b_i)}(M_\sbt))$
then:
\[
    \dim_{k_{\alpha_n}}
    \left(
    \pi_i\left(
        M_\sbt
        \otimes^\bL_{A}
        \frac{
            k_{\alpha_n}[t]
        }{
            (t^{n+1})
        }
    \right)
    \right)
    =
    (n+1)\,b_i
    +
    \ord_{\alpha_n}(
        \Csi_{(i,b_i)}(M_\sbt)
    ).
\]

If the jet $\alpha_n$ is the truncation of an arc $\alpha$, then the jet
pull-back $M_\sbt \otimes_A^\bL k_{\alpha_n}[t]/(t^{n+1})$ is obtained from the
arc pull-back $M_\sbt \otimes_A^\bL k_\alpha\llbracket t \rrbracket$ by modding
out by $t^{n+1}$. In particular the invariant factors of $\alpha_n$ are
determined by the invariant factors of $\alpha$, via the formula
$a_{i,j}(\alpha_n) = \min\{a_{i,j}(\alpha), n+1\}$. The Betti number $b_i$
increases by one for each $a_{i,j}$ that is larger than $n$. In particular, the
Betti numbers do not change if $n$ is large enough:
\[
    \text{if}\quad
    n \ge a_{i,c_i}(\alpha)\quad
    \text{then}\quad
    b_{i}(\alpha_n)
    =
    b_{i}(\alpha).
\]
Notice that if $n$ is larger than or equal to the order of contact of $\alpha$
with the cohomological support ideal of level $(i,b_i)$, or of level
$(i+1,b_{i+1})$, then we automatically have $n \ge a_{i,c_i}(\alpha)$.

\subsection{Fibers of the cotangent complex}

After introducing cohomological support ideals, we can now define higher
analogues of Jacobian ideals.

\begin{dff}
Let $X$ be a classical $k$-scheme such that $\bL_{X/k}$ is pseudo-coherent. The
\emph{higher Jacobian ideals} of $X$ are defined as
\[
    \Jac_X^{(i,p)} = \Csi_{(i,p)}(\bL_{X/k}).
\]
\end{dff}

\begin{rmk}
If $k$ is Noetherian and $X$ is essentially of finite type over $k$, then
$\bL_{X/k}$ is pseudo-coherent. Indeed, by \SPcite{08PX} small-enough affine
charts $U = \Spec(A) \subset X$ admit cofibrant replacements $P_\sbt \to A$
where each term $P_i$ is a localization of a finite-type polynomial algebra
over $k$. Then $\bL_{A/k}$ is quasi-isomorphic to $\Omega_{P_\sbt/k}
\otimes_{P_\sbt} A$, which is a bounded below complex of finite free
$A$-modules.
\end{rmk}

\begin{rmk}
If $k$ is a field and $X$ is essentially of finite type over $k$, reduced, and
equidimensional with $d = \dim X$, then the classical \emph{Jacobian ideal} of
$X$ can be computed as:
\[
    \Jac_X = \Jac_X^{(0,d)} = \Csi_{(0,d)}(\bL_{X/k}) = \Fitt^{d}(\Omega_{X/k}).
\]
\end{rmk}

In line with the discussion above, for the rest of the section we assume $X =
\Spec(A)$ is an affine $k$-scheme such that $\bL_{A/k}$ is pseudo-coherent. We
have fixed an arc $\alpha \colon A \to k_\alpha\llbracket t \rrbracket$ and its
truncations $\alpha_n$. Unless otherwise specified, Betti numbers and invariant
factors of $\alpha$ and $\alpha_n$ are always considered with respect to the
cotangent complex $\bL_{A/k}$. The discussion in the previous subsections
immediately gives the following.

\begin{prop}
\label{thm:cotangent-fibers-general}
With the notation above, if $n \ge \ord_\alpha(\Jac_X^{(i,b_i(\alpha))})$ then
we have an equality of Betti numbers $b_i(\alpha) = b_i(\alpha_n)$ and
\[
    \dim_{k_{\alpha_n}}
    (
    \pi_i(
        \bL_{\bL J_n(A)/k}
        \otimes^{\bL}
        k_{\alpha_n}
    )
    )
    =
    (n+1) b_i(\alpha)
    +
    \ord_\alpha\left(
            \Jac_X^{(i,b_i(\alpha))}
    \right).
\]
\end{prop}

At the level of arc spaces the situation is more delicate, as the dimensions
are infinite and also $V_\infty(k_\alpha)$ is not free over $k_\alpha\llbracket
t \rrbracket$. It will be sufficient for us to consider the following
consequence of the Künneth formula.

\begin{prop}\label{prop:picotangent}
With the notation above, if we write
\[
    \pi_i(
        \bL_{A/k} \otimes^\bL_A k_\alpha\llbracket t \rrbracket
    )
    \cong
    k_\alpha\llbracket t \rrbracket^{b_{i}}
    \oplus
    \bigoplus_{j=1}^{c_i}
    \frac{
        k_\alpha\llbracket t \rrbracket
    }{
        (t^{a_{i,j}})
    },
\]
then we have
\[
    \pi_i(
        \bL_{\bL J_\infty(A)/k} \otimes^\bL_{\bL J_\infty(A)} k_\alpha
    )
    \cong
    V_\infty(k_\alpha)^{b_{i}}
    \oplus
    \bigoplus_{j=1}^{c_{i-1}}
    \frac{
        k_\alpha\llbracket t \rrbracket
    }{
        (t^{a_{i-1,j}})
    }.
\]
\end{prop}

\begin{proof}
Recall that $ \bL_{\bL J_\infty(A)/k} \otimes^\bL_{\bL J_\infty(A)} k_\alpha
\cong \bL_{A/k} \otimes^\bL_A k_\alpha\llbracket t \rrbracket
\otimes^\bL_{k_\alpha\llbracket t \rrbracket} V_\infty(k_\alpha)$. Since
$k_\alpha\llbracket t \rrbracket$ is a principal ideal domain, the result
follows from the Künneth formula \cite[Sec.~3.6]{Wei94}.
\end{proof}

The following condition on the arc appears frequently in the literature and
allows us to make more precise statements.

\begin{dff}
\label{dff:non-degenerate}
Let $X$ be a classical scheme essentially of finite type over a field $k$. We
say that a field-valued arc $\alpha$ on $X$ is \emph{non-degenerate} if its
generic point belongs to the smooth locus of $X$.
\end{dff}

Assume the arc $\alpha$ is non-degenerate, and denote by $d$ the dimension of
$X$ at the generic point of the arc $\alpha$. Since $X$ is smooth at this
generic point, the fiber $\bL_{A/k} \otimes^\bL_A k_\alpha\llparen t\rrparen$
is quasi-isomorphic to a free $k_\alpha\llparen t\rrparen$-module of rank $d$.
In particular, the Betti numbers are determined: $b_0(\alpha) = d$ and
$b_i(\alpha) = 0$ for $i >d$. With these remarks, the following corollary is
immediate.

\begin{cor}
\label{cor:cotangent-fibers-non-deg}
With the notation above, assume furthermore that $k$ is a field, that $X$ is
essentially of finite type over $k$, and that $\alpha$ is non-degenerate. Let
$d$ be the dimension of $X$ at the generic point of $\alpha$. If $n \ge
\ord_\alpha(\Jac_X^{(0,d)})$ we have an equality of Betti numbers $b_0(\alpha)
= b_0(\alpha_n) = d$, and
\[
    \dim_{k_{\alpha_n}}
    (
    \pi_0(
    \bL_{\bL J_n(A)/k}
    \otimes^\bL
    k_{\alpha_n}
    )
    )
    =
    (n+1)d
    +
    \ord_\alpha\left(
            \Jac_X^{(0,d)}
    \right).
\]
Similarly, if $n \ge \ord_\alpha(\Jac^{(i,0)}_X)$ we have $b_i(\alpha) =
b_i(\alpha_n) = 0$ and
\[
    \dim_{k_{\alpha_n}}
    (
    \pi_i(
    \bL_{\bL J_n(A)/k}
    \otimes^\bL
    k_{\alpha_n}
    )
    )
    =
    \ord_\alpha\left(
        \Jac_X^{(i,0)}
    \right).
\]
At the level of arc spaces, the $0$-th homotopy group is
\[
    \pi_0(
        \bL_{\bL J_\infty(A)/k}
        \otimes^\bL
        k_\alpha
    )
    \cong
    V_\infty(k_\alpha)^d,
\]
and the higher homotopy groups are torsion over $k_\alpha\llbracket t
\rrbracket$ with
\[
    \dim_{k_\alpha}
    (
    \pi_1(
    \bL_{\bL J_\infty(A)/k}
    \otimes^\bL
    k_\alpha
    )
    )
    =
    \ord_\alpha\left(
        \Jac_X^{(0,d)}
    \right)
\]
and
\[
    \dim_{k_\alpha}
    (
    \pi_i(
    \bL_{\bL J_\infty(A)/k}
    \otimes^\bL
    k_\alpha
    )
    )
    =
    \ord_\alpha\left(
        \Jac_X^{(i-1,0)}
    \right)
    \qquad
    \text{for $i \ge 2$}.
\]

\end{cor}

\subsection{Behaviour under truncation}

In the following sections we will also need to understand the behaviour of
these fibers under the truncations maps. For $m \ge n$, the natural map $\bL
J_n(A) \to \bL J_m(A)$ induces a morphism at the level of cotangent complexes:
\[
    \bL_{\bL J_n(A)/k} \otimes^\bL_{\bL J_n(A)} \bL J_m(A)
    \lra
    \bL_{\bL J_m(A)/k}.
\]
Taking fibers at the truncations $\alpha_n$ and $\alpha_m$, we get
\[
    \bL_{\bL J_n(A)/k} \otimes^\bL_{\bL J_n(A)}
    k_{\alpha_n}
    \otimes_{k_{\alpha_n}}
    k_{\alpha_m}
    \lra
    \bL_{\bL J_m(A)/k} \otimes^\bL_{\bL J_n(A)} k_{\alpha_m},
\]
which, via \cref{thm:affinederiveddFD} (see \cref{sec:remarks-before-csi}),
takes the following form:
\[
    \bL_{A/k}
    \otimes^\bL_{A}
    V_n(k_{\alpha_m})
    \lra
    \bL_{A/k}
    \otimes^\bL_{A}
    V_m(k_{\alpha_m}).
\]
This map is in turn induced by the following natural inclusion (recall that $m
\ge n$, see \cref{rmk:explicit-Pm-and-trunctions,rmk:truncation-ugly-details}):
\[
    V_n(k_{\alpha_m})
    =
    \frac{
        t^{-n}\,k_{\alpha_m}[t]
    }{
        t\,k_{\alpha_m}[t]
    }
    \hookrightarrow
    V_m(k_{\alpha_m})
    =
    \frac{
        t^{-m}\,k_{\alpha_m}[t]
    }{
        t\,k_{\alpha_m}[t]
    }
    .
\]
After identifying $V_n(k_{\alpha_m})$ with $k_{\alpha_m}[t]/(t^{n+1})$, and
similarly for $V_m(k_{\alpha_m})$, the resulting map looks like
``multiplication by $t^{m-n}$'':
\[
    \bL_{A/k}
    \otimes^\bL_{A}
    \frac{
        k_{\alpha_m}[t]
    }{
        (t^{n+1})
    }
    \xra{\quad1\otimes t^{m-n}\quad}
    \bL_{A/k}
    \otimes^\bL_{A}
    \frac{
        k_{\alpha_m}[t]
    }{
        (t^{m+1})
    }
    .
\]
The following diagram gives a computation of the kernel and cokernel, assuming
$m-n \ge a_{i,c_i}$.
\[
\begin{tikzcd}[
    baseline=(Z.base),
    row sep=small,
    column sep=0pt,
    cells={font=\everymath\expandafter{\the\everymath\displaystyle}},
]
    \ker(\pi_i(\psi_{n,m}))
    \ar[d]
&
    \cong
&
    0
&
    \oplus
&
    \bigoplus_{j=1}^{c_i}
    \frac{
        k_{\alpha_m}[t]
    }{
        (t^{a_{i,j}})
    }
&
    \oplus
&
    0
\\
    \pi_i(\bL_{\bL J_n(A)/k} \otimes^\bL k_{\alpha_m})
    \ar[d,"\smash{\cdot\, t^{m-n}}"]
&
    \cong
&
    \left(
        \frac{
            k_{\alpha_m}[t]
        }{
            (t^{n+1})
        }
    \right)^{b_i}
&
    \oplus
&
    \bigoplus_{j=1}^{c_i}
    \frac{
        k_{\alpha_m}[t]
    }{
        (t^{a_{i,j}})
    }
&
    \oplus
&
    \bigoplus_{j=1}^{c_{i-1}}
    \frac{
        (t^{n+1-a_{i-1,j}})
    }{
        (t^{n+1})
    }
\\
    \pi_i(\bL_{\bL J_m(A)/k} \otimes^\bL k_{\alpha_m})
    \ar[d]
&
    \cong
&
    \left(
        \frac{
            k_{\alpha_m}[t]
        }{
            (t^{m+1})
        }
    \right)^{b_i}
&
    \oplus
&
    \bigoplus_{j=1}^{c_i}
    \frac{
        k_{\alpha_m}[t]
    }{
        (t^{a_{i,j}})
    }
&
    \oplus
&
    \bigoplus_{j=1}^{c_{i-1}}
    \frac{
        (t^{m+1-a_{i-1,j}})
    }{
        (t^{m+1})
    }
\\
    \coker(\pi_i(\psi_{n,m}))
&
    \cong
&
    \left(
        \frac{
            k_{\alpha_m}[t]
        }{
            (t^{m-n})
        }
    \right)^{b_i}
&
    \oplus
&
    \bigoplus_{j=1}^{c_i}
    \frac{
        k_{\alpha_m}[t]
    }{
        (t^{a_{i,j}})
    }
&
    \oplus
&
    |[alias=Z]| 0
\end{tikzcd}
\qedhere
\]

We summarize the above discussion in the following proposition.

\begin{prop}
\label{thm:trunc-cotangent-fibers-general}
With the notation above, let $m \ge n$ and consider the truncations $\alpha_n$,
$\alpha_m$ and the corresponding cotangent map induced by truncation:
\[
    \psi_{n,m} \colon
    \bL_{\bL J_n(A)/k} \otimes^\bL k_{\alpha_m}
    \lra
    \bL_{\bL J_m(A)/k} \otimes^\bL k_{\alpha_m}.
\]
If $m \ge n + \ord_\alpha(\Jac_X^{(i,b_i(\alpha_n))})$
and $n \ge \max_{0 \le j \le i}
\{\ord_\alpha(\Jac_X^{(j,b_j(\alpha))})\}$, then $b_i(\alpha_n) = b_i(\alpha_m)
= b_i(\alpha)$,
\[
\begin{aligned}
    \dim_{k_{\alpha_m}}(\ker(\pi_i(\psi_{n,m})))
    &=
    a_{i,1}(\alpha) + \cdots + a_{i,c_i(\alpha)}(\alpha)
    ,
\end{aligned}
\]
and
\[
    \dim_{k_{\alpha_m}}(\coker(\pi_i(\psi_{n,m})))
    =
    (m-n) b_i(\alpha)
    + a_{i,1}(\alpha) + \cdots + a_{i,c_i(\alpha)}(\alpha).
\]
\end{prop}

It is also natural to look at the relative cotangent complex $\bL_{\bL
J_m(A)/\bL J_n(A)}$. For this, recall that the sums of invariant factors above
can be recovered from orders of contact with respect to cohomological support
ideals:
\[
    a_{i,1}(\alpha) + \cdots + a_{i,c_i(\alpha)}(\alpha)
    =
    \sum_{j=0}^i
    (-1)^j
    \ord_\alpha\left(
        \Jac^{(i-j,b_{i-j}(\alpha))}_X
    \right)
    .
\]

\begin{cor}
With the assumptions of \cref{thm:trunc-cotangent-fibers-general}, we have
\[
    \dim_{k_{\alpha_m}}(
    \pi_i(
        \bL_{
            \bL J_m(A)/\bL J_n(A)
        }
        \otimes^{\bL}
        k_{\alpha_m}
    ))
    =
    (m-n) b_{i}(\alpha)
    + \ord_\alpha\left(\Jac_X^{(i,b_i(\alpha))}\right).
\]
\end{cor}

\begin{proof}
The composition
$k \to \bL J_n(A) \to \bL J_m(A)$ yields an exact triangle of
$k_{\alpha_m}$-complexes:
\[
    \bL_{\bL J_n(A)/k} \otimes^\bL k_{\alpha_m}
    \xra{\psi_{n,m}}
    \bL_{\bL J_m(A)/k} \otimes^\bL k_{\alpha_m}
    \lra
    \bL_{\bL J_m(A)/\bL J_n(A)} \otimes^\bL k_{\alpha_m}
    \xra{-1}.
\]
From the corresponding long exact sequence we get:
\[
\begin{multlined}
    \dim_{k_{\alpha_m}}(
    \pi_i(
        \bL_{
            \bL J_m(A)/\bL J_n(A)
        }
        \otimes^{\bL}
        k_{\alpha_m}
    ))=
    \\\qquad=
    \dim_{k_{\alpha_m}}(\ker(\pi_{i-1}(\psi_{n,m})))
    +\dim_{k_{\alpha_m}}(\coker(\pi_{i}(\psi_{n,m}))).
\end{multlined}
\]
The result now follows from
\cref{thm:trunc-cotangent-fibers-general}.
\end{proof}

In the case of non-degenerate arcs we immediately get the following corollary.

\begin{cor}
With the notation of \cref{thm:trunc-cotangent-fibers-general}, assume
furthermore that $k$ is a field, that $X$ is essentially of finite type over
$k$, and that $\alpha$ is non-degenerate. Let $d$ be the dimension of $X$ at
the generic point of $\alpha$. If $m \ge n+ \ord_\alpha(\Jac_X^{(0,d)})$ and $n
\ge \ord_\alpha(\Jac_X^{(0,d)})$ we have
\[
    \dim_{k_{\alpha_m}}(\ker(\pi_0(\psi_{n,m})))
    =
    \ord_\alpha\left(
            \Jac_X^{(0,d)}
    \right)
\]
and
\[
    \dim_{k_{\alpha_m}}(\Omega_{J_m(A)/J_n(A)}\otimes k_{\alpha_m})
    =
    \dim_{k_{\alpha_m}}(\coker(\pi_0(\psi_{n,m})))
    =
    (m-n) d +
    \ord_\alpha\left(
            \Jac_X^{(0,d)}
    \right).
\]
Similarly, if $m \ge n+ \ord_\alpha(\Jac_X^{(1,0)})$
and $n \ge \ord_\alpha(\Jac_X^{(1,0)})$
then we have
\[
    \dim_{k_{\alpha_m}}(\ker(\pi_1(\psi_{n,m})))
    =
    \dim_{k_{\alpha_m}}(\coker(\pi_1(\psi_{n,m})))
    =
    \ord_\alpha\left(
            \Jac_X^{(1,0)}
    \right)
    -
    \ord_\alpha\left(
            \Jac_X^{(0,d)}
    \right)
\]
and
\[
    \dim_{k_{\alpha_m}}(
    \pi_1(
        \bL_{\bL J_m(A)/\bL J_n(A)}\otimes^\bL k_{\alpha_m}
    ))
    =
    \ord_\alpha\left(
            \Jac_X^{(1,0)}
    \right).
\]
\end{cor}

\section{Embedding dimension}

\label{sec:emb-dim}

In preparation for applications in
\cref{sec:curveselection,sec:cotangent_maps}, we study in detail the embedding
dimension of points in jet schemes and arc spaces. The interaction between the
embedding dimension and the jet codimension introduced in
\cref{sec:curveselection} is the key technical ingredient in our approach to
the study of stable arcs and the curve selection lemma. We take advantage of
the homological algebra unlocked by considering the derived arc space, as well
as the computational apparatus developed in \cref{sec:fibers-fitting}.

\subsection{Transcendence degree and the Cartier equality}

\label{sec:cartier-equality}

Before studying the embedding dimension on arc spaces, we recall the connection
between transcendence degree of field extensions and cotangent complexes.
Assume $k$ is a field and let $Z$ be a classical $k$-scheme essentially of
finite type over $k$. Let $p \in Z$ be a point with residue field $k_p$. We
denote by $\overline{\{p\}} \subset Z$ the Zariski closure of $p$ in $Z$,
considered as a subscheme with its reduced structure. We have:
\[
    \dim(\overline{\{p\}})
    =
    \trdeg(k_p/k)
    =
    \dim_{k_p}(\Omega_{k_p/k})
    -
    \dim_{k_p}(\pi_1(\bL_{k_p/k})).
\]
The first equality is a standard result in dimension theory. The second is
known as the Cartier equality, see \SPcite{07E1} or \cite[Thm.~26.10]{Mat89}.
When $k$ is perfect, any field extension $k_p/k$ satisfies
$\pi_{1}(\bL_{k_p/k}) = 0$. In general, the module $\pi_{1}(\bL_{k_p/k})$ is
called the \emph{module of imperfections} of the field extension $k_p/k$. It is
worth noting that for any field extension $k_p/k$ we have
$\pi_i(\bL_{k_p/k})=0$ for all $i \ge 2$, see \cite[8.10]{Iye07} or
\SPcite{07BV}.

\subsection{Setup for arcs and jets}

Our goal in this section is to control embedding dimensions on arc spaces, and
we start by fixing some notation. We let $k$ be a field, and assume $X$ is a
classical $k$-scheme which is essentially of finite type over $k$. In
particular, notice that $\bL_{X/k}$ is pseudo-coherent in the sense discussed
in \cref{sec:csi}.  We fix a point in the arc space $\alpha \in J_\infty(X)$
with residue field $k_\alpha$ and consider its truncations $\alpha_n \in
J_n(X)$. We continue to use the terminology introduced in
\cref{sec:fibers-fitting}. For instance, we write $b_i(\alpha)$ for the Betti
numbers of $\bL_{X/k}$ with respect to $\alpha$; for $i=0$ this means that
\[
    \Omega_{X/k} \otimes_{\cO_X} k_\alpha\llbracket t \rrbracket
    \cong
    k_\alpha\llbracket t \rrbracket^{b_0(\alpha)}
    \oplus (\text{$t$-torsion}),
\]
or equivalently that
\[
    b_0(\alpha) =
    \dim_{k_{\alpha(\eta)}}(\Omega_{X/k} \otimes_{\cO_X} k_{\alpha(\eta)}).
\]
Here $\eta$ denotes the generic point of $\Spec(k_\alpha\llbracket t
\rrbracket)$, and $\alpha(\eta) \in X$ denotes the image of $\eta$ via the
induced map $\alpha \colon \Spec(k_\alpha\llbracket t \rrbracket) \to X$. As
usual, $k_{\alpha(\eta)}$ denotes the residue field of $\alpha(\eta)$ viewed as
point of $X$. If $\alpha$ is non-degenerate we have $b_0(\alpha) = d$, where
$d$ is the dimension of $X$ at $\alpha(\eta)$, and $b_i(\alpha) = 0$ for $i>0$.
We also denote by $b_i(\alpha_n)$ the Betti numbers with respect to the
truncations, and recall that $b_i(\alpha_n) = b_i(\alpha)$ when $n \ge
\ord_\alpha(\Jac_X^{(i, b_i(\alpha))})$.

\subsection{Embedding dimension on derived jet schemes}

We consider the exact triangle induced by the composition
$\{\alpha_n\} \to \bL J_n(X) \to \Spec(k)$:
\[
    \bL_{\bL J_n(X)/k}
    \otimes_{\cO_{\bL J_n(X)}}^\bL
    k_{\alpha_n}
    \lra
    \bL_{k_{\alpha_n}/k}
    \lra
    \bL_{k_{\alpha_n}/\bL J_n(X)}
    \xra{-1}.
\]
Keeping all aspects of notation will get oppressively heavy. As such we
introduce some simplifications. We write $k_n$ instead of $k_{\alpha_n}$ and
$\bL J_n$ instead of $\bL J_n(X)$. We also suppress the bases of tensor
products unless they are not clear from context. With this simplified notation,
the above exact triangle can be written more compactly as follows:
\[
    \bL_{\bL J_n/k}
    \Lotimes
    k_n
    \lra
    \bL_{k_n/k}
    \lra
    \bL_{k_n/\bL J_n}
    \xra{-1}.
\]
We focus on the initial part of the resulting long exact sequence, which looks
as follows:
\begin{equation}
\Label[diagram]{eq:six-term-exact-seq}
\begin{tikzcd}[
    row sep=0,
    execute at begin node={\rule[-2.2ex]{0mm}{5.1ex}},
    execute at begin picture={
        \clip (-6.6,-1.45) rectangle (6.75,1.45);
    },
]
    \cdots \rar
    & 0 \rar
    & \pi_2(\bL_{k_n/\bL J_n})
\\
    \pi_1(\bL_{\bL J_n/k} \Lotimes k_n) \rar
    & \pi_1(\bL_{k_n/k}) \rar
    & \pi_1(\bL_{k_n/\bL J_n})
\\
    \Omega_{J_n/k} \otimes k_n \rar
    & \Omega_{k_n/k} \rar
    & 0.
    \ar[phantom, from=1-2, to=2-2,
        to path={
            (\tikztostart.center)
            --
            node[name=Y, anchor=center]{}
            (\tikztotarget.center)
        }]{d}
    \ar[rounded corners=1.6ex, from=1-3, to=2-1,
        to path={ -- ([xshift=2.3cm]\tikztostart.center)
                |- (Y.center) \tikztonodes
                -| ([xshift=-3cm]\tikztotarget.center)
                -- (\tikztotarget)
        }]{dll}
    \ar[phantom, from=2-2, to=3-2,
        to path={
            (\tikztostart.center)
            --
            node[name=X, anchor=center]{}
            (\tikztotarget.center)
        }]{d}
    \ar[rounded corners=1.6ex, from=2-3, to=3-1,
        to path={ -- ([xshift=2.3cm]\tikztostart.center)
                |- (X.center) \tikztonodes
                -| ([xshift=-3cm]\tikztotarget.center)
                -- (\tikztotarget)
        }]{dll}
\end{tikzcd}
\end{equation}
For convenience, we introduce notation for the following kernel:
\[
    K_{n}
    :=
    \ker\left(
        \pi_{1}(\bL_{k_n/k})
        \lra
        \pi_1(\bL_{k_n/\bL J_n})
    \right).
\]
We remark that when the field extension $k_n/k$ is separable, we have that $K_n
= 0$. From \cref{eq:six-term-exact-seq} we can alternatively write
\[
    K_{n}
    \cong
    \coker\big(
        \pi_2(\bL_{k_n/\bL J_n})
        \lra
        \pi_1(\bL_{\bL J_n/k} \Lotimes k_n)
    \big).
\]
In particular we also have that $K_n = 0$ when $X$ is smooth at $\alpha(0)$, by
\cref{thm:smoothobjectshavediscretejetsandarcs}.

\begin{lem}
\label{lem:embdimofjets}
Denote $d = b_0(\alpha)$ and assume $n \ge \ord_\alpha(\Jac_X^{(0,d)})$. Then
\[
    \embdim(\cO_{J_n(X),\alpha_n})
    =
    (n+1) d
    - \dim(\overline{\{\alpha_{n}\}})
    + \ord_\alpha(\Jac_X^{(0,d)})
    - \dim_{k_{\alpha_n}}(K_{n}).
\]
\end{lem}

\begin{proof}
The result follows immediately from the long exact sequence
\eqref{eq:six-term-exact-seq}
and the
equalities
\begin{gather*}
    \dim_{k_n}(\Omega_{J_n/k} \otimes k_n) =
        (n+1) d
        + \ord_\alpha(\Jac_X^{(0,d)}),
\\
    \dim(\overline{\{\alpha_{n}\}}) =
        \dim_{k_n}(\Omega_{k_n/k})
        - \dim_{k_n}(\pi_1(\bL_{k_n/k})),
\\
    \embdim(\cO_{J_n(X),\alpha_n}) =
        \dim_{k_n}(\pi_1(\bL_{k_n/\bL J_n})),
\end{gather*}
which were discussed in \cref{thm:cotangent-fibers-general},
\cref{sec:cartier-equality}, and \cref{prop:embdim-via-pi1}.
\end{proof}

\begin{rmk}
We emphasize that when $X$ is smooth or when the field extension
$k_{\alpha_n}/k$ is separable, the term $\dim_{k_{\alpha_n}}(K_n)$ in
\cref{lem:embdimofjets} vanishes. In the smooth case the embedding dimension
becomes simply $(n+1) d - \dim(\overline{\{\alpha_{n}\}})$. When
$k_{\alpha_n}/k$ is separable, we get an exact formula for the embedding
dimension in terms of the order of contact with respect to the Jacobian ideal
of level $(0,d)$. This always happens for example when the base field $k$ is
perfect. When $k$ is not perfect, the extra term might be necessary, giving
evidence of the delicate interplay between Jacobian ideals and the dependence
on $k$ of the construction of the jet schemes.
\end{rmk}

The size of $K_n$ can be easily bounded using the techniques of
\cref{sec:fibers-fitting}. We present a result for non-degenerate arcs.

\begin{cor}
\label{lem:embdimofjets-non-deg}
Consider a non-degenerate arc $\alpha$, denote $d = b_0(\alpha)$ and assume
that $n \ge \ord_\alpha(\Jac_X^{(1,0)})$. Then
\[
\begin{aligned}
\MoveEqLeft
    (n+1) d
    - \dim(\overline{\{\alpha_{n}\}})
    + \ord_\alpha(\Jac_X^{(0,d)})
    - \ord_\alpha(\Jac_X^{(1,0)})
\\&
    \le
    \embdim(\cO_{J_n(X),\alpha_n})
\\&
    \le
    (n+1) d
    - \dim(\overline{\{\alpha_{n}\}})
    + \ord_\alpha(\Jac_X^{(0,d)}).
\end{aligned}
\]
\end{cor}

\begin{proof}
\cref{cor:cotangent-fibers-non-deg} gives the last
equality in the following equation:
\[
    0
    \le
    \dim_{k_n}
    (K_n)
    \le
    \dim_{k_n}
    (
        \pi_1(
            \bL_{\bL J_n/k}
            \Lotimes
            k_n
        )
    )
    =
    \ord_\alpha(
        \Jac_X^{(1,0)}
    ).
    \qedhere
\]
\end{proof}

\begin{rmk}
Notice that we always have $ \ord_\alpha(\Jac_X^{(0,d)}) \le
\ord_\alpha(\Jac_X^{(1,0)}) $. Indeed, as we saw in
\cref{sec:pull-backs-complexes-arcs}, for a non-degenerate arc we have the
equality
\[
    a_{1,1}(\alpha) + \cdots + a_{1,c_1}(\alpha)
    =
    \ord_\alpha(\Jac_X^{(1,0)})
    -
    \ord_\alpha(\Jac_X^{(0,d)}).
\]
We also have
\[
    a_{0,1}(\alpha) + \cdots + a_{0,c_0}(\alpha)
    =
    \ord_\alpha(\Jac_X^{(0,d)}),
\]
so the bounds of \cref{lem:embdimofjets-non-deg} could be rewritten in terms of
the invariant factors of $\alpha$.
\end{rmk}

\subsection{Analysis of truncation maps}

For $m>n$, the truncation map $\bL J_{n}(X)\to\bL J_{m}(X)$ induces a ladder of
two triangles of $k_\alpha$-complexes
\begin{equation}
\Label[diagram]{eq:ladder}
\begin{tikzcd}[column sep=small]
    \bL_{\bL J_{n}/k}
    \Lotimes
    k_\alpha
        \arrow[r]
        \arrow[d]
    & \bL_{k_n/k} \otimes k_\alpha
        \arrow[r]
        \arrow[d]
    & \bL_{k_n/\bL J_{n}} \otimes k_\alpha
        \arrow[r,"-1"]
        \arrow[d]
    & \phantom{x}
\\
    \bL_{\bL J_{m}/k}
    \Lotimes
    k_\alpha
        \arrow[r]
    & \bL_{k_m/k} \otimes k_\alpha
        \arrow[r]
    & \bL_{k_m/\bL J_{m}} \otimes k_\alpha
        \arrow[r,"-1"]
    & .
\end{tikzcd}
\end{equation}
Note that since $k_{n}$ and $k_{m}$ are fields, the derived tensor over them is
an ordinary tensor and commutes with homotopy. We indicated this in the last
two columns above by suppressing the $\bL$ on top of the tensor product. For
the first column, we have
\[
    \bL_{\bL J_{n}/k}
    \Lotimes
    k_\alpha
    =
    \bL_{\bL J_{n}/k}
    \Lotimes
    k_n
    \otimes
    k_\alpha.
\]
In what follows, write $\lambda_{n,m}$ for the map
\[
    \lambda_{n,m}
    \colon
    \pi_1(\bL_{k_n/\bL J_n}) \otimes k_\alpha
    \lra
    \pi_1(\bL_{k_m/\bL J_m}) \otimes k_\alpha,
\]
and $\lambda_{n}$ for
\[
    \lambda_n
    \colon
    \pi_1(\bL_{k_n/\bL J_n}) \otimes k_\alpha
    \lra
    \pi_1(\bL_{k_\alpha/\bL J_\infty}).
\]
In view of \cref{prop:embdim-via-pi1}, the map $\lambda_{n,m}$ agrees with
$\fm_{\alpha_n}/\fm_{\alpha_n}^{2}\otimes k_\alpha \to
\fm_{\alpha_m}/\fm_{\alpha_m}^{2}\otimes k_\alpha$, the map induced by the
truncation map on the Zariski cotangent spaces. Since the arc space is the
limit of the jet schemes, it immediately follows that
\[
    \bL_{k_\alpha/\bL J_\infty}
    =
    \hocolim_n \, (
    \bL_{k_n/\bL J_n} \otimes k_\alpha ).
\]
Furthermore, since all of these complexes have vanishing $\pi_0$, we also get
that
\[
    \fm_\alpha/\fm_\alpha^2 =
    \pi_1(\bL_{k_\alpha/\bL J_\infty})
    =
    \colim_n \, (
    \pi_1(\bL_{k_n/\bL J_n}) \otimes k_\alpha )
    =
    \colim_n \, (
    \fm_{\alpha_n}/\fm_{\alpha_n}^2 \otimes k_\alpha ).
\]

\begin{lem}
\label{lem:boundonimagelambdan}
Let $d = b_0(\alpha)$ and assume $n \ge \ord_\alpha(\Jac_X^{(0,d)})$ and $m\geq
n+\ord_\alpha(\Jac_X^{(0,d)})$. Under the notation above, we have
\begin{equation}
    \label{eq:dimofimageinequality}
    \dim_{k_\alpha}(\Im(\lambda_{n,m}))
    \geq
    (n+1) \, d
    - \dim(\overline{\{\alpha_{n}\}})
    - \dim_{k_n}(K_{n}).
\end{equation}
Moreover, if the map $\Omega_{k_{n}/k}\otimes
k_\alpha\to\Omega_{k_{m}/k}\otimes k_\alpha$ is injective, then
\begin{equation}
    \label{eq:dimofimageequality}
    \dim_{k_\alpha}(\Im(\lambda_{n,m}))
    =
    (n+1) \, d
    - \dim(\overline{\{\alpha_{n}\}})
    - \dim_{k_n}(K_{n}).
\end{equation}
\end{lem}

\begin{rmk}
The injectivity of $\Omega_{k_{n}/k}\otimes k_\alpha\to\Omega_{k_{m}/k}\otimes
k_\alpha$ is guaranteed as soon as the field extension $k_m/k_n$ is separable.
In particular, it holds for arbitrary arcs in characteristic zero, and for
stable arcs and $n$ large enough in arbitrary characteristic (see
\cref{sec:curveselection}).
\end{rmk}

\begin{rmk}
\Cref{eq:dimofimageinequality} is an extension of \cite[Lemma~8.3]{dFD20} to
the non-perfect case. When $\Omega_{k_{n}/k}\otimes k_\alpha \to
\Omega_{k_{m}/k} \otimes k_\alpha$ is injective and $K_n = 0$ we get the
stronger identity
\[
    \dim_{k_\alpha}(\Im(\lambda_{n,m}))
    =
    (n+1) \, d
    - \dim(\overline{\{\alpha_{n}\}}).
\]
For example, this always holds in characteristic zero. Recall that the
condition $K_n = 0$ is satisfied when $X$ is smooth or $k$ is perfect.
\end{rmk}

\begin{proof}
From the two triangles we produce the two truncated long exact sequences in the
bottom two rows below. The top nonzero row in the diagram below consists of the
kernels of the vertical maps induced by the ladder of \cref{eq:ladder}.
\[
\begin{tikzcd}[column sep=0.65em]
    &
    & 0 \arrow[d]
    & 0 \arrow[d]
    & 0 \arrow[d]
    & 0 \arrow[d]
    &
\\[-0.2cm]
    &
    & \widetilde{K}_{n,m}
        \arrow[r]
        \arrow[d]
    & K'_{n,m}
        \arrow[r,"\iota_{n,m}"]
        \arrow[d]
    & K_{n,m}
        \arrow[r]
        \arrow[d]
    & K''_{n,m}
        \arrow[d]
    &
\\[-0.2cm]
    0 \arrow[r]
    &
    K_{n}\otimes k_\alpha
        \arrow[r]
        \arrow[d]
    & \pi_{1}(\bL_{k_{n}/k}\otimes k_\alpha)
        \arrow[r]
        \arrow[d]
    & \pi_1(\bL_{k_n/\bL J_n}) \otimes k_\alpha
        \arrow[r]
        \arrow[d,"\lambda_{n,m}"]
    & \Omega_{J_{n}/k}\otimes k_\alpha
        \arrow[r]
        \arrow[d,"\pi_0(\psi_{n,m})"]
    & \Omega_{k_{n}/k}\otimes k_\alpha
        \arrow[r]
        \arrow[d]
    & 0
\\
    0 \arrow[r]
    &
    K_{m}\otimes k_\alpha
        \arrow[r]
    & \pi_{1}(\bL_{k_{m}/k}\otimes k_\alpha)
        \arrow[r]
    & \pi_1(\bL_{k_m/\bL J_m}) \otimes k_\alpha
        \arrow[r]
    & \Omega_{J_{m}/k}\otimes k_\alpha
        \arrow[r]
    & \Omega_{k_{m}/k}\otimes k_\alpha
        \arrow[r]
    & 0
\end{tikzcd}
\]
First, we claim $\widetilde{K}_{n,m}=0$. Consider the triangle of cotangent
complexes induced by the composition $k\to k_{n}\to k_{m}$, base changed to
$k_\alpha$:
\begin{align*}
    \bL_{k_{n}/k}\otimes k_\alpha
    \lra
    \bL_{k_{m}/k}\otimes k_\alpha
    \lra
    \bL_{k_{m}/k_{n}}\otimes k_\kappa
    \xra{-1}.
\end{align*}
Since this triangle is induced by the truncation map $\bL J_{n}(X)\to\bL
J_{m}(X)$, it agrees with the middle vertical map in (\ref{eq:ladder}).
Consequently, $\widetilde{K}_{n,m}$ is a quotient of
$\pi_{2}(\bL_{k_{m}/k_{n}}\otimes k_\alpha)$. But
$\pi_{2}(\bL_{k_{m}/k_{n}}\otimes k_\alpha)=0$ since $k_{n}\to k_{m}$ is a
morphism of fields \cite[8.10]{Iye07}. Next, by \Cref{lem:embdimofjets},
\[
    \embdim(\cO_{\bL J_n(X),\alpha_n})
    =
    (n+1) d
    - \dim(\overline{\{\alpha_{n}\}})
    + \ord_\alpha(\Jac_X^{(0,d)})
    - \dim_{k_{\alpha_n}}(K_{n}).
\]
By \cref{thm:trunc-cotangent-fibers-general},
\[
    \dim_{k_\alpha} (K_{n,m})
    =
    \dim_{k_{\alpha_m}}(\ker(\pi_0(\psi_{n,m})))
    =
    \ord_\alpha(\Jac_X^{(0,d)})
    .
\]
Additionally, observe that since $\widetilde{K}_{n,m}=0$,
\[
    \dim_{k_\alpha}(K'_{n,m})
    =
    \dim_{k_\alpha}(\Im(\iota_{n,m})).
\]
Combining the above statements, one has
\begin{align*}
    \dim_{k_\alpha}(\Im(\lambda_{n,m}))
    &= \dim_{k_\alpha}(\pi_1(\bL_{k_n/\bL J_n}) \otimes k_\alpha)
        - \dim_{k_\alpha}K'_{n,m} \\
    &= \embdim(\cO_{\bL J_n(X),\alpha_n})
        - \dim_{k_\alpha}(\Im(\iota_{n,m})) \\
    &\geq \embdim(\cO_{\bL J_n(X),\alpha_n})
        - \dim_{k_\alpha}(K_{n,m}) \\
    &= (n+1) \, d
        - \dim(\overline{\{\alpha_{n}\}})
        - \dim_{k_{\alpha_n}}(K_{n}).
\end{align*}
This gives the inequality (\ref{eq:dimofimageinequality}). For the equality
statement (\ref{eq:dimofimageequality}), it is enough to observe that if the
map $\Omega_{k_{n}/k}\otimes k_\alpha\to\Omega_{k_{m}/k}\otimes k_\alpha$ is
injective, then $K''_{n,m}=0$ and $\Im(\iota_{n,m})=K_{n,m}$.
\end{proof}

Using the bounds on the size of $K_n$ obtained in \cref{sec:fibers-fitting} we
get the following estimate for non-degenerate arcs.

\begin{cor}
\label{lem:boundonimagelambdan-non-deg}
Consider a non-degenerate arc $\alpha$, denote $d = b_0(\alpha)$ and assume
that $n \ge \ord_\alpha(\Jac_X^{(1,0)})$ and $m\geq
n+\ord_\alpha(\Jac_X^{(0,d)})$. Then
\[
    - \ord_\alpha(\Jac_X^{(1,0)})
    \ \le\
    \dim_{k_\alpha}(\Im(\lambda_{n,m}))
    - (n+1) d
    + \dim(\overline{\{\alpha_{n}\}})
    \ \le\
    \ord_\alpha(\Jac_X^{(0,d)}).
\]
Moreover, if the map $\Omega_{k_{n}/k}\otimes
k_\alpha\to\Omega_{k_{m}/k}\otimes k_\alpha$ is injective, then
\[
    - \ord_\alpha(\Jac_X^{(1,0)})
    \ \le\
    \dim_{k_\alpha}(\Im(\lambda_{n,m}))
    - (n+1) d
    + \dim(\overline{\{\alpha_{n}\}})
    \ \le\
    0.
\]
\end{cor}

\begin{proof}
The upper bounds are a consequence of \cref{lem:embdimofjets-non-deg} and the
obvious inequality $\dim_{k_\alpha}(\Im(\lambda_{n,m})) \le
\embdim(\cO_{J_n(X),\alpha_n})$. For the lower bounds, notice that $K_n$ is a
quotient of $\pi_1( \bL_{\bL J_n/k} \Lotimes k_n)$, which has dimension
$\ord_\alpha(\Jac_X^{(1,0)})$.
\end{proof}

\begin{rmk}
\label{lem:boundonimagelambdan-non-deg-rmk}
When the ground field is perfect, we can replace $-\ord_\alpha(\Jac_X^{(1,0)})$
by $0$ in the bounds of \cref{lem:boundonimagelambdan-non-deg}.
\end{rmk}

\subsection{Embedding dimension on arc spaces}

Next, we explore bounds on the embedding dimension of the arc $\alpha$ itself.

\begin{lem}
\label{lem:embdimlimit}
For $d = b_0(\alpha)$,
\begin{equation}
    \label{eq:embdimlimit-1}
    \embdim(\cO_{J_\infty(X),\alpha})
    \ge
    \limsup_{n\to\infty}\left(
        (n+1) \, d
        - \dim(\overline{\{\alpha_{n}\}})
        -\dim_{k_{n}}(K_{n})
    \right).
\end{equation}
If $k_{m}/k_{n}$ is a separable extension for $m \geq n$
large enough, then
\begin{equation}
    \label{eq:embdimlimit-2}
    \embdim(\cO_{J_\infty(X),\alpha})
    =
    \lim_{n\to\infty}\left(
        (n+1) \, d
        - \dim(\overline{\{\alpha_{n}\}})
        -\dim_{k_{n}}(K_{n})
    \right).
\end{equation}
If $k$ is perfect and $k_{m}/k_{n}$ is a separable extension for $m \geq n$
large enough, or if $X$ is smooth at $\alpha(0)$, then
\begin{equation}
    \label{eq:embdimlimit-3}
    \embdim(\cO_{J_\infty(X),\alpha})
    =
    \lim_{n\to\infty}\left(
        (n+1) \, d
        - \dim(\overline{\{\alpha_{n}\}})
    \right).
\end{equation}
\end{lem}

\begin{proof}
The result follows from the identification $ \pi_1(\bL_{k_\alpha/\bL J_\infty})
= \colim_n \, ( \pi_1(\bL_{k_n/\bL J_n}) \otimes k_\alpha ) $ and
\cref{lem:boundonimagelambdan}. Notice when $X$ is smooth at $\alpha(0)$ we can
apply \cref{lem:boundonimagelambdan-non-deg}.
\end{proof}

\begin{rmk}
The quantity appearing on the right-hand-side of \cref{eq:embdimlimit-3} is
known as the \emph{jet codimension} and will be studied in depth in
\cref{sec:curveselection}. Notice that \cref{lem:embdimlimit}
\eqref{eq:embdimlimit-3} holds in characteristic zero. We will show in
\cref{thm:edim-equals-jetcodim} that the equality \eqref{eq:embdimlimit-3}
holds unconditionally for all arcs in all characteristics when $X$ is
generically smooth.
\end{rmk}

\section{Cotangent maps}

\label{sec:cotangent_maps}

This section is devoted to using cotangent complexes on jets and arc spaces to
provide a series of improvements to previously known results regarding
cotangent maps. Specifically we extend many of the results of Chiu, de Fernex,
and Docampo in \cite{dFD20,CdFD22,CdFD23} by weakening some hypotheses. It
should be noted that while the proofs are often similar to those of
\cite{dFD20,CdFD22,CdFD23}, they differ in one significant way: previous works
did not have access to a specific exact $3 \by 3$ grid in a derived category,
see \cref{eq:3_times_3_grid_of_complexes}. From this one extracts an infinite
grid of homotopy groups with an interesting periodicity. When we can determine
the value of a term, this periodicity allows us to determine the value of terms
which lie far away as well, leading to our results.

\subsection{Cotangent and higher cotangent maps}

\label{subsec:cotangent_maps_assumptions}

Throughout the section we fix an arbitrary ground field $k$. Let $f\colon X\to
Y$ be a morphism of classical $k$-schemes and consider the corresponding map
$f_\infty = J_\infty(f) \colon J_\infty(X)\to J_\infty(Y)$ of arc spaces. Pick
an arc $\alpha\in J_\infty (X)$ and set $\beta=f_\infty(\alpha)\in
J_\infty(Y)$. We write $k_\alpha$ and $k_\beta$ for the residue fields of
$\alpha$ and $\beta$ considered as points of the corresponding arc spaces. Let
$\fm_\alpha$ and $\fm_\beta$ be the defining ideals of $\alpha$ and $\beta$.
The \emph{cotangent map} of $f_\infty$ at $\alpha$, denoted
$T_\alpha^{*}f_\infty$, is the natural map
\[
    T_\alpha^{*}f_\infty
    \colon
    \fm_\beta/\fm_\beta^{2}\otimes_{k_\beta}k_\alpha
    \lra
    \fm_\alpha/\fm_\alpha^{2}
\]
induced by $f_\infty$. Recalling that $\fm_\alpha/\fm_\alpha^2 \cong
\pi_{1}(\bL_{k_\alpha/\bL J_\infty (X)})$, it is natural to also consider the
following homomorphisms:
\[
    \pi_{i}(\bL_{k_\beta/\bL J_\infty(Y)}\otimes_{k_\beta}^\bL k_\alpha)
    \lra
    \pi_{i}(\bL_{k_\alpha/\bL J_\infty(X)}).
\]
We call them the \emph{higher cotangent maps} of $\bL J_\infty(f)$ at $\alpha$.

\subsection{The exact \(3\by 3\) grid}

\label{subsec:higher-cotangent-maps}

We consider the following $3 \by 3$ grid of complexes in the derived category
of $k_\alpha$, where the rows and columns are exact triangles.
\begin{equation}
\Label[diagram]{eq:3_times_3_grid_of_complexes}
    \begin{tikzcd}[column sep=small, row sep=small]
        \bL_{\bL J_\infty (Y)/k}\otimes_{\cO_{\bL J_\infty (Y)}}^\bL k_\alpha
        \arrow[r] \arrow[d]
    &
        \bL_{k_\beta/k}\otimes_{k_\beta}^\bL k_\alpha
        \arrow[r] \arrow[d]
    &
        \bL_{k_\beta/\bL J_\infty(Y)}\otimes_{k_\beta}^\bL k_\alpha
        \arrow[r,"-1"] \arrow[d]
    &
        \phantom{x}
    \\
        \bL_{\bL J_\infty (X)/k}\otimes_{\cO_{\bL J_\infty (X)}}^\bL k_\alpha
        \arrow[r] \arrow[d]
    &
        \bL_{k_\alpha/k}
        \arrow[r] \arrow[d]
    &
        \bL_{k_\alpha/\bL J_\infty(X)}
        \arrow[r,"-1"] \arrow[d]
    &
        \phantom{x}
    \\
        \bL_{\bL J_\infty (X)/\bL J_\infty (Y)}
        \otimes_{\cO_{\bL J_\infty (X)}}^\bL
        k_\alpha
        \arrow[r] \arrow[d,"-1"]
    &
        \bL_{k_\alpha/k_\beta}
        \arrow[r] \arrow[d,"-1"]
    &
        C_\sbt\arrow[r,"-1"]
        \arrow[d,"-1"]
    &
        \phantom{x}
    \\
        \phantom{x}
    &
        \phantom{x}
    &
        \phantom{x}
    \\[-1.5em]
    \end{tikzcd}
\end{equation}
The first two rows come from the compositions $k\to\cO_{\bL J_\infty(Y)}\to
k_\beta$ and $k\to\cO_{\bL J_\infty(X)}\to k_\alpha$, while the first two
columns come from the compositions $k\to\cO_{\bL J_\infty(Y)}\to\cO_{\bL
J_\infty(X)}$ and $k\to k_\beta\to k_\alpha$. Any exact triangles which are not
complexes of $k_\alpha$-modules are subsequently base changed to $k_\alpha$.
This leaves the third row and third column, and the ninth $k_\alpha$-complex.
Given the situation described so far, the $3 \by 3$ lemma of triangulated
categories \cite[Lem.~1.7]{May05} shows that there is an isomorphism between
the mapping cones
\[
    \cone(
        \bL_{k_\beta/\bL J_\infty(Y)}\otimes_{k_\beta}^\bL k_\alpha
        \to
        \bL_{k_\alpha/\bL J_\infty(X)}
    )
    \textrm{ and }
    \cone(
        \bL_{\bL J_\infty (X)/\bL J_\infty (Y)} \otimes_{\cO_{\bL J_\infty (X)}}^\bL k_\alpha
        \to
        \bL_{k_\alpha/k_\beta}
    ).
\]
In this way we get the ninth complex, which we have denoted simply by $C_\sbt$,
and the morphisms producing the third row and third column.

Morally, $C_\sbt$ plays the role a ``relative cotangent complex'' of the
morphism $k_\alpha/\bL J_\infty(X)$ over the morphism $k_\beta/\bL
J_\infty(Y)$. The $3 \by 3$ lemma tells us that $C_\sbt$ is also a ``relative
cotangent complex'' of $k_\alpha/k_\beta$ over $\bL J_\infty(X)/\bL
J_\infty(Y)$, and \cref{eq:3_times_3_grid_of_complexes} should be seen as the
analogue of the fundamental triangle for cotangent complexes in this relative
setting.

The cotangent maps and higher cotangent maps appear in connection to the third
column of the $3 \by 3$ grid of \cref{eq:3_times_3_grid_of_complexes}. More
precisely, the long exact sequence associated to this third column looks as
follows:
\[
    \pi_{i+1}(C_\sbt)
    \lra
    \pi_{i}(\bL_{k_\beta/\bL J_\infty(Y)}\otimes_{k_\beta}^\bL k_\alpha)
    \lra
    \pi_{i}(\bL_{k_\alpha/\bL J_\infty(X)})
    \lra
    \pi_{i}(C_\sbt).
\]
Notice that the middle morphism above is what we call a higher cotangent
map, and that in this sequence the homotopy groups vanish when $i \le 0$. The
usual cotangent map appears when taking $i=1$:
\[
    \pi_{2}(C_\sbt)
    \lra
    \fm_\beta/\fm_\beta^{2}\otimes_{k_\beta}k_\alpha
    \lra[\quad T_\alpha^{*}f_\infty \quad]
    \fm_\alpha/\fm_\alpha^{2}
    \lra
    \pi_{1}(C_\sbt).
\]

An unwinding of the long exact sequences corresponding to
\cref{eq:3_times_3_grid_of_complexes} produces an infinite two-dimensional grid
of homotopy groups, notably with each specific group appearing in multiple
locations. This is a feature that can be exploited to our advantage. In
particular, if a diagram chase forces some particular homotopy group to be
isomorphic to $0$, then of course it must be isomorphic to $0$ everywhere else
it appears, but notably in locations where its neighbors differ from those of
the original diagram chase. This allows results in one location to affect
behavior of groups in many other locations and is the key ingredient permitting
the calculations on the higher cotangent maps discussed below. We explain this
in \cref{subsec:entanglement} after developing a bit more notation.

\subsection{Free plus torsion decomposition}

\label{subsec:free-plus-torsion-decomp}

\Cref{prop:picotangent} gives information about the first column of the $3 \by
3$ grid in \cref{eq:3_times_3_grid_of_complexes}. Namely, the corresponding
homotopy groups only depend on $\bL_{X/k}$ and $\bL_{Y/k}$, and they have the
structure of $k_\alpha\llbracket t \rrbracket$-modules (and not just
$k_\alpha$-vector spaces). It will be convenient to keep track of this
information, as follows.

We write $\pi_i(\bL_{X/k} \otimes_{\cO_X} k_\alpha\llbracket t \rrbracket)
\cong F_X^{i} \oplus T_X^{i}$, where $F_X^{i}$ is a free $k_\alpha\llbracket
t\rrbracket$-module and $T_X^{i}$ is a torsion $k_\alpha\llbracket t
\rrbracket$-module. \Cref{prop:picotangent} describes $F_X^{i}$ and $T_X^{i}$ in
terms of the Betti numbers and the invariant factors of $\bL_{X/k}$ with
respect to $\alpha$, and shows that
\[
    \pi_{i}(
        \bL_{\bL J_\infty (X)/k}
        \otimes_{\cO_{\bL J_\infty (X)}}^\bL
        k_\alpha
    )
    \cong
    (F_X^{i} \otimes_{k_\alpha\llbracket t \rrbracket} V_\infty(k_\alpha))
    \oplus
    T_X^{i-1}.
\]
To lighten the notation we introduce $G_X^{i} = F_X^{i}
\otimes_{k_\alpha\llbracket t \rrbracket} V_\infty(k_\alpha)$. We similarly
consider $F_Y^{i}$, $G_Y^{i}$, $T_Y^{i}$, $F_{X/Y}^{i}$, $G_{X/Y}^{i}$, and
$T_{X/Y}^{i}$ via $\pi_i(\bL_{Y/k} \otimes_{\cO_Y} k_\alpha\llbracket t
\rrbracket) \cong F_Y^{i} \oplus T_Y^{i}$ and $\pi_i(\bL_{X/Y} \otimes_{\cO_X}
k_\alpha\llbracket t \rrbracket) \cong F_{X/Y}^{i} \oplus T_{X/Y}^{i}$.
\Cref{prop:picotangent} and \cref{cor:deriveddFDtriangle} give:
\begin{align*}
    \pi_{i}(
        \bL_{\bL J_\infty (X)/k}
        \otimes_{\cO_{\bL J_\infty (X)}}^\bL
        k_\alpha
    )
    &\cong
    G_X^{i}
    \oplus
    T_X^{i-1}
\\
    \pi_{i}(
        \bL_{\bL J_\infty (Y)/k}
        \otimes_{\cO_{\bL J_\infty (Y)}}^\bL
        k_\alpha
    )
    & \cong
    G_Y^{i}
    \oplus
    T_Y^{i-1},
\\
    \pi_{i}(
        \bL_{\bL J_\infty (X)/\bL J_\infty (Y)}
        \otimes_{\cO_{\bL J_\infty (X)}}^\bL
        k_\alpha
    )
    & \cong
    G_{X/Y}^{i}
    \oplus
    T_{X/Y}^{i-1}.
\end{align*}
Observe that while the $G_{\Arg}^i$ are constructed from the free modules
$F_{\Arg}^i$, they are not themselves free. Instead, they are direct sum of
finitely many copies of $V_\infty(k_\alpha)$. Also note that $T_{X}^{-1}$,
$T_{Y}^{-1}$, and $T_{X/Y}^{-1}$ are all $0$, which is relevant in the above
three formulas in the case of $i = 0$.

\subsection{Entanglement}

\label{subsec:entanglement}

Most of the results in this sections will follow from analyzing a particular
portion of the infinite grid obtained by taking homotopy of the $3 \by 3$ grid
of \cref{eq:3_times_3_grid_of_complexes}. Specifically, we focus on the entry
$\pi_0(\bL_{k_\alpha/k_\beta}) = \Omega_{k_\alpha/k_\beta}$ and observe that we
have vanishing of all the other entries that are located either below or to the
right.
In \cref{eq:diagram_needed_to_prove_multiple_theorems_about_cotangent_map},
appearing below, we write the most relevant part of the resulting diagram. We
have used the discussion and notation from
\cref{subsec:higher-cotangent-maps,subsec:free-plus-torsion-decomp} to rewrite
some of the homotopy groups. Notice that the cotangent map $T^*_\alpha
f_\infty$ appears in the second column.
Additionally, note that the entry in the upper-left corner of
\cref{eq:diagram_needed_to_prove_multiple_theorems_about_cotangent_map} must be
$0$, because $k_\alpha/k_\beta$ is an extension of fields and thus
$\pi_{2}(\bL_{k_\alpha/k_\beta})=0$ \cite[8.10]{Iye07}.

We would like to emphasize a key fact: any specific term of the infinite grid
of homotopy groups appears multiple times in distinct locations. Because of
this, when we can determine the value of a term using one part of the diagram,
we can apply these results in other sections lying far away as well. For
example, the bottom-left term $\pi_{1}(\bL_{k_\alpha/k_\beta})$ in
\cref{eq:diagram_needed_to_prove_multiple_theorems_about_cotangent_map} appears
again in the top-right position. Thus if a diagram chase forces
$\pi_{1}(\bL_{k_\alpha/k_\beta})$ to be $0$ somewhere, it is $0$ everywhere. We
refer to this behavior as an \emph{entanglement}. The entanglement of
$\pi_{1}(\bL_{k_\alpha/k_\beta})$ is featured multiple times in the proofs of
this section.

\begin{equation}
\Label[diagram]{eq:diagram_needed_to_prove_multiple_theorems_about_cotangent_map}
\begin{tikzcd}[row sep=small, column sep=small]
    0 \arrow[r] \arrow[d]
    & \pi_2(C_\sbt) \arrow[r] \arrow[d]
    & G_{X/Y}^1 \oplus T_{X/Y}^0 \arrow[r] \arrow[d]
    & \pi_{1}(\bL_{k_\alpha/k_\beta}) \arrow[d]
\\
    \pi_{1}(\bL_{k_\beta/k}\otimes^\bL k_\alpha) \arrow[r] \arrow[d]
    & \fm_\beta/\fm_\beta^2 \otimes k_\alpha
        \arrow[r] \arrow[d,"T_\alpha^{*}f_\infty"]
    & G_{Y}^0 \arrow[r] \arrow[d,"\varphi"]
    & \Omega_{k_\beta/k}\otimes k_\alpha \arrow[r] \arrow[d]
    & 0
\\
    \pi_{1}(\bL_{k_\alpha/k}) \arrow[r] \arrow[d]
    & \fm_\alpha/\fm_\alpha^2 \arrow[r] \arrow[d]
    & G_X^0 \arrow[r] \arrow[d]
    & \Omega_{k_\alpha/k} \arrow[r] \arrow[d]
    & 0
\\
    \pi_{1}(\bL_{k_\alpha/k_\beta}) \arrow[r]
    & \pi_{1}(C_\sbt)\arrow[r] \arrow[d]
    & G_{X/Y}^0 \arrow[r] \arrow[d]
    & \Omega_{k_\alpha/k_\beta} \arrow[r] \arrow[d]
    & 0
\\
    \phantom{.} & 0 & 0 & 0
\end{tikzcd}
\end{equation}

\subsection{Study of general linear projections}

\label{sec:gen-lin-proj}

After the above preparations, we now address generalizing a series of theorems
in the literature. We start with \cite[Thm.~4.3~(1)]{CdFD23}, for which we
relax the perfectness assumption on the base field.

\begin{thm}
\label{thm:unramified_at_alpha_eta_implies_surjective_cotangent_map}
We retain the notation and assumptions of
\cref{subsec:cotangent_maps_assumptions}. If we furthermore assume that $f$ is
unramified at $\alpha(\eta)$, then $T_\alpha^{*}f_\infty$ is surjective.
Moreover, $k_\alpha/k_\beta$ is a finite separable field extension and in
particular $\bL_{k_\alpha/k_\beta}=0$.
\end{thm}

\begin{proof}
We use diagram
\eqref{eq:diagram_needed_to_prove_multiple_theorems_about_cotangent_map}. If
$f$ is unramified at $\alpha(\eta)$, then $\Omega_{X/Y} \otimes
k_\alpha\llparen t \rrparen = 0$, and therefore $\Omega_{X/Y} \otimes
k_\alpha\llbracket t \rrbracket$ is torsion. Thus we have the vanishings
$F_{X/Y}^0=0$ and $G_{X/Y}^0=0$, forcing $\Omega_{k_\alpha/k_\beta}$ to be $0$
as well by analyzing the bottom row of diagram
\eqref{eq:diagram_needed_to_prove_multiple_theorems_about_cotangent_map}.

Furthermore, since $f$ is unramified at $\alpha(\eta)$, the extension
$k_{\alpha(\eta)}/k_{\beta(\eta)}$ is finite separable \SPcite{090W}. This
forces $k_\alpha/k_\beta$ to be finite \cite[Prop.~3.7]{CdFD23}. When this is
combined with the fact that $\Omega_{k_\alpha/k_\beta}$ is $0$, we see that
$k_\alpha/k_\beta$ is separable \SPcite{090W}. Separability forces
$\pi_{1}(\bL_{k_\alpha/k_\beta}) = 0$ \cite[Thm.~3.4.1]{LS67}. Since
automatically $\pi_{i}(\bL_{k_\alpha/k_\beta})=0$ for $i > 1$
\cite[8.10]{Iye07}, we see $\bL_{k_\alpha/k_\beta} = 0$. By the bottom row of
diagram
\eqref{eq:diagram_needed_to_prove_multiple_theorems_about_cotangent_map}, if
$\pi_{1}(\bL_{k_\alpha/k_\beta})$ and $G_{X/Y}^0$ are both $0$, then so is
$\pi_{1}(C_\sbt)$. Thus $T_\alpha^{*}f_\infty$ is surjective.
\end{proof}

Next we address \cite[Thm.~8.1]{CdFD22} and its generalization
\cite[Cor.~4.6]{CdFD23}, again relaxing the perfectness assumption on the base
field. To facilitate discussion of this result, we first state
\lcnamecref{thm:generic_projection_properties} which specifies concretely what
we mean by \emph{general} linear projection for our purposes. As indicated,
such projections always exist, even over non algebraically closed base fields.
We then state the theorem, and afterwards prove both the lemma and the theorem
in succession.

\begin{lem}
\label{thm:generic_projection_properties}
We retain the notation and assumptions of
\cref{subsec:cotangent_maps_assumptions}. Assume furthermore that $X$ is affine
and of finite type over $k$, and fix a closed embedding $X \subset \AA^n$. Let
$d=\dim_{k_{\alpha(\eta)}}(\Omega_{X/k}\otimes_{\cO_X}k_{\alpha(\eta)})$.
A general linear projection $f \colon X \to Y = \AA^d$ satisfies the following
two properties:
\begin{enumerate}
    \item
        \label{thm:generic_projection_properties-item-1}
        $f$ is unramified at $\alpha(\eta)$.
    \item
        \label{thm:generic_projection_properties-item-2}
        $
        \ord_\alpha(\Fitt^d(\Omega_{X/k}))
        =
        \ord_\alpha(\Fitt^0(\Omega_{X/Y})) < \infty
        $.
\end{enumerate}
Moreover, there always exists such a projection defined over the base field $k$.
\end{lem}

The desired theorem generalizing \cite[Thm.~8.1]{CdFD22} and
\cite[Cor.~4.6]{CdFD23} is then the following.

\begin{thm}
\label{thm:general_linear_projection_implies_cotangent_map_is_an_isomorphism}
We retain the notation and assumptions of
\cref{thm:generic_projection_properties}. Furthermore, if $f \colon X \to Y =
\AA^d$ is the morphism induced by a general linear projection $\AA^n \to
\AA^d$, then $T_\alpha^{*}f_\infty$ is an isomorphism. In particular:
\[
    \embdim(\cO_{J_\infty(X),\alpha})
    =
    \embdim(\cO_{J_\infty(Y),\beta}).
\]
\end{thm}

\begin{proof}[Proof of \cref{thm:generic_projection_properties}]
Fix coordinates $x_1, \ldots, x_n$ in $\AA^n$ and $y_1, \ldots, y_d$ in
$\AA^d$. A linear projection $\AA^n \to \AA^d$ is determined setting $y_j =
b_{j} + \sum a_{ij} x_i$, where the coefficients $b_j$, $a_{ij}$ are in $k$ and
the $n \by d$ matrix $A = (a_{ij})$ has rank $d$. Such projections are
therefore parametrized by the $k$-valued points of an open set $\cW$ in
$\AA^{(n+1)d}$.

Let $g_1, \ldots, g_m \in k[x_1, \ldots, x_n]$ be generators for the ideal of
$X$, and write $J = (\partial g_j/\partial x_i)$ for the associated Jacobian
matrix. The modules
\[
    \Omega_{X/k}
    = \frac{\Omega_{\AA^n} \otimes \cO_X}
        {\langle dg_1, \ldots dg_m \rangle}
    \qquad\text{and}\qquad
    \Omega_{X/Y}
    = \frac{\Omega_{\AA^n} \otimes \cO_X}
        {\langle dg_1, \ldots dg_m, dy_1, \ldots, dy_d \rangle}
\]
have presentation matrices $J$ and $(J | A)$, respectively. In particular
$\Fitt^d(\Omega_{X/k})$ is generated by the $(n-d)\by(n-d)$ minors of $J$, and
$\Fitt^0(\Omega_{X/Y})$ by the $n\by n$ minors of $(J|A)$. Let $\Delta_J$ be an
$(n-d)\by(n-d)$ minor of $J$ computing $\ord_\alpha(\Fitt^d(\Omega_{X/k})) =
\ord_\alpha(\Delta_J)$. After possibly reindexing the coordinates $x_1, \ldots,
x_n$, we can assume that $\Delta_J$ involves precisely the last $n-d$ rows of
$J$. We let $\Delta_A$ be the $d \by d$ minor of $A$ corresponding to the first
$d$ rows of $A$, and let $\cV \subset \cW$ be the open subset defined by the
non-vanishing of $\Delta_A$. Notice that $\cV$ always contains $k$-valued
points. For example, we can always consider the projection onto the first $d$
coordinates, that is, $y_j = x_j$ for $1 \le j \le d$.

Any projection corresponding to a $k$-valued point in $\cV$ satisfies the
conditions of the \lcnamecref{thm:generic_projection_properties}. Indeed, fix
such a projection and change coordinates so that $y_j = x_j$ for $1 \le j \le
d$. Then the generators of $\Fitt^0(\Omega_{X/Y})$ are of the form
$\Delta_A\cdot \widetilde \Delta = \widetilde \Delta$, where $\widetilde
\Delta$ is a minor of $J$ involving the last $(n-d)$ rows. It immediately
follows that $ \ord_\alpha(\Fitt^d(\Omega_{X/k})) =
\ord_\alpha(\Fitt^0(\Omega_{X/Y})) = \ord_\alpha(\Delta_J)$, and this number is
finite because of the assumption
$d=\dim_{k_{\alpha(\eta)}}(\Omega_{X/k}\otimes_{\cO_X}k_{\alpha(\eta)})$. This
gives item (\ref{thm:generic_projection_properties-item-2}). For item
(\ref{thm:generic_projection_properties-item-1}) notice that $\Omega_{X/Y}
\otimes k_{\alpha(\eta)}$ is the cokernel $(J({\alpha(\eta)})|A)$, where
$J({\alpha(\eta)})$ denotes the evaluation of $J$ in $k_{\alpha(\eta)}$. But
one of the $n \by n$ minors of $(J({\alpha(\eta)})|A)$ equals
$\Delta_J(\alpha(\eta))$, which is non-zero in $k_{\alpha(\eta)}$ because
$\ord_\alpha(\Delta_J) < \infty$. Therefore $(J({\alpha(\eta)})|A)$ is
full-rank and its cokernel $\Omega_{X/Y} \otimes k_{\alpha(\eta)}$ is zero,
showing that $f$ is unramified at $\alpha(\eta)$.
\end{proof}

\begin{proof}[Proof of
\cref{thm:general_linear_projection_implies_cotangent_map_is_an_isomorphism}]
By \cref{thm:generic_projection_properties}, $f$ is unramified at
$\alpha(\eta)$, hence $T^*_\alpha f_\infty$ is surjective by
\cref{thm:unramified_at_alpha_eta_implies_surjective_cotangent_map}. To show
injectivity we use diagram
\eqref{eq:diagram_needed_to_prove_multiple_theorems_about_cotangent_map}. Again
by \cref{thm:unramified_at_alpha_eta_implies_surjective_cotangent_map} we see
that $\pi_1(\bL_{k_\alpha/k_\beta}) = 0$, and diagram
\eqref{eq:diagram_needed_to_prove_multiple_theorems_about_cotangent_map} shows
that $\pi_2(C_\sbt) \cong G^1_{X/Y} \oplus T^0_{X/Y}$. Therefore $\ker
(T^*_\alpha f_\infty) \cong \ker (\varphi)$, and we are reduced to showing that
$\ker (\varphi) = 0$.

Consider the long exact sequence of homotopy obtained from the fundamental
triangle $\bL_{Y/k} \otimes \cO_X \to \bL_{X/k} \to \bL_{X/Y}$ via pull-back to
the arc $\alpha$:
\[
    \cdots
    \to
    F^1_X \oplus T^1_X
    \to
    F^1_{X/Y} \oplus T^1_{X/Y}
    \to
    F^0_Y \oplus T^0_Y
    \to
    F^0_X \oplus T^0_X
    \to
    F^0_{X/Y} \oplus T^0_{X/Y}
    \to
    0.
\]
Since $Y$ is smooth, $T^i_Y=0$ for $i \ge 0$ and $F^i_Y=0$ for $i \ge 1$. Since
$f$ is unramified at $\alpha(\eta)$ we have $F^0_{X/Y} = 0$. The above long
exact sequence gets reduced to
\[
    0
    \to
    F^1_X \oplus T^1_X
    \lra
    F^1_{X/Y} \oplus T^1_{X/Y}
    \lra
    F^0_Y
    \lra[~\psi~]
    F^0_X \oplus T^0_X
    \lra
    T^0_{X/Y}
    \to
    0.
\]
Notice that $\psi \otimes k_\alpha\llparen t\rrparen$ is an isomorphism,
because $f$ is unramified at $\alpha(\eta)$ and both $F^0_Y$ and $F^0_X$ have
rank $d$. Therefore $\psi$ must be injective, leading to the short exact
sequence
\[
    0
    \to
    F^0_Y
    \lra[~\psi~]
    F^0_X \oplus T^0_X
    \lra[~\rho~]
    T^0_{X/Y}
    \to
    0.
\]
Injectivity of $\psi$ also gives that $\psi(F_Y^0) \cap T^0_X = 0$, and by
exactness $\ker(\rho) \cap T^0_X = 0$ and hence $\rho$ induces an injection
$T^0_X \hookrightarrow T^0_{X/Y}$. As discussed in
\cref{sec:pull-backs-complexes-arcs} we have:
\[
    \dim_{k_\alpha}(T^0_X) = \ord_\alpha(\Fitt^d(\Omega_{X/k}))
    \quad\text{and}\quad
    \dim_{k_\alpha}(T^0_{X/Y}) = \ord_\alpha(\Fitt^0(\Omega_{X/Y})).
\]
By \cref{thm:generic_projection_properties}, these two dimensions are equal, so
$\rho$ restricts to an isomorphism $T^0_X \cong T^0_{X/Y}$ and hence the
decomposition $F_X^0 \oplus T^0_X$ can be chosen so that $\psi(F_Y^0) = F_X^0$.
After tensoring with $V_\infty(k_\alpha)$ we get that $\varphi$ gives an
isomorphism $G^0_Y \to G^0_X$, showing that $\ker(\varphi)=0$, as required.
\end{proof}

\subsection{Study of resolutions of singularities and alterations}

We now consider \cite[Thms.~9.2,~9.3]{dFD20} and \cite[Thm.~4.3~(2)]{CdFD23},
for which we relax the perfectness assumption on the base field. While we are
mostly interested in understanding generically étale (or even birational) maps,
we start with a statement about generically smooth maps.

\begin{thm}
\label{thm:smooth_at_generic_point_implies_bound_on_kernel_of_cotangent_map}
We retain the notation and assumptions of
\cref{subsec:cotangent_maps_assumptions}. Assume furthermore that $X$ and $Y$
are essentially of finite type over $k$ and $f$ is is smooth at $\alpha(\eta)$,
and let $r=\dim(\Omega_{X/Y}\otimes_{\cO_X}k_{\alpha(\eta)})$. Then
\begin{align*}
    \dim_{k_\alpha}(\ker (T_\alpha^{*}f_\infty))
    \leq
    \ord_\alpha(\Fitt^{r}(\Omega_{X/Y})).
\end{align*}
\end{thm}

\begin{proof}
If $f$ is smooth at $\alpha(\eta)$, then $F_{X/Y}^1=0$ and therefore
$G_{X/Y}^1=0$. From the top row of diagram
\eqref{eq:diagram_needed_to_prove_multiple_theorems_about_cotangent_map} we see
that $\ker(T^*_\alpha f_\infty) \subseteq \ker(\varphi) \subseteq T_{X/Y}^0$.
The statement then follows by noticing that $\dim_{k_\alpha}(T_{X/Y}^0) =
\ord_\alpha(\Fitt^{r}(\Omega_{X/Y}))$, see
\cref{sec:pull-backs-complexes-arcs}.
\end{proof}

The following can be seen as a numerical version of the birational
transformation rule in motivic integration, extended to case of generically
étale maps.

\begin{thm}
\label{thm:gen_etale_implies_embedding_dimension_bounds}
We retain the notation and assumptions of
\cref{subsec:cotangent_maps_assumptions}. Assume furthermore that $X$ and $Y$
are essentially of finite type over $k$ and $f\colon X\to Y$ is étale at
$\alpha(\eta)$, and let $\Jac_{f}=\Fitt^{0}(\Omega_{X/Y})$ be the Jacobian
ideal of the map $f$. Then $T^*_\alpha f_\infty$ is surjective with kernel
dimension bounded by $\ord_\alpha(\Jac_{f})$, and
\begin{equation}
    \label{eq:edim-etale-bound-1}
    \embdim(\cO_{J_\infty (X),\alpha})
    \leq
    \embdim(\cO_{J_\infty (Y),\beta})
    \leq
    \embdim(\cO_{J_\infty (X),\alpha})+\ord_\alpha(\Jac_{f}).
\end{equation}
If in addition $X$ is smooth at $\alpha(0)$, then $\ker(T^*_\alpha f_\infty)$
has dimension $\ord_\alpha(\Jac_{f})$, and
\begin{equation}
    \label{eq:edim-etale-bound-2}
    \embdim(\cO_{J_\infty (Y),\beta})
    =
    \embdim(\cO_{J_\infty (X),\alpha})+\ord_\alpha(\Jac_{f}).
\end{equation}
\end{thm}

\begin{rmk}
If $f$ is generically étale (for example, a birational map),
\cref{eq:edim-etale-bound-1} holds for \emph{all} arcs $\alpha \in
J_\infty(X)$. Indeed, if such $f$ is not étale at $\alpha(\eta)$ then $\alpha$
is contained in the ramification locus. It is therefore \emph{thin}, in the
sense of \cref{dff:jetcodim}. This forces $\beta$ to be thin as well, and
\eqref{eq:edim-etale-bound-1} becomes vacuously true from
\cref{thm:thin-infinite-edim}. Similarly, if $f$ is generically étale and $X$
is smooth (for example, $f$ is a resolution of singularities),
\cref{eq:edim-etale-bound-2} also holds for all arcs.
\end{rmk}

\begin{proof}
Since $f$ is étale at $\alpha(\eta)$ we see that $\pi_i(\bL_{X/Y}
\otimes_{\cO_X} k_\alpha\llbracket t \rrbracket)$ is torsion, and hence we have
the vanishing $F_{X/Y}^i = 0$ for all $i \ge 0$. Moreover, from
\cref{thm:unramified_at_alpha_eta_implies_surjective_cotangent_map} we see that
$\pi_1(\bL_{k_\alpha/k_\beta})=0$ (see \cite[Lemma~9.1]{dFD20} for an argument
showing that $k_\alpha = k_\beta$ when $f$ is proper and birational at
$\alpha(\eta)$).

We use the diagram
\eqref{eq:diagram_needed_to_prove_multiple_theorems_about_cotangent_map}.
Because $F_{X/Y}^i=0$ we see that $G_{X/Y}^i = 0$, and we also know that
$\pi_1(\bL_{k_\alpha/k_\beta})=0$. Looking at the bottom row of the diagram
\eqref{eq:diagram_needed_to_prove_multiple_theorems_about_cotangent_map} we get
that $\pi_1(C_\sbt) = 0$, and from the top row we notice that $\pi_2(C_\sbt)
\cong T^0_{X/Y}$. Furthermore $\ker(T^*_\alpha f_\infty) \cong \ker(\varphi)
\subseteq T^0_{X/Y}$. Putting everything together we see that
\[
    \embdim(\cO_{J_\infty (Y),\beta})
    =
    \embdim(\cO_{J_\infty (X),\alpha}) +
    \dim_{k_\alpha}(\ker(\varphi))
\]
and
\[
    0
    \le
    \dim_{k_\alpha}(\ker(\varphi))
    \le
    \dim_{k_\alpha}(T^0_{X/Y}) = \ord_\alpha(\Jac_f),
\]
giving the first part of the
\lcnamecref{thm:gen_etale_implies_embedding_dimension_bounds}.

For the second part of the
\lcnamecref{thm:gen_etale_implies_embedding_dimension_bounds}, notice that the
third column in diagram
\eqref{eq:diagram_needed_to_prove_multiple_theorems_about_cotangent_map}
continues on top with $\pi_1(\bL_{\bL J_\infty(X)/k} \otimes^\bL k_\alpha)
\cong G^1_X \oplus T^0_X$. But if $X$ is smooth at $\alpha(0)$ we see that
$F^i_X = 0$ for $i > 0$ and $T^i_X = 0$ for $i \ge 0$. In particular
$\pi_1(\bL_{\bL J_\infty(X)/k} \otimes^\bL k_\alpha) = 0$ and hence
$\ker(\varphi) = T^0_{X/Y}$, giving the result.
\end{proof}

\section{Stable arcs and the curve selection lemma}

\label{sec:curveselection}

In this section we extend Reguera's study of stable arcs to the non-perfect
setting, including the curve selection lemma \cite{Reg06,dFD20,Reg21}. We note
that derived jets and arcs do not appear directly in this section, but several
of the proofs rely on results that ultimately depend on \cref{thm:deriveddFD}.

\subsection{Stable arcs}

\label{sec:stable_arcs}

We start by recalling the definition of \emph{stable arc}. This notion has a
complicated presentation in the literature, and in particular most of the works
that deal with it assume that the ground field is algebraically closed or
perfect. Our applications lie outside these hypotheses so we specify precisely
which notion of stable arc works best in this general setting.

\begin{dff}\label{dff:usefulstable}
Let $X$ be a classical scheme essentially of finite type over a field $k$.
Consider an arc $\alpha \in J_\infty(X)$ and let $d = \dim_{\alpha(\eta)}(X)$
be the dimension of $X$ at the generic point of $\alpha$. We say that $\alpha$
is \emph{stable} if it is non-degenerate and the field extensions
$k_{\alpha_m}/k_{\alpha_n}$ are purely transcendental of degree $(m-n) d$ for
$m \ge n\gg0$ large enough. Here $\alpha_n$ denotes the truncation of $\alpha$
to $J_n(X)$ and $k_{\alpha_n}$ is the residue field of $\alpha_n$ as a point of
$J_n(X)$.
\end{dff}

Notice that there are no stable arcs unless $X$ has a generically smooth
component.

\begin{rmk}
The notion of stability of on arc spaces dates back to \cite{DL99}, where the
foundations of the theory of motivic integration were laid out in detail.
Precisely, in \cite[(2.7)]{DL99} the authors introduce \emph{stable subsets} of
the arc space of a variety as those constructible subsets $C \subseteq
J_\infty(X)$ for which the truncation maps $\theta_{m,n} \colon \theta_m(C) \to
\theta_n(C)$ are piecewise locally trivial fibrations with fiber
$\AA^{(m-n)\dim X}$. This is precisely the condition that is needed to
guarantee stabilization in the limit that defines the motivic integral.

Since the arc space is the projective limit of the jet schemes, all
constructible subsets $C \subseteq J_\infty(X)$ have the form $C =
\theta_n^{-1}(C_n)$ for some $n$, where $C_n$ is constructible in $J_n(X)$
\cite[Th\'eor\`eme~8.3.11]{Gro66}. For this reason, constructible subsets are
also known as \emph{cylinders}. In \cite{DL99} it is shown that stable subsets
can be alternatively characterized as cylinders not completely contained in the
arc space of the singular locus (recall from \Cref{dff:non-degenerate} that
this last condition is known as \emph{non-degeneracy}). In \cite{DL99} the
authors assume a ground field of characteristic zero, but many of the proofs do
not require this hypotheses. This is true in particular for
\cite[Lem.~4.1]{DL99}, the so-called ``Denef-Loeser Lemma.'' In
\cite[Prop.~4.1]{EM09} we find a version over algebraically closed fields in
arbitrary characteristic, and the fully general case is treated in
\cite[Thm.~2.3.11]{CLNS18}.

In \cite{Reg06} the weaker notion of \emph{generically stable subset} is
introduced, initially overlooking the condition of non-degeneracy, see
\cite{Reg21}. These are defined to be the irreducible subsets of the arc space
containing an affine non-degenerate constructible open dense subset. Finally,
in \cite{Reg09} we find the first appearance of the term \emph{stable arc},
where various definitions are proven to be equivalent. In these works, it is
always assumed that the ground field is perfect but, as explained in
\cite[Remark ~(C.10)]{Reg21} this condition is only used to guarantee the
equivalence between smoothness and regularity. The ultimate conclusion of the
developments of \cite{DL99,Reg06,Reg09,CLNS18,Reg21} is the following
characterization: an arc $\alpha \in J_\infty(X)$ is stable in the sense of
\cref{dff:usefulstable} if and only if it is non-degenerate and it is the
generic point of an irreducible constructible subset of $J_\infty(X)$.
\end{rmk}

\subsection{Jet codimension}

We will follow the strategy of \cite{dFD20} to prove finiteness results for
stable arcs. To make the discussion more conceptual, we recall some
terminology.

\begin{dff}
\label{dff:jetcodim}
Let $X$ be a classical scheme which is essentially of finite type over a field
$k$. Consider an arc $\alpha$ in $J_\infty(X)$, let $\alpha(\eta) \in X$ be its
generic point, and let $d = \dim_{\alpha(\eta)}(X)$ be the dimension of $X$ at
$\alpha(\eta)$. We let $\alpha_n$ denote the truncation of $\alpha$ in
$J_n(X)$, and denote by $k_{\alpha_n}$ its residue field.
\begin{enumerate}
\item
    We say that $\alpha$ is \emph{thin} if it is completely contained in
    $J_\infty(Y)$ for some integral closed subscheme $Y \subset X$ with
    dimension $\dim(Y) < d$.
\item
    The \emph{jet codimension} of $\alpha$ in $J_\infty(X)$ is the quantity
    \[
        \jetcodim(\alpha, J_\infty(X))
        =
        \limsup_{n \to \infty} \left(
            (n+1) \, d
            - \dim(\overline{\{\alpha_{n}\}})
        \right)
        \in \NN \cup \{ \infty \}.
    \]
    Notice that the jet codimension of a stable arc is manifestly finite. As we
    observe below, stable arcs are in fact characterized by having finite jet
    codimension.
\item
    We say that $X$ is \emph{generically smooth} over $k$ if $X$ is smooth at
    the generic point of each of its irreducible components.
\end{enumerate}
\end{dff}

We start with some elementary observations.

\begin{lem}
\label{lem:degenerate-is-thin}
If $X$ is generically smooth, degenerate arcs are thin.
\end{lem}

\begin{proof}
If $X$ is generically smooth, each component of the singular locus $\Sing(X)$
is non-generic on $X$. A non-degenerate arc $\alpha$ is contained in a
component of $\Sing(X)$, and therefore it is thin.
\end{proof}

\begin{lem}
\label{lem:thin-infinte-jcodim}
Thin arcs have infinite jet codimension.
\end{lem}

\begin{proof}
For an integral scheme $Y$ which is essentially of finite type and has
dimension $d'$, the image $J_\infty(Y) \to J_n(Y)$ has dimension $(n+1)d'$.
Therefore, for a thin arc $\alpha$ contained in $Y$ we have
\[
    \jetcodim(\alpha, J_\infty(X))
    \ge
        \limsup_{n \to \infty} \left(
            (n+1) \, (d-d')
        \right)
    = \infty.
    \qedhere
\]
\end{proof}

The jet codimension can be controlled via generic projections. The following is
a preliminary bound, which will be upgraded to a general identity in
\cref{thm:jetcodim_generic_projection_full}.

\begin{lem}
\label{thm:jetcodim_generic_projection}
With the same assumptions as
\cref{thm:general_linear_projection_implies_cotangent_map_is_an_isomorphism},
for an arc $\alpha$ in $J_\infty(X)$, a general linear projection $f \colon X
\to Y = \AA^d$, and the image arc $\beta = f_\infty(\alpha)$ in $J_\infty(Y)$,
we have
\[
    \jetcodim(\alpha, J_\infty(X))
    \le
    \jetcodim(\beta, J_\infty(Y))
    .
\]
\end{lem}

\begin{proof}
By \cref{thm:generic_projection_properties}, $f$ is unramified at
$\alpha(\eta)$, and from
\cref{thm:unramified_at_alpha_eta_implies_surjective_cotangent_map} the field
extension $k_\alpha/k_\beta$ is finite separable. Write $k_\alpha =
k_\beta[\xi]$, where $\xi$ is a primitive element of $k_\alpha$ over $k_\beta$.

The truncation maps induce field extensions $k_\alpha/k_{\alpha_n}$ and
$k_\beta/k_{\beta_n}$, and we have $k_\alpha = \cup_n k_{\alpha_n}$ and
$k_\beta = \cup_n k_{\beta_n}$. Pick $n_0$ large enough so that $\xi$ belongs
to $k_{\alpha_{n_0}}$ and $k_{\beta_{n_0}}$ contains the coefficients of the
minimal polynomial of $\xi$ over $k_\beta$. When $n \ge n_0$ there is a
containment $k_{\beta_n}[\xi] \subset k_{\alpha_n}$ and moreover
$k_{\beta_n}[\xi]$ is finite over $k_{\beta_n}$. We see:
\[
    \dim(\overline{\{\beta_{n}\}})
    =
    \trdeg(k_{\beta_n}/k)
    =
    \trdeg(k_{\beta_n}[\xi]/k)
    \le
    \trdeg(k_{\alpha_n}/k)
    =
    \dim(\overline{\{\alpha_{n}\}}).
\]
If $\alpha$ is non-degenerate, we have that $d = \dim_{\alpha(\eta)}(X) =
\dim_{\beta(\eta)}(Y)$, and the result follows from the definition of jet
codimension. If $\alpha$ is degenerate, the image arc $\beta$ is thin, as it is
fully contained in $f(X)$, whose dimension at $\beta(\eta)$ is strictly smaller than $d$.
In this case the result follows vacuously from \cref{lem:thin-infinte-jcodim}.
\end{proof}

We finish this subsection stating the following well-known corollary of the
``Denef-Loeser Lemma'' \cite[Lem.~4.1]{DL99}.

\begin{lem}
\label{thm:stable-via-jetcodim}
Let $k$ be a field, let $X$ be a classical $k$-scheme which is essentially of
finite type and generically smooth over $k$. An arc $\alpha \in J_\infty(X)$
is stable if and only it has finite jet codimension.
\end{lem}

\begin{proof}
For degenerate arcs use \cref{lem:degenerate-is-thin,lem:thin-infinte-jcodim}.
In the non-degenerate case the result follows easily from \cite[Lem.~4.1]{DL99}
or \cite[Thm.~2.3.11]{CLNS18}.
\end{proof}

\subsection{Comparing embedding dimension and jet codimension}

We show that the embedding dimension and the jet codimension agree for arcs on
generically smooth schemes. As a first step, we start with the following
interesting corollary of
\cref{thm:general_linear_projection_implies_cotangent_map_is_an_isomorphism}.

\begin{thm}
\label{thm:thin-infinite-edim}
Let $k$ be a field, let $X$ be a classical $k$-scheme which is essentially of
finite type over $k$, and let $\alpha \in J_\infty(X)$ an arc. If $\alpha$ is
thin or degenerate then
\[
    \embdim(\cO_{J_\infty(X),\alpha}) = \infty.
\]
\end{thm}

\begin{proof}
We can work locally and assume $X$ is affine, embedded in $\AA^n$. Pick a
general linear projection $f \colon X \to \AA^d$ as in
\cref{thm:general_linear_projection_implies_cotangent_map_is_an_isomorphism},
in particular $d =
\dim_{k_{\alpha(\eta)}}(\Omega_{X/k}\otimes_{\cO_X}k_{\alpha(\eta)})$. Set
$\beta = f_\infty(\alpha)$. If $\alpha$ is thin we can find $Z \subset X$
containing $\alpha$ and with dimension strictly smaller than $d$. If $\alpha$
is degenerate but not thin let $Z \subset X$ be an irreducible component of $X$
containing $\alpha$, and notice that it has dimension strictly smaller than
$d$. Indeed, this can only happen when $X$ is not generically smooth, and this
case we must have $d =
\dim_{k_{\alpha(\eta)}}(\Omega_{X/k}\otimes_{\cO_X}k_{\alpha(\eta)}) >
\dim_{\alpha(\eta)}(X)$. In either case, since $f$ is quasi-finite at
$\alpha(\eta)$ \SPcite{02V5}, the image $f(Z)$ also has dimension strictly
smaller than $d$, showing that $\beta$ is thin. Notice that $\AA^d$ is smooth,
so we can apply \cref{lem:embdimlimit} \eqref{eq:embdimlimit-3} to obtain the
middle equality in:
\[
    \embdim(\cO_{J_\infty(X),\alpha})
    =
    \embdim(\cO_{J_\infty(\AA^d),\beta})
    =
    \jetcodim(\beta, J_\infty(\AA^d)) = \infty.
\]
Here the first equality follows from
\cref{thm:general_linear_projection_implies_cotangent_map_is_an_isomorphism},
and for the last identity we have used \cref{lem:thin-infinte-jcodim}.
\end{proof}

The following is our main result in this section, which extends
\cite[Thm.~10.7]{dFD20} to the non-perfect case.

\begin{thm}
\label{thm:edim-equals-jetcodim}
Let $k$ be a field, let $X$ be a classical scheme and assume that it is
essentially of finite type and generically smooth over $k$. For $\alpha\in
J_\infty(X)$ an arc,
\[
    \embdim(\cO_{J_\infty(X),\alpha})
    =
    \jetcodim(\alpha,J_\infty(X)).
\]
\end{thm}

\begin{proof}
If $\alpha$ is degenerate, the result follows from
\cref{thm:thin-infinite-edim,lem:degenerate-is-thin,lem:thin-infinte-jcodim}.
When $\alpha$ is non-degenerate the $0$-th Betti number of $\bL_{X/k}$ with
respect to $\alpha$ is $d = b_0(\alpha) = \dim_{\alpha(\eta)}(X)$.
From \cref{lem:boundonimagelambdan-non-deg} we get the bounds
\[
\begin{aligned}
    & \jetcodim(\alpha,J_\infty(X))
    -\ord_\alpha(\Jac_X^{(1,0)})
\\
    & \le
    \embdim(\cO_{J_\infty(X),\alpha})
\\
    & \le
    \jetcodim(\alpha,J_\infty(X))
    +\ord_\alpha(\Jac_X^{(0,d)}).
\end{aligned}
\]
As a consequence, $\alpha$ has finite embedding dimension if and only if it has
finite jet codimension. By \cref{thm:stable-via-jetcodim}, we can assume
$\alpha$ is stable. In this situation we can apply \cref{lem:embdimlimit}
\eqref{eq:embdimlimit-2} and we see that
\[
    \embdim(\cO_{J_\infty(X),\alpha})
    \le
    \jetcodim(\alpha,J_\infty(X)).
\]
Using a generic projection, from
\cref{thm:jetcodim_generic_projection},
\cref{lem:embdimlimit}~\eqref{eq:embdimlimit-3}, and
\cref{thm:general_linear_projection_implies_cotangent_map_is_an_isomorphism}
we get
\[
    \jetcodim(\alpha,J_\infty(X))
    \le
    \jetcodim(\beta,J_\infty(Y))
    =
    \embdim(\cO_{J_\infty(Y),\beta})
    =
    \embdim(\cO_{J_\infty(X),\alpha}),
\]
and the result follows.
\end{proof}

As a consequence, all the results obtained in \cref{sec:cotangent_maps} for the
embedding dimension have counterparts for the jet codimension in the
generically smooth case. For instance, we get the following statement for
general projections.

\begin{cor}
\label{thm:jetcodim_generic_projection_full}
With the same assumptions as
\cref{thm:general_linear_projection_implies_cotangent_map_is_an_isomorphism},
for a generically smooth $X$, an arc $\alpha$ in $J_\infty(X)$, a general
linear projection $f \colon X \to Y = \AA^d$, and the image arc $\beta =
f_\infty(\alpha)$ in $J_\infty(Y)$, we have
\[
    \jetcodim(\alpha, J_\infty(X))
    =
    \jetcodim(\beta, J_\infty(Y))
    .
\]
\end{cor}

\subsection{The curve selection lemma for stable arcs}

\label{subsec:curveselection}

In the spirit of \cite{dFD20} we obtain the following finiteness result for
stable arcs.

\begin{cor}
\label{cor:curve-selection}
\label{thm:stable-via-edim}
Let $k$ be a field, let $X$ be a classical $k$-scheme which is essentially of
finite type and generically smooth over $k$. For $\alpha \in J_\infty(X)$
an arc, the following are equivalent:
\begin{enumerate}
    \item $\alpha$ is stable.
    \item $\jetcodim(\alpha,J_\infty(X)) < \infty$.
    \item $\embdim(\cO_{J_\infty(X),\alpha}) < \infty$.
    \item The completed local ring $\widehat\cO_{J_\infty(X),\alpha}$ is
    Noetherian.
\end{enumerate}
\end{cor}

\begin{proof}
After \cref{thm:edim-equals-jetcodim,thm:stable-via-jetcodim} only the
equivalence between the last two items needs discussion. But this is a
well-known fact about completions, see for instance \cite[Lem.~10.12]{dFD20}.
\end{proof}

\begin{rmk}
Noetherianity of completed local rings at stable points is known by the experts
as ``Reguera's curve selection lemma.'' It is a key result which is used,
either explicitly or implicitly, in almost all developments related to the Nash
problem. For further details and motivation of the name ``curve selection'' we
refer the reader to \cite{Reg06}. The above result generalizes previously known
versions by removing any hypotheses on the ground field.
\end{rmk}

\subsection{Stable arcs, birational maps, and maximal divisorial arcs}

After establishing \cref{cor:curve-selection} it is interesting to revisit some
of the results of \cref{sec:cotangent_maps} and explore their consequences for
stable arcs. We start by addressing \cite[Cor.~9.4]{dFD20} in the non-perfect
case.

\begin{cor}
Let $X$ and $Y$ be classical schemes which are essentially of finite
type over $k$. For $f\colon X\to Y$ a proper birational map, $f_\infty$
induces a bijection
\[
    \{\,
        \alpha \in J_\infty(X)
    \,\,|\,\,
        \embdim(\cO_{J_\infty(X),\alpha}) < \infty
    \,\}
    \xra{1-1}
    \{\,
        \beta \in J_\infty(Y)
    \,\,|\,\,
        \embdim(\cO_{J_\infty(Y),\beta}) < \infty
    \,\}.
\]
If $X$ is generically smooth (or equivalently $Y$ is generically smooth) we can
rephrase this bijection as
\[
    \{\,
        \alpha \in J_\infty(X)
    \,\,|\,\,
        \text{$\alpha$ stable}
    \,\}
    \xra{\quad 1-1 \quad}
    \{\,
        \beta \in J_\infty(Y)
    \,\,|\,\,
        \text{$\beta$ stable}
    \,\}.
\]
\end{cor}

\begin{proof}
Notice that the valuative criterion of properness ensures that the arcs in $Y$
which are not liftable to $X$ must be thin.
Now use
\cref{thm:gen_etale_implies_embedding_dimension_bounds,thm:thin-infinite-edim}.
\end{proof}

In birational geometry and singularity theory it is typical to study
\emph{divisorial valuations} on varieties, as well as their associated
\emph{discrepancies} (of different types). We will not develop this theory
here, and instead refer the reader to \cite[Sec.~11]{dFD20} for a detailed
discussion. But we want to point out that our techniques extend previously
known results to the non-perfect case. The next statement summarizes some of
these extensions.

\begin{cor}
Let $k$ be a field, let $X$ be a classical $k$-scheme which is essentially of
finite type and generically smooth over $k$. Let $v = q \ord_E$ be a divisorial
valuation on $X$. Then the corresponding maximal divisorial arc $\alpha_v$ in
$J_\infty(X)$ is stable, and we have the equalities
\[
    \embdim(\cO_{J_\infty(X),\alpha_v})
    =
    \jetcodim(\alpha_v, J_\infty(X))
    =
    q(\widehat k_E(X)+1)
    <
    \infty
    ,
\]
where $\widehat k_E(X)$ denotes the Mather discrepancy of $E$ over $X$. If $f
\colon X \to Y$ is a proper birational map, then $\beta_v = f_\infty(\alpha_v)$
is the maximal divisorial arc in $J_\infty(Y)$ associated to $v$, and any
maximal divisorial arc on $J_\infty(Y)$ arises uniquely in this way. We have
the bounds
\[
    \widehat k_E(X)
    \le
    \widehat k_E(Y)
    \le
    \widehat k_E(X)
    + \ord_E(\Jac_f).
\]
If $X$ is smooth at the center of $E$ we furthermore have
\[
    \widehat k_E(Y)
    =
    \widehat k_E(X)
    + \ord_E(\Jac_f).
\]
\end{cor}


\hbadness 5000
\hfuzz 50pt


\begin{bibdiv}
\begin{biblist}

\bib{AR94}{book}{
    author    = {Ad\'{a}mek, Ji\v{r}\'{\i}},
    author    = {Rosick\'{y}, Ji\v{r}\'{\i}},
    title     = {Locally presentable and accessible categories},
    series    = {London Mathematical Society Lecture Note Series},
    volume    = {189},
    publisher = {Cambridge University Press, Cambridge},
    date      = {1994},
    pages     = {xiv+316},
    isbn      = {0-521-42261-2},
    review    = {\MR{1294136}},
    doi       = {10.1017/CBO9780511600579},
}

\bib{ARV10}{article}{
    author  = {Ad\'{a}mek, J.},
    author  = {Rosick\'{y}, J.},
    author  = {Vitale, E. M.},
    title   = {What are sifted colimits?},
    journal = {Theory Appl. Categ.},
    volume  = {23},
    date    = {2010},
    pages   = {No. 13, 251--260},
    review  = {\MR{2720191}},
}

\bib{BMS19}{article}{
    author  = {Bhatt, Bhargav},
    author  = {Morrow, Matthew},
    author  = {Scholze, Peter},
    title   = {Topological Hochschild homology and integral $p$-adic Hodge theory},
    journal = {Publ. Math. Inst. Hautes \'{E}tudes Sci.},
    volume  = {129},
    date    = {2019},
    pages   = {199--310},
    issn    = {0073-8301},
    review  = {\MR{3949030}},
    doi     = {10.1007/s10240-019-00106-9},
}


\bib{CLNS18}{book}{
    author    = {Chambert-Loir, Antoine},
    author    = {Nicaise, Johannes},
    author    = {Sebag, Julien},
    title     = {Motivic integration},
    series    = {Progress in Mathematics},
    volume    = {325},
    publisher = {Birkh\"auser/Springer, New York},
    date      = {2018},
    pages     = {xx+526},
    isbn      = {978-1-4939-7885-4},
    isbn      = {978-1-4939-7887-8},
    review    = {\MR{3838446}},
    doi       = {10.1007/978-1-4939-7887-8},
}

\bib{CdFD22}{article}{
    author  = {Chiu, Christopher},
    author  = {de Fernex, Tommaso},
    author  = {Docampo, Roi},
    title   = {Embedding codimension of the space of arcs},
    journal = {Forum Math. Pi},
    volume  = {10},
    date    = {2022},
    pages   = {Paper No. e4, 37},
    review  = {\MR{4386350}},
    doi     = {10.1017/fmp.2021.19},
}

\bib{CdFD23}{article}{
    author  = {Chiu, Christopher},
    author  = {de Fernex, Tommaso},
    author  = {Docampo, Roi},
    title   = {On arc fibers of morphisms of schemes},
    journal = {J. Eur. Math. Soc. (JEMS)},
    volume  = {28},
    date    = {2026},
    number  = {4},
    pages   = {1489--1531},
    issn    = {1435-9855},
    review  = {\MR{5032983}},
    doi     = {10.4171/jems/1500},
}

\bib{CNM23}{article}{
    author  = {Chiu, Christopher},
    author  = {Narváez Macarro, Luis},
    title   = {Higher derivations of modules and the Hasse-Schmidt module},
    journal = {Michigan Math. J.},
    volume  = {73},
    date    = {2023},
    number  = {3},
    pages   = {473--487},
    issn    = {0026-2285},
    review  = {\MR{4612160}},
    doi     = {10.1307/mmj/20205958},
}

\bib{CS20}{article}{
    author  = {\v Cesnavi\v cius, K\polhk estutis},
    author  = {Scholze, Peter},
    title   = {Purity for flat cohomology},
    journal = {Ann. of Math. (2)},
    volume  = {199},
    date    = {2024},
    number  = {1},
    pages   = {51--180},
    issn    = {0003-486X},
    review  = {\MR{4681144}},
    doi     = {10.4007/annals.2024.199.1.2},
}

\bib{Cis19}{book}{
    author    = {Cisinski, Denis-Charles},
    title     = {Higher categories and homotopical algebra},
    series    = {Cambridge Studies in Advanced Mathematics},
    volume    = {180},
    publisher = {Cambridge University Press, Cambridge},
    date      = {2019},
    pages     = {xviii+430},
    isbn      = {978-1-108-47320-0},
    review    = {\MR{3931682}},
    doi       = {10.1017/9781108588737},
}

\bib{dF18}{article}{
    author={de Fernex, Tommaso},
    title={The space of arcs of an algebraic variety},
    conference={
       title={Algebraic geometry: Salt Lake City 2015},
    },
    book={
       series={Proc. Sympos. Pure Math.},
       volume={97.1},
       publisher={Amer. Math. Soc., Providence, RI},
    },
    isbn={978-1-4704-3577-6},
    date={2018},
    pages={169--197},
    review={\MR{3821149}},
    doi={10.1090/pspum/097.1/06},
}

\bib{dFD14}{article}{
    author  = {de Fernex, Tommaso},
    author  = {Docampo, Roi},
    title   = {Jacobian discrepancies and rational singularities},
    journal = {J. Eur. Math. Soc. (JEMS)},
    volume  = {16},
    date    = {2014},
    number  = {1},
    pages   = {165--199},
    issn    = {1435-9855},
    review  = {\MR{3141731}},
    doi     = {10.4171/JEMS/430},
}

\bib{dFD16}{article}{
    author  = {de Fernex, Tommaso},
    author  = {Docampo, Roi},
    title   = {Terminal valuations and the Nash problem},
    journal = {Invent. Math.},
    volume  = {203},
    date    = {2016},
    number  = {1},
    pages   = {303--331},
    issn    = {0020-9910},
    review  = {\MR{3437873}},
    doi     = {10.1007/s00222-015-0597-5},
}

\bib{dFD20}{article}{
    author  = {de Fernex, Tommaso},
    author  = {Docampo, Roi},
    title   = {Differentials on the arc space},
    journal = {Duke Math. J.},
    volume  = {169},
    date    = {2020},
    number  = {2},
    pages   = {353--396},
    issn    = {0012-7094},
    review  = {\MR{4057146}},
    doi     = {10.1215/00127094-2019-0043},
}

\bib{DL99}{article}{
    author  = {Denef, Jan},
    author  = {Loeser, Fran\c{c}ois},
    title   = {Germs of arcs on singular algebraic varieties and motivic integration},
    journal = {Inventiones Mathematicae},
    volume  = {135},
    date    = {1999},
    number  = {1},
    pages   = {201--232},
    issn    = {0012-7094},
    review  = {\MR{4057146}},
    doi     = {10.1215/00127094-2019-0043},
}

\bib{EI15}{article}{
    author={Ein, Lawrence},
    author={Ishii, Shihoko},
    title={Singularities with respect to Mather-Jacobian discrepancies},
    conference={
       title={Commutative algebra and noncommutative algebraic geometry. Vol.
       II},
    },
    book={
       series={Math. Sci. Res. Inst. Publ.},
       volume={68},
       publisher={Cambridge Univ. Press, New York},
    },
    isbn={978-1-107-14972-4},
    date={2015},
    pages={125--168},
    review={\MR{3496863}},
}

\bib{EIM16}{article}{
    author={Ein, Lawrence},
    author={Ishii, Shihoko},
    author={Musta\c t\u a, Mircea},
    title={Multiplier ideals via Mather discrepancy},
    conference={
       title={Minimal models and extremal rays},
       address={Kyoto},
       date={2011},
    },
    book={
       series={Adv. Stud. Pure Math.},
       volume={70},
       publisher={Math. Soc. Japan, [Tokyo]},
    },
    isbn={978-4-86497-036-5},
    date={2016},
    pages={9--28},
    review={\MR{3617776}},
    doi={10.2969/aspm/07010009},
}

\bib{EM04}{article}{
    author  = {Ein, Lawrence},
    author  = {Musta\c t\u a, Mircea},
    title   = {Inversion of adjunction for local complete intersection varieties},
    journal = {Amer. J. Math.},
    volume  = {126},
    date    = {2004},
    number  = {6},
    pages   = {1355--1365},
    issn    = {0002-9327},
    review  = {\MR{2102399}},
}

\bib{EM09}{article}{
    author={Ein, Lawrence},
    author={Mustaţă, Mircea},
    title={Jet schemes and singularities},
    conference={
       title={Algebraic geometry---Seattle 2005. Part 2},
    },
    book={
       series={Proc. Sympos. Pure Math.},
       volume={80, Part 2},
       publisher={Amer. Math. Soc., Providence, RI},
    },
    isbn={978-0-8218-4703-9},
    date={2009},
    pages={505--546},
    review={\MR{2483946}},
    doi={10.1090/pspum/080.2/2483946},
}

\bib{EMY03}{article}{
    author  = {Ein, Lawrence},
    author  = {Musta\c t\u a, Mircea},
    author  = {Yasuda, Takehiko},
    title   = {Jet schemes, log discrepancies and inversion of adjunction},
    journal = {Invent. Math.},
    volume  = {153},
    date    = {2003},
    number  = {3},
    pages   = {519--535},
    issn    = {0020-9910},
    review  = {\MR{2000468}},
    doi     = {10.1007/s00222-003-0298-3},
}

\bib{Eis95}{book}{
    author    = {Eisenbud, David},
    title     = {Commutative algebra},
    series    = {Graduate Texts in Mathematics},
    volume    = {150},
    note      = {With a view toward algebraic geometry},
    publisher = {Springer-Verlag, New York},
    date      = {1995},
    pages     = {xvi+785},
    isbn      = {0-387-94268-8},
    isbn      = {0-387-94269-6},
    review    = {\MR{1322960}},
    doi       = {10.1007/978-1-4612-5350-1},
}

\bib{FdBPP12}{article}{
    author  = {Fern\'{a}ndez de Bobadilla, Javier},
    author  = {Pereira, Mar\'{\i}a Pe},
    title   = {The Nash problem for surfaces},
    journal = {Ann. of Math. (2)},
    volume  = {176},
    date    = {2012},
    number  = {3},
    pages   = {2003--2029},
    issn    = {0003-486X},
    review  = {\MR{2979864}},
    doi     = {10.4007/annals.2012.176.3.11},
}

\bib{GR12}{article}{
    author={Gaitsgory, Dennis},
    author={Rozenblyum, Nick},
    title={DG indschemes},
    conference={
       title={Perspectives in representation theory},
    },
    book={
       series={Contemp. Math.},
       volume={610},
       publisher={Amer. Math. Soc., Providence, RI},
    },
    isbn={978-0-8218-9170-4},
    date={2014},
    pages={139--251},
    review={\MR{3220628}},
    doi={10.1090/conm/610/12080},
}

\bib{GL87}{article}{
    author  = {Green, Mark},
    author  = {Lazarsfeld, Robert},
    title   = {Deformation theory, generic vanishing theorems, and some
              conjectures of Enriques, Catanese and Beauville},
    journal = {Invent. Math.},
    volume  = {90},
    date    = {1987},
    number  = {2},
    pages   = {389--407},
    issn    = {0020-9910},
    review  = {\MR{0910207}},
    doi     = {10.1007/BF01388711},
}

\bib{Gro66}{article}{
    author  = {Grothendieck, A.},
    title   = {\'El\'ements de g\'eom\'etrie alg\'ebrique. IV. \'Etude locale des
              sch\'emas et des morphismes de sch\'emas. III},
    journal = {Inst. Hautes \'Etudes Sci. Publ. Math.},
    number  = {28},
    date    = {1966},
    pages   = {255},
    issn    = {0073-8301},
    review  = {\MR{0217086}},
}
\bib{Hen15}{thesis}{
    title  = {Formal loops spaces and tangent Lie algebras},
    author = {Hennion, Benjamin},
    type   = {Ph.D.~Thesis},
    school = {Universit{\'e} Montpellier},
    year   = {2015},
    url    = {https://theses.hal.science/tel-02024827},
    note   = {Availabe at \url{https://theses.hal.science/tel-02024827}},
}

\bib{Ish13}{article}{
    author   = {Ishii, Shihoko},
    title    = {Mather discrepancy and the arc spaces},
    language = {English, with English and French summaries},
    journal  = {Ann. Inst. Fourier (Grenoble)},
    volume   = {63},
    date     = {2013},
    number   = {1},
    pages    = {89--111},
    issn     = {0373-0956},
    review   = {\MR{3089196}},
    doi      = {10.5802/aif.2756},
}

\bib{Iye07}{article}{
    author={Iyengar, Srikanth},
    title={Andr\'{e}-Quillen homology of commutative algebras},
    conference={
       title={Interactions between homotopy theory and algebra},
    },
    book={
       series={Contemp. Math.},
       volume={436},
       publisher={Amer. Math. Soc., Providence, RI},
    },
    isbn={978-0-8218-3814-3},
    date={2007},
    pages={203--234},
    review={\MR{2355775}},
    doi={10.1090/conm/436/08410},
}

\bib{KR25}{article}{
    author  = {Khan, Adeel A.},
    author  = {Rydh, David},
    title   = {Virtual Cartier divisors and blow-ups},
    journal = {Selecta Math. (N.S.)},
    volume  = {31},
    date    = {2025},
    number  = {4},
    pages   = {Paper No. 67, 28},
    issn    = {1022-1824},
    review  = {\MR{4932694}},
    doi     = {10.1007/s00029-025-01060-7},
}


\bib{LS67}{article}{
    author  = {Lichtenbaum, S.},
    author  = {Schlessinger, M.},
    title   = {The cotangent complex of a morphism},
    journal = {Trans. Amer. Math. Soc.},
    volume  = {128},
    date    = {1967},
    pages   = {41--70},
    issn    = {0002-9947},
    review  = {\MR{0209339}},
    doi     = {10.2307/1994516},
}

\bib{DAG}{thesis}{
    label     = {Lur-DAG},
    author    = {Lurie, Jacob},
    title     = {Derived algebraic geometry},
    type      = {Ph.D.~Thesis},
    school    = {Massachusetts Institute of Technology},
    note      = {Availabe at \url{https://www.math.ias.edu/~lurie/}},
    publisher = {ProQuest LLC, Ann Arbor, MI},
    date      = {2004},
    pages     = {(no paging)},
    review    = {\MR{2717174}},
    url       = {https://www.math.ias.edu/~lurie/},
}

\bib{Lur09}{book}{
    label     = {Lur-HTT},
    author    = {Lurie, Jacob},
    title     = {Higher topos theory},
    series    = {Annals of Mathematics Studies},
    volume    = {170},
    publisher = {Princeton University Press, Princeton, NJ},
    date      = {2009},
    pages     = {xviii+925},
    isbn      = {978-0-691-14049-0},
    isbn      = {0-691-14049-9},
    review    = {\MR{2522659}},
    doi       = {10.1515/9781400830558},
}



\bib{SAG}{article}{
    label  = {Lur-SAG},
    author = {Lurie, Jacob},
    title  = {Spectral algebraic geometry},
    note   = {Unfinished book,
             version Feb.\ 8, 2018,
             freely availabe at \url{https://www.math.ias.edu/~lurie/}},
    url    = {https://www.math.ias.edu/~lurie/},
}

\bib{Mat89}{book}{
    author    = {Matsumura, Hideyuki},
    title     = {Commutative ring theory},
    series    = {Cambridge Studies in Advanced Mathematics},
    volume    = {8},
    edition   = {2},
    publisher = {Cambridge University Press, Cambridge},
    date      = {1989},
    pages     = {xiv+320},
    isbn      = {0-521-36764-6},
    review    = {\MR{1011461}},
}

\bib{May05}{article}{
    author = {May, Jon Peter},
    title  = {The axioms for triangulated categories},
    year   = {2005},
    note   = {Note, availabe at the author's webpage.},
    url    = {https://www.math.uchicago.edu/~may/MISC/Triangulate.pdf},
}

\bib{MR18}{article}{
    author  = {Mourtada, Hussein},
    author  = {Reguera, Ana J.},
    title   = {Mather discrepancy as an embedding dimension in the space of arcs},
    journal = {Publ. Res. Inst. Math. Sci.},
    volume  = {54},
    date    = {2018},
    number  = {1},
    pages   = {105--139},
    issn    = {0034-5318},
    review  = {\MR{3749346}},
    doi     = {10.4171/PRIMS/54-1-4},
}

\bib{Mus01}{article}{
    author  = {Mustaţă, Mircea},
    title   = {Jet schemes of locally complete intersection canonical singularities},
    note    = {With an appendix by David Eisenbud and Edward Frenkel},
    journal = {Invent. Math.},
    volume  = {145},
    date    = {2001},
    number  = {3},
    pages   = {397--424},
    issn    = {0020-9910},
    review  = {\MR{1856396}},
    doi     = {10.1007/s002220100152},
}

\bib{Nas95}{article}{
    author  = {Nash, John F.},
    title   = {Arc structure of singularities},
    journal = {Duke. Math. J.},
    volume  = {81},
    date    = {1995},
    number  = {1},
    pages   = {31--38},
    issn    = {0012-7094},
    review  = {\MR{4057146}},
    doi     = {10.1215/00127094-2019-0043},
}

\bib{Reg06}{article}{
    author  = {Reguera, Ana J.},
    title   = {A curve selection lemma in spaces of arcs and the image of the
              nash map},
    journal = {Compos. Math.},
    volume  = {142},
    date    = {2006},
    number  = {1},
    pages   = {119--130},
    issn    = {0010-437X},
    review  = {\MR{2197405}},
    doi     = {10.1112/S0010437X05001582},
    note    = {Corrected in \cite{Reg21}},
}

\bib{Reg09}{article}{
    author  = {Reguera, Ana J.},
    title   = {Towards the singular locus of the space of arcs},
    journal = {Amer. J. Math.},
    volume  = {131},
    date    = {2009},
    number  = {2},
    pages   = {313--350},
    issn    = {0002-9327},
    review  = {\MR{2503985}},
    doi     = {10.1353/ajm.0.0046},
}

\bib{Reg21}{article}{
    author  = {Reguera, Ana J.},
    title   = {Corrigendum: A curve selection lemma in spaces of arcs and the
              image of the nash map},
    journal = {Compos. Math.},
    volume  = {157},
    date    = {2021},
    number  = {3},
    pages   = {641--648},
    review  = {\MR{4236197}},
    doi     = {10.1112/S0010437X20007733},
}

\bib{SP}{webpage}{
    label  = {SP},
    author = {{{Stac}ks Project Authors}, {The}},
    title  = {The Stacks Project},
    url    = {http://stacks.math.columbia.edu},
}

\bib{Toe14}{article}{
    author  = {To\"{e}n, Bertrand},
    title   = {Derived algebraic geometry},
    journal = {EMS Surv. Math. Sci.},
    volume  = {1},
    date    = {2014},
    number  = {2},
    pages   = {153--240},
    issn    = {2308-2151},
    review  = {\MR{3285853}},
    doi     = {10.4171/EMSS/4},
}

\bib{TV08}{article}{
    author  = {To\"en, Bertrand},
    author  = {Vezzosi, Gabriele},
    title   = {Homotopical algebraic geometry. II. Geometric stacks and
              applications},
    journal = {Mem. Amer. Math. Soc.},
    volume  = {193},
    date    = {2008},
    number  = {902},
    pages   = {x+224},
    issn    = {0065-9266},
    review  = {\MR{2394633}},
    doi     = {10.1090/memo/0902},
}

\bib{Voj07}{article}{
    author={Vojta, Paul},
    title={Jets via Hasse-Schmidt derivations},
    conference={
       title={Diophantine geometry},
    },
    book={
       series={CRM Series},
       volume={4},
       publisher={Ed. Norm., Pisa},
    },
    isbn={978-88-7642-206-5},
    isbn={88-7642-206-5},
    date={2007},
    pages={335--361},
    review={\MR{2349665}},
}

\bib{Wei94}{book}{
    author    = {Weibel, Charles A.},
    title     = {An introduction to homological algebra},
    series    = {Cambridge Studies in Advanced Mathematics},
    volume    = {38},
    publisher = {Cambridge University Press, Cambridge},
    date      = {1994},
    pages     = {xiv+450},
    isbn      = {0-521-43500-5},
    isbn      = {0-521-55987-1},
    review    = {\MR{1269324}},
    doi       = {10.1017/CBO9781139644136},
}

\end{biblist}
\end{bibdiv}
\end{document}